\documentclass[11pt,a4paper]{article} 

\usepackage[ruled,linesnumbered]{algorithm2e} 
\usepackage{amscd} 
\usepackage{amsfonts} 
\usepackage{amsmath} 
\usepackage{amssymb} 
\usepackage{amsthm} 
\usepackage{bm} 
\usepackage{booktabs} 
\usepackage{delarray} 
\usepackage{extarrows} 
\usepackage{geometry} 
\usepackage{graphicx} 
\usepackage{hyperref} 
\usepackage{longtable} 
\usepackage{tikz-cd} 
\usepackage[all]{xy} 
\usepackage[affil-it]{authblk} 

\usepackage{adjustbox}
\usepackage{booktabs}
\usepackage{caption}
\usepackage{color}
\usepackage{enumitem} 
\usepackage{float}
\usepackage{hyperref}
\usepackage{indentfirst} 
\usepackage{makecell} 
\usepackage{mathrsfs} 
\usepackage{mathtools} 
\usepackage{multirow} 
\usepackage{relsize} 
\usepackage{subfigure}

\setlist[enumerate]{itemsep=0pt,parsep=0pt}

\title{Diamond diagrams and multivariable ($\varphi,\OK\x$)-modules}
\author{Yitong Wang\thanks{E-mail address: \texttt{yitongw.wang@utoronto.ca}}}
\date{}

\allowdisplaybreaks[4] 

\def\AA{\mathbb{A}}

\def\FF{\mathbb{F}}

\def\VV{\mathbb{V}}

\def\ZZ{\mathbb{Z}}

\def\NNN{\mathbb{Z}_{\geq0}}

\def\cJ{\mathcal{J}}

\def\cM{\mathcal{M}}

\def\cO{\mathcal{O}}

\def\fM{\mathfrak{M}}

\def\fm{\mathfrak{m}}

\def\ft{\mathfrak{t}}
\def\fw{\mathfrak{w}}

\def\det{\operatorname{det}}

\def\Fil{\operatorname{Fil}}

\def\Gal{\operatorname{Gal}}

\def\GL{\operatorname{GL}}

\def\Hom{\operatorname{Hom}}
\def\id{\operatorname{id}}

\def\Ind{\operatorname{Ind}}

\def\JH{\operatorname{JH}}

\def\M{\operatorname{M}}
\def\Mat{\operatorname{Mat}}
\def\max{\operatorname{max}}

\def\min{\operatorname{min}}
\def\mod{\operatorname{mod}}

\def\nss{\operatorname{nss}}

\def\soc{\operatorname{soc}}

\def\ss{\operatorname{ss}}

\def\Sym{\operatorname{Sym}}

\def\unr{\operatorname{un}}

\def\cc{^{\circ}}

\def\eqdef{\overset{\mathrm{def}}{=}}
\def\Fp{\mathbb{F}_p}

\def\Fq{\mathbb{F}_q}
\def\iff{\Leftrightarrow}

\def\into{\hookrightarrow}
\def\inv{^{-1}}

\def\j{^{(j)}}

\def\jj{^{(f-1-j)}}

\def\OK{\mathcal{O}_K}
\def\onto{\twoheadrightarrow}

\def\Qp{\mathbb{Q}_p}
\def\Qpbar{\overline{\mathbb{Q}}_p}
\def\rbar{\overline{r}}
\def\rhobar{\overline{\rho}}
\def\sprod{\scalebox{1}{$\prod$}}
\def\ssum{\scalebox{1}{$\sum$}}

\def\x{^{\times}}

\newcommand{\abs}[1]{|#1|}

\newcommand{\babs}[1]{\left|#1\right|}

\newcommand{\bbbra}[1]{\left[#1\right]}
\newcommand{\bbra}[1]{\left(#1\right)}
\newcommand{\bigabs}[1]{\big|#1\big|}
\newcommand{\bigang}[1]{\big\langle#1\big\rangle}
\newcommand{\bigbbbra}[1]{\big[#1\big]}
\newcommand{\bigbra}[1]{\big(#1\big)}
\newcommand{\Bigbra}[1]{\Big(#1\Big)}
\newcommand{\bigset}[1]{\big\{#1\big\}}
\newcommand{\bra}[1]{(#1)}
\newcommand{\dbra}[1]{(\!(#1)\!)}
\newcommand{\ddbra}[1]{[\![#1]\!]}

\newcommand{\op}[1]{\operatorname{#1}}
\newcommand{\ovl}[1]{\overline{#1}}
\newcommand{\pmat}[1]{\begin{pmatrix}#1\end{pmatrix}}
\newcommand{\set}[1]{\{ #1 \}}
\newcommand{\smat}[1]{\left(\begin{smallmatrix}#1\end{smallmatrix}\right)}
\newcommand{\sset}[1]{\left\{ #1 \right\}}
\newcommand{\un}[1]{\underline{#1}}
\newcommand{\wh}[1]{\widehat{#1}}
\newcommand{\wt}[1]{\widetilde{#1}}

\begin{document}

\newtheorem{definition}{Definition}[section] 
\newtheorem{remark}[definition]{Remark}
\newtheorem{example}[definition]{Example}
\newtheorem{proposition}[definition]{Proposition}
\newtheorem{lemma}[definition]{Lemma}
\newtheorem{corollary}[definition]{Corollary}
\newtheorem{theorem}[definition]{Theorem}
\newtheorem{conjecture}[definition]{Conjecture}

\maketitle

\begin{abstract}
    Let $p$ be a prime number and $K$ a finite unramified extension of $\Qp$. Let $\pi$ be an admissible smooth mod $p$ representation of $\GL_2(K)$ occurring in some Hecke eigenspaces of the mod $p$ cohomology and $\rbar$ be its underlying global two-dimensional Galois representation. When $\rbar$ satisfies some Taylor--Wiles hypotheses and is sufficiently generic at $p$, we compute explicitly certain constants appearing in the diagram associated to $\pi$, generalizing the results of Dotto-Le in \cite{DL21}. As a result, we prove that the associated \'etale $(\varphi,\OK\x)$-module $D_A(\pi)$ defined by Breuil-Herzig-Hu-Morra-Schraen is explicitly determined by the restriction of $\rbar$ to the decomposition group at $p$, generalizing the results of Breuil-Herzig-Hu-Morra-Schraen in \cite{BHHMS3} and the author in \cite{Wang3}.
\end{abstract}

\tableofcontents

\section{Introduction}

Let $p$ be a prime number and $F$ be a totally real number field that is unramified at places above $p$. Let $D$ be a quaternion algebra with center $F$ that is split at all places above $p$ and at exactly one infinite place. For each compact open subgroup $U\subseteq(D\otimes_F\AA_F^{\infty})$ where $\AA_F^{\infty}$ is the set of finite ad\`eles of $F$, we denote by $X_U$ the associated smooth projective algebraic Shimura curve over $F$.

Let $\FF$ be a sufficiently large finite extension of $\Fp$. We fix an absolutely irreducible continuous representation $\rbar:\Gal(\ovl{F}/F)\to\GL_2(\FF)$. For $w$ a finite place of $F$, we write $\rbar_w\eqdef\rbar|_{\Gal(\ovl{F}_w/F_w)}$. We let $S_D$ be the set of finite places where $D$ ramifies, $S_{\rbar}$ be the set of finite places where $\rbar$ ramifies, and $S_p$ the set of places above $p$. We fix a place $v\in S_p$ and write $K\eqdef F_v$. We assume that $p\geq5$, that $\rbar|_{\Gal(\ovl{F}/F(\sqrt[p]{1}))}$ is absolutely irreducible, that the image of $\rbar(G_{F(\sqrt[5]{1})})$ in $\op{PGL}_2(\FF)$ is not isomorphic to $A_5$, that $\rbar_w$ is generic in the sense of \cite[Def.~11.7]{BP12} for $w\in S_p$ and that $\rbar_w$ is non-scalar for $w\in S_D$. Then there is a so-called ``local factor'' defined in \cite[\S3.3]{BD14} and \cite[\S6.5]{EGS15} as follows:
\begin{equation}\label{Constants Eq local factor}
    \pi\eqdef\Hom_{U^v}\bigg(\ovl{M}^v,\Hom_{\Gal(\ovl{F}/F)}\Big(\rbar,\varinjlim\limits_V H^1_{\text{\'et}}(X_V\times_F\ovl{F},\FF)\Big)\bigg)[\fm'],
\end{equation}
where the inductive limit runs over the compact open subgroups $V\subseteq(D\otimes_F\AA_F^{\infty})\x$, and we refer to \cite[\S3.3]{BD14} and \cite[\S6.5]{EGS15} for the definitions of the compact open subgroup $U^v\subseteq(D\otimes_F\AA_F^{\infty,v})\x$, the (finite-dimensional) irreducible smooth representation $\ovl{M}^v$ of $U^v$ over $\FF$, and the maximal ideal $\fm'$ in a certain Hecke algebra. We assume that $\rbar$ is modular in the sense that $\pi\neq0$.

Then the key question is to understand the $\GL_2(K)$-representation $\pi$ in \eqref{Constants Eq local factor}. It is hoped that the representation $\pi$ can be used to realize the mod $p$ Langlands correspondence for $\GL_2(K)$. In particular, we hope that $\pi$ only depends on $\rbar_v$ and would like to find a description of $\pi$ in terms of $\rbar_v$. There have been many results on the representation-theoretic properties of $\pi$ as above. For example, under some mild assumptions on $\rbar$ it is known that
\begin{enumerate}
    \item 
    $\pi^{K_1}\cong D_0(\rbar_v^{\vee})$ as $K\x\GL_2(\OK)$-representations (\cite{Le19}), where $K_1\eqdef1+p\M_2(\OK)$, $D_0(\rbar_v^{\vee})$ is the (finite-dimensional) representation of $\GL_2(\OK)$ defined in \cite[\S13]{BP12} and $K\x$ acts on $D_0(\rbar_v^{\vee})$ by the character $\det(\rbar_v^{\vee})\omega\inv$ with $\omega$ the mod $p$ cyclotomic character.
    \item 
    the Diamond diagram $(\pi^{I_1}\into\pi^{K_1})$ only depends on $\rbar_v$ (\cite{DL21}), where $I_1\eqdef\smat{1+p\OK&\OK\\p\OK&1+p\OK}$. Moreover, \cite{DL21} computed explicitly many constants appearing in the diagram when $\rbar_v$ is assumed to be semisimple.
\end{enumerate}
However, the complete understanding of $\pi$ still seems a long way off. 

In this article, we generalize the computation of \cite{DL21} and compute explicitly many constants appearing in the diagram $(\pi^{I_1}\into\pi^{K_1})$ when $\rbar_v$ is non-semisimple, which is much more complicated than in the semisimple case. As a result, (when $p$ is sufficiently large with respect to $f\eqdef[K:\Qp]$ with some more assumptions on $\rbar$) we prove a local-global compatibility result for $\pi$ as above which was conjectured by Breuil-Herzig-Hu-Morra-Schraen (\cite[Conj.~3.1.2]{BHHMS3}).

\hspace{\fill}

To state the main result, we refer to \cite{BHHMS2} for the definition of the ring $A$ and the notion of \'etale $(\varphi,\OK\x)$-modules over $A$ (see also \S\ref{Constants Sec main result}). In \cite{BHHMS2}, Breuil-Herzig-Hu-Morra-Schraen attached to $\pi$ as in \eqref{Constants Eq local factor} an \'etale $(\varphi,\OK\x)$-module $D_A(\pi)$ over $A$. In \cite{BHHMS3}, they also gave a conjectural description of $D_A(\pi)$ in terms of $\rbar_v$ by constructing a functor $D_A^{\otimes}$ from the category of finite-dimensional continuous representations of $\Gal(\ovl{K}/K)$ over $\FF$ to the category of \'etale $(\varphi,\OK\x)$-modules over $A$.

We assume moreover that
\begin{enumerate}
    \item
    the framed deformation ring $R_{\rbar_w}$ of $\rbar_w$ over the Witt vectors $W(\FF)$ is formally smooth for $w\in(S_D\cup S_{\rbar})\setminus S_p$;
    \item
    $\rbar_v$ is of the following form up to twist:
    \begin{equation}\label{Constants Eq rbarv}
        \rbar_{v}|_{I_{K}}\cong\pmat{\omega_f^{\sum\nolimits_{j=0}^{f-1}(r_j+1)p^j}&*\\0&1}~\text{with}~\max\set{12,2f\!+\!1}\leq r_j\leq p\!-\!\max\set{15,2f\!+\!3}~\forall\,j,
    \end{equation}
    where $I_{K}\subseteq\Gal(\ovl{K}/K)$ is the inertia subgroup and $\omega_f$ is the fundamental character of level $f$. 
\end{enumerate}
Our main result is the following, which verifies the conjecture \cite[Conj.~3.1.2]{BHHMS3} in our setting.

\begin{theorem}\label{Constants Thm main}
    Let $\pi$ be as in \eqref{Constants Eq local factor} and keep all the assumptions on $\rbar$. Then we have an isomorphism of \'etale $(\varphi,\OK\x)$-modules
    \begin{equation*}
        D_A(\pi)\cong D_A^{\otimes}(\rbar_v(1)).
    \end{equation*}
\end{theorem}


Theorem \ref{Constants Thm main} is proved by \cite[Thm.~3.1.3]{BHHMS3} when $\rbar_v$ is semisimple, and proved by \cite[Thm.1.1]{Wang3} when $\rbar_v$ is maximally non-split in the sense that $|W(\rbar_v)|=1$, where $W(\rbar_v)$ is the set of Serre weights of $\rbar_v$ defined in \cite[\S3]{BDJ10}. The proof of Theorem \ref{Constants Thm main} is based on the explicit computation of certain constants appearing in the Diamond diagram $(\pi^{I_1}\into\pi^{K_1})$ in the sense of \cite{DL21}, together with the results of \cite{Wang2} on $D_A(\pi)$ and the results of \cite{Wang3} on $D_A^{\otimes}(\rbar_v(1))$.

\hspace{\fill}

We describe certain constants in the Diamond diagram and the strategy of the proof of Theorem \ref{Constants Thm main} in more detail. 


We let $R:\pi^{I_1}\to\bigbra{\soc_{\GL_2(\OK)}\pi}^{I_1}$ be the map defined as in \cite[Def.~4.1]{DL21} and we write $I\eqdef\smat{\OK\x&\OK\\p\OK&\OK\x}$. Given an $I$-character $\chi$, we write $R\chi$ for the $I$-character such that $R\bigbra{\pi^{I_1}[\chi]}=\pi^{I_1}[R\chi]$. In particular, we have $R\chi=\chi$ if and only if $\chi$ appears in $\bigbra{\soc_{\GL_2(\OK)}\pi}^{I_1}$. Then we define the nonzero map $g_{\chi}:\pi^{I_1}[R\chi]\to \pi^{I_1}[R\chi^s]$ between $1$-dimensional $\FF$-vector spaces by the formula $g_{\chi}\bra{R(v)}=R\bigbra{\!\!\smat{0&1\\p&0}v}$ for $v\in \pi^{I_1}[\chi]$, where $\chi^s$ is the conjugation of $\chi$ by the matrix $\smat{0&1\\p&0}$. All the constants in the Diamond diagram are then defined in terms of suitable compositions of the maps $g_{\chi}$. We write $\delta(\chi)\eqdef R\chi^s$.

When $\rbar_v$ is semisimple, for any $I$-character $\chi$ appearing in $\pi^{I_1}$ such that $R\chi=\chi$ there exists an integer $d\geq1$ such that $\delta^d(\chi)=\chi$. Then the composition 
\begin{equation*}
    \pi^{I_1}[\chi]\xrightarrow{g_{\chi}} \pi^{I_1}[\delta(\chi)]\xrightarrow{g_{\delta(\chi)}}\cdots\rightarrow \pi^{I_1}[\delta^d(\chi)]=\pi^{I_1}[\chi]
\end{equation*}
is given by a scalar $g(\chi)\in\FF\x$, which is an example of the constants in the Diamond diagram. The constants $g(\chi)$ for $R\chi=\chi$ are computed explicitly by \cite{DL21}, which is enough to determine the structure of $D_A(\pi)$ and to conclude Theorem \ref{Constants Thm main} in the semisimple case.

When $\rbar_v$ is non-semisimple, the situation is completely different. For any $I$-character $\chi$ appearing in $\pi^{I_1}$, it converges to a distinguished $I$-character $\chi_{0}$ in the sense that there exists $\ell(\chi)\geq1$ such that $\delta^{\ell}(\chi)=\chi_0$ for all $\ell\geq\ell(\chi)$. To obtain a constant in the Diamond diagram, we consider the following two maps:
\begin{equation*}
\begin{aligned}
    &\pi^{I_1}[R\chi]\xrightarrow{g_{\chi}} \pi^{I_1}[\delta(\chi)]\xrightarrow{g_{\delta(\chi)}}\cdots\rightarrow \pi^{I_1}[\delta^{\ell(\chi)}(\chi)]=\pi^{I_1}[\chi_0];\\
    &\pi^{I_1}[R\chi]\xrightarrow{g_{R\chi}} \pi^{I_1}[\delta(R\chi)]\xrightarrow{g_{\delta(R\chi)}}\cdots\rightarrow \pi^{I_1}[\delta^{\ell(R\chi)}(R\chi)]=\pi^{I_1}[\chi_0].
\end{aligned}
\end{equation*}    
Then the composition 
\begin{equation*}
    \bbra{\sprod^{0}_{i=\ell(R\chi)-1}g_{\delta^i(R\chi)}}\inv\circ\bbra{\sprod^{0}_{i=\ell(\chi)-1}g_{\delta^i(\chi)}}:\pi^{I_1}[R\chi]\to \pi^{I_1}[R\chi]
\end{equation*}
is given by a scalar $g(\chi)\in\FF\x$, which is an example of the constants in the Diamond diagram.

The key step in the non-semisimple case is to find a suitable collection of $I$-characters $\chi$ such that the constants $g(\chi)$ are enough to determine the structure of $D_A(\pi)$. This is based on a detailed study of the relation between $I_1$-invariants of $\pi$.

Then we compute explicitly these constants $g(\chi)$ following the strategy of \cite{DL21} and we refer to \cite[\S1]{DL21} for a more detailed introduction. The computation is much more delicate than in the semisimple case and takes up a substantial portion of the article. Finally, we are able to conclude Theorem \ref{Constants Thm main} using these computations together with the results of \cite{Wang2} on $D_A(\pi)$ and the results of \cite{Wang3} on $D_A^{\otimes}(\rbar_v(1))$.

The proof of Theorem \ref{Constants Thm main} is very computational. There may exist a more conceptual proof one day, which will hopefully avoid the genericity assumptions on $\rbar_v$ and the technical computations, but such proof is not known so far.

\subsection*{Organization of the article}

In \S\ref{Constants Sec sw}, we define all the basic objects that are needed throughout this article. In \S\ref{Constants Sec mu}, we study the relation between $I_1$-invariants of $\pi$ that are needed to form the necessary constants in the Diamond diagram. In \S\ref{Constants Sec Diamond diagram}, we review the strategy of \cite{DL21} and specialize to the non-semisimple case. In particular, the computation of the constants in the Diamond diagram can be divided into two parts: one comes from certain elements in tamely potentially Barsotti--Tate deformation rings and is the content of \S\ref{Constants Sec Kisin}, the other comes from the action of certain elements of the group algebra of $\GL_2(\OK)$ on tame types and is the content of \S\ref{Constants Sec c'J}. Finally, in \S\ref{Constants Sec main result} we combine the results of the previous sections and the results of \cite{Wang2} and \cite{Wang3} to finish the proof of Theorem \ref{Constants Thm main}.

\subsection*{Acknowledgements}

We thank Christophe Breuil for his guidance in this area. We thank Florian Herzig for his interest in this work, for regular communication, and for his comments.

This work was supported by the University of Toronto.

\subsection*{Notation}

Let $p$ be a prime number. We fix an algebraic closure $\Qpbar$ of $\Qp$. Let $K\subseteq\Qpbar$ be the unramified extension of $\Qp$ of degree $f\geq1$ with ring of integers $\OK$ and residue field $\Fq$ (hence $q=p^f$). We denote by $G_K\eqdef\Gal(\Qpbar/K)$ the absolute Galois group of $K$ and $I_K\subseteq G_K$ the inertia subgroup. Let $\FF$ be a large enough finite extension of $\FF_p$. Fix an embedding $\sigma_0:\Fq\into\FF$ and let $\sigma_j\eqdef\sigma_0\circ\varphi^j$ for $j\in\ZZ$, where $\varphi:x\mapsto x^p$ is the arithmetic Frobenius on $\Fq$. We identify $\cJ\eqdef\Hom(\Fq,\FF)$ with $\set{0,1,\ldots,f-1}$, which is also identified with $\ZZ/f\ZZ$ so that the addition and subtraction in $\cJ$ are modulo $f$. For $a\in\OK$, we denote by $\ovl{a}\in\Fq$ its reduction modulo $p$. For $a\in\Fq$, we also view it as an element of $\FF$ via $\sigma_0$.

For $F$ a perfect ring of characteristic $p$, we denote by $W(F)$ the ring of Witt vectors of $F$. For $x\in F$, we denote by $[x]\in W(F)$ its Techm\"uller lift.

Let $E$ be a finite extension of $\Qp$ with ring of integers $\cO$ and residue field $\FF$. We view elements of $\OK$ and elements of $W(\FF)$ as elements of $\cO$ via the embedding $\OK=W(\Fq)\stackrel{\sigma_0}{\into} W(\FF)\into\cO$.

Let $I\eqdef\smat{\OK\x&\OK\\p\OK&\OK\x}$ be the Iwahori subgroup of $\GL_2(\OK)$, $I_1\eqdef\smat{1+p\OK&\OK\\p\OK&1+p\OK}$ be the pro-$p$ Iwahori subgroup, $K_1\eqdef1+p\M_2(\OK)$ be the first congruence subgroup, $N_0\eqdef\smat{1&\OK\\0&1}$ and $H\eqdef\smat{[\Fq\x]&0\\0&[\Fq\x]}$.

For $P$ a statement, we let $\delta_P\eqdef1$ if $P$ is true and $\delta_P\eqdef0$ otherwise.

Throughout this article, we let $\pi$ be as in \eqref{Constants Eq local factor} and $\rhobar\eqdef\rbar_v^{\vee}$. Twisting $\rhobar$ and $\pi$ using \cite[Lemma~2.9.7]{BHHMS3} and \cite[Lemma~3.1.1]{BHHMS3}, we assume moreover that
\begin{equation}\label{Constants Eq rhobar}
    \rhobar\cong\pmat{\omega_f^{\sum\nolimits_{j=0}^{f-1}(r_j+1)p^j}\unr(\xi)&*\\0&\unr(\xi)\inv},
\end{equation}
where $r_j$ is as in \eqref{Constants Eq rbarv}, $\xi\in\FF\x$, $\unr(\xi):G_K\to\FF\x$ is the unramified character sending geometric Frobenius elements to $\xi$, and $\omega_f:G_K\to\FF\x$ is such that $\omega_f(g)$ is the reduction modulo $p$ of $g(\sqrt[q-1]{-p})/\sqrt[q-1]{-p}\in\cO\x$ for all $g\in G_K$ and for any choice of a $(q-1)$-th root $\sqrt[q-1]{-p}$ of $-p$. In particular, $p$ acts trivially on $\pi$.

\section{Preliminaries}\label{Constants Sec sw}

We write $\un{i}$ for an element $(i_0,\ldots,i_{f-1})\in\ZZ^f$. For $a\in\ZZ$, we write $\un{a}\eqdef(a,\ldots,a)\in\ZZ^f$. For $J\subseteq\cJ$, we define $\un{e}^J\in\ZZ^f$ by $e^J_j\eqdef\delta_{j\in J}$. We say that $\un{i}\leq\un{i}'$ if $i_j\leq i'_j$ for all $j$. We define the left shift $\delta:\ZZ^f\to\ZZ^f$ by $\delta(\un{i})_j\eqdef i_{j+1}$. Abusing notation, we still denote by $\un{i}$ the integer $\ssum_{j=0}^{f-1}i_jp^j\in\ZZ$.

A \textbf{Serre weight} is an isomorphism class of an absolutely irreducible representation of $\GL_2(\Fq)$ over $\FF$. For $\lambda=(\un{\lambda}_1,\un{\lambda}_2)\in\ZZ^{2f}$ such that $\un{0}\leq\un{\lambda}_1-\un{\lambda}_2\leq\un{p}-\un{1}$, we define 
\begin{equation*}
    F(\lambda)\eqdef\scalebox{1}{$\bigotimes$}_{j=0}^{f-1}\bbra{\!\bbra{\Sym^{\lambda_{1,j}-\lambda_{2,j}}\Fq^2\otimes_{\Fq}\det^{\lambda_{2,j}}}\otimes_{\Fq,\sigma_j}\FF}.
\end{equation*}
We also denote it by $(\un{\lambda}_1-\un{\lambda}_2)\otimes\det^{\un{\lambda}_2}$. 

For $\lambda=(\un{\lambda}_1,\un{\lambda}_2)\in\ZZ^{2f}$, we define the character $\chi_{\lambda}:I\to\FF\x$ by $\smat{a&b\\pc&d}\mapsto(\ovl{a})^{\un{\lambda}_1}(\ovl{d})^{\un{\lambda}_2}$, where $a,d\in\OK\x$ and $b,c\in\OK$. In particular, if $\un{0}\leq\un{\lambda}_1-\un{\lambda}_2\leq\un{p}-\un{1}$, then $\chi_{\lambda}$ is the $I$-character acting on $F(\lambda)^{I_1}$. We still denote by $\chi_{\lambda}$ its restriction to $H$.

We write $\un{r}=(r_0,\ldots,r_{f-1})$ with $r_j$ as in \eqref{Constants Eq rbarv}. For $\un{b}\in\ZZ^f$ such that $-\un{r}\leq\un{b}\leq\un{p}-\un{2}-\un{r}$, we denote by $\sigma_{\un{b}}$ the Serre weight $F\bigbra{\ft_{\bra{\un{r},\un{0}}}(\un{b})}$ (see \cite[\S2.4]{BHHMS1} for $\ft_{\bra{\un{r},\un{0}}}(\un{b})\in\ZZ^{2f}$). 

For $\rhobar$ as in \eqref{Constants Eq rhobar}, we let $W(\rhobar)$ be the set of Serre weights of $\rhobar$ defined in \cite[\S3]{BDJ10} and $J_{\rhobar}\subseteq\cJ$ be the subset as in \cite[(17)]{Bre14}. Then by \cite[Prop.~A.3]{Bre14} and \cite[(14)]{BHHMS1}, the subset $J_{\rhobar}\subseteq\cJ$ is characterized by
\begin{equation*}
    W(\rhobar)=\sset{\sigma_{\un{b}}:~\begin{array}{ll}b_j=0&\text{if}~j\notin J_{\rhobar}\\b_j\in\set{0,1}&\text{if}~j\in J_{\rhobar}\end{array}}.
\end{equation*}
In particular, $\rhobar$ is semisimple if and only if $J_{\rhobar}=\cJ$. 

For each $J\subseteq\cJ$, as in \cite[\S2]{Wang2} we define the (distinct) Serre weight $\sigma_J\eqdef F\bra{\lambda_J}$ with $\lambda_J\eqdef\bra{\un{s}^J+\un{t}^J,\un{t}^J}$, where
\begin{align}
    \label{Constants Eq sJ}
    s^J_j&\eqdef
    \begin{cases}
        r_j&\text{if}~j\notin J,~j+1\notin J\\
        r_j+1&\text{if}~j\in J,~j+1\notin J\\
        p-2-r_j&\text{if}~j\notin J,~j+1\in J\\
        p-1-r_j&\text{if}~j\in J,~j+1\in J,~j\notin J_{\rhobar}\\
        p-3-r_j&\text{if}~j\in J,~j+1\in J,~j\in J_{\rhobar};
    \end{cases}\\
    \label{Constants Eq tJ}
    t^J_j&\eqdef
    \begin{cases}
        0&\text{if}~j\notin J,~j+1\notin J\\
        -1&\text{if}~j\in J,~j+1\notin J\\
        r_j+1&\text{if}~j\notin J,~j+1\in J~\text{or}~j\in J,~j+1\in J,~j\in J_{\rhobar}\\
        r_j&\text{if}~j\in J,~j+1\in J,~j\notin J_{\rhobar}.
    \end{cases}
\end{align}
In particular, for $J\subseteq J_{\rhobar}$ we have $\sigma_J=\sigma_{\un{e}^J}$. We also define the Serre weight $\sigma_{J}^s\eqdef F\bra{\lambda_{J^s}}$ with $\lambda_{J^s}\eqdef\bra{\un{s}^{J^s}+\un{t}^{J^s},\un{t}^{J^s}}$, where $s^{J^s}_j\eqdef p-1-s^J_j$ and $t^{J^s}_j\eqdef r_j-t^J_j$. We write $\chi_J\eqdef\chi_{\lambda_J}$ and $\chi_{J^s}\eqdef\chi_{\lambda_{J^s}}$ so that $\chi_J$ (resp.\,$\chi_{J^s}$) is the $I$-character acting on $\sigma_J^{I_1}$ (resp.\,$\bra{\sigma_J^s}^{I_1}$). For each $I$-character $\chi$, we denote by $\chi^s$ its conjugation by the matrix $\smat{0&1\\p&0}$. Then we have $\chi_J^s=\chi_{J^s}$ and $\chi_{J^s}^s=\chi_J$ for all $J\subseteq\cJ$.

For $J\subseteq\cJ$ and $k\in\ZZ$, we write $J+k\eqdef\set{j+k:j\in J}\subseteq\cJ$. Then we define 
\begin{equation}\label{Constants Eq various J}
\begin{array}{lll}
    J^{\ss}\eqdef J\cap J_{\rhobar};&J^{\nss}\eqdef J\setminus J_{\rhobar};&J^c\eqdef\cJ\setminus J;\\
    \partial J\eqdef J\setminus(J-1);&\delta_{\ss}(J)\eqdef(J-1)^{\ss};&J^{\delta}\eqdef J\Delta(J-1)^{\ss}.
\end{array}
\end{equation}
Here we recall that $J\Delta J'\eqdef (J\setminus J')\sqcup(J'\setminus J)$.

\begin{lemma}\label{Constants Lem J*}
    We have $\sigma_{J^s}=\sigma_{(J-1)^{\ss}}$ if and only if $J\Delta(J-1)^{\ss}=\cJ$, which happens for a unique $J^*\subseteq\cJ$ if $J_{\rhobar}\neq\cJ$. 
\end{lemma}

\begin{proof}
    The if part follows from a case-by-case examination. Conversely, by \eqref{Constants Eq tJ} we have
    \begin{equation}\label{Constants Eq Lem J*}
        \sset{j:t^{J}_j\in\set{r_j,r_j+1}}+1=J.
    \end{equation}
    Hence $\sigma_{J^s}=\sigma_{(J-1)^{\ss}}$ implies $\un{t}^{J^s}=\un{t}^{(J-1)^{\ss}}$, which implies $J^c=(J-1)^{\ss}$ by \eqref{Constants Eq Lem J*}. If moreover $J_{\rhobar}\neq\cJ$, then $J^*\subseteq\cJ$ is uniquely characterized by the property that $j\in J^*$ if and only if $j\notin J_{\rhobar}$ or ($j\in J_{\rhobar}$ and $j+1\notin J$).
\end{proof}

\begin{lemma}\label{Constants Lem compare sJ}
    For $J\subseteq\cJ$, we have
    \begin{enumerate}
        \item 
        $s^J_j=
    \begin{cases}
        s^{(J-1)^{\ss}}_j+\delta_{j\in J\Delta(J-1)^{\ss}}&\text{if}~j+1\notin J\Delta(J-1)^{\ss}\\
        p-2-s^{(J-1)^{\ss}}_j+\delta_{j\in J\Delta(J-1)^{\ss}}&\text{if}~j+1\in J\Delta(J-1)^{\ss};
    \end{cases}$
    \item 
    $s^J_j=
    \begin{cases}
        s^{J^{\ss}}_j+\delta_{j\in J^{\nss}}&\text{if}~j+1\notin J^{\nss}\\
        p-2-s^{J^{\ss}}_j+\delta_{j\in J^{\nss}}&\text{if}~j+1\in J^{\nss}.
    \end{cases}$
    \end{enumerate}
\end{lemma}

\begin{proof}
    This follows from a case-by-case examination and is left as an exercise. We refer to \cite[Lemma~D.1]{Wang2} for the second part of (i).
\end{proof}

\begin{lemma}\label{Constants Lem t+t+s}
    For $J,J'\subseteq\cJ$, we have $(-1)^{\un{t}^J+\un{t}^{J^s}+\un{s}^{J'}}=1$.
\end{lemma}

\begin{proof}
    By definition we have $\un{t}^J+\un{t}^{J^s}=\un{r}$. By \eqref{Constants Eq sJ} we also have $s^{J'}_j\not\equiv r_j\,\mod\,2$ if and only if $j\in\partial(J')$ or $j\in\partial\bra{(J')^c}$, which proves the result since $\babs{\partial(J')}=\babs{\partial\bra{(J')^c}}$.
\end{proof}

\section{The relation between \texorpdfstring{$I_1$}{.}-invariants}\label{Constants Sec mu}

Let $\pi$ be as in \eqref{Constants Eq local factor}. We study the relation between $I_1$-invariants of $\pi$, which generalizes some results of \cite[\S5]{Wang2}. The main results are Proposition \ref{Constants Prop mu for Js} and Proposition \ref{Constants Prop ratio Js}.

From now on, we identify $\pi^{K_1}$ with $D_0(\rhobar)$ (see Introduction), which is the (finite-dimensional) representation of $\GL_2(\OK)$ defined in \cite[\S13]{BP12}. For each $J\subseteq\cJ$, the character $\chi_J$ appears in $\pi^{I_1}=D_0(\rhobar)^{I_1}$ with multiplicity one by \cite[Lemma 4.1(ii)]{Wang2}. We fix a choice of $0\neq v_J\in\pi^{I_1}$ with $I$-character $\chi_J$. For each $j\in\cJ$ we define
\begin{equation*}
    Y_j\eqdef\sum\limits_{a\in\Fq\x}a^{-p^j}\pmat{1&[a]\\0&1}\in\FF\ddbra{N_0}.
\end{equation*}
For $\un{i}=(i_0,\ldots,i_{f-1})\in\ZZ^f$, we write $\un{Y}^{\un{i}}$ for $\sprod_{j=0}^{f-1}Y_j^{i_j}$.

For $J,J'\subseteq\cJ$ such that $(J-1)^{\ss}=(J')^{\ss}$, we define $\mu_{J,J'}\in\FF\x$ as in \cite[(47)]{Wang2}. In particular, in the case $(J')^{\nss}\neq\cJ$ such that
\begin{equation}\label{Constants Eq condition simple mu}
    (J')^{\nss}\subseteq(J-1)^{\nss}\Delta(J'-1)^{\nss},
\end{equation}
the element $\mu_{J,J'}\in\FF\x$ is defined by the formula
\begin{equation}\label{Constants Eq muJJ'}
    \bbbra{\sprod_{j+1\in J\Delta J'}Y_j^{s^{J'}_j}\sprod_{j+1\notin J\Delta J'}Y_j^{p-1}}\smat{p&0\\0&1}\bbra{\un{Y}^{-\un{e}^{(J\cap J')^{\nss}}}v_J}=\mu_{J,J'}v_{J'},
\end{equation}
where $\un{Y}^{-\un{e}^{(J\cap J')^{\nss}}}v_J\in D_0(\rhobar)$ is a suitable shift of $v_J$ defined in \cite[Prop.~4.2]{Wang2}. Then for $J_1,J_2,J_3,J_4\subseteq\cJ$ such that $(J_1-1)^{\ss}=(J_2-1)^{\ss}=J_3^{\ss}=J_4^{\ss}$ we have $\mu_{J_1,J_3}/\mu_{J_1,J_4}=\mu_{J_2,J_3}/\mu_{J_2,J_4}$. In particular, for $J,J'$ such that $J^{\ss}=(J')^{\ss}$, the quantity $\mu_{J'',J}/\mu_{J'',J'}$ does not depend on $J''$ such that $(J''-1)^{\ss}=J^{\ss}$, and we denote it by $\mu_{*,J}/\mu_{*,J'}$.

\hspace{\fill}

For each $J\subseteq\cJ$, the character $\chi_J^s$ also appears in $\pi^{I_1}=D_0(\rhobar)^{I_1}$ with multiplicity one by \cite[Lemma 4.1(ii)]{Wang2}. We fix a choice of $0\neq v_{J^s}\in\pi^{I_1}$ with $I$-character $\chi_J^s$. Since $p$ acts trivially on $\pi$, by rescaling the vectors $v_J$ and $v_{J^s}$ we assume from now on that $v_{J^s}=\smat{0&1\\p&0}v_J$ for all $J\subseteq\cJ$. We prove some analogous results for the vectors $v_{J^s}$.

\begin{proposition}\label{Constants Prop mu for Js}
\begin{enumerate}
    \item 
    For $J\subseteq\cJ$, there exists a unique element $\mu_{J^s,J^{\ss}}\in\FF\x$ such that 
    \begin{equation}\label{Constants Eq muJsJss}
        \bbbra{\sprod_{j+1\notin J^{\nss}}Y_j^{s^{J^{\ss}}_j}\sprod_{j+1\in J^{\nss}}Y_j^{p-1}}\smat{p&0\\0&1}v_{J^s}=\mu_{J^s,J^{\ss}}v_{J^{\ss}}.
    \end{equation}
    \item 
    For $J\subseteq\cJ$ such that $J^{\nss}\neq\cJ$, there exists a unique element $\mu_{J^s,J}\in\FF\x$ such that
    \begin{equation*}
        \un{Y}^{\un{s}^J}\smat{p&0\\0&1}v_{J^s}=\mu_{J^s,J}v_{J}.
    \end{equation*}
\end{enumerate}
\end{proposition}

\begin{proof}
    (i). Since $v_J\in\pi^{I_1}$ has $I$-character $\chi_J$, by Frobenius reciprocity there is a $\GL_2(\OK)$-equivariant map
    \begin{equation}\label{Constants Eq Frob rec}
    \begin{aligned}
        \Ind_{I}^{\GL_2(\OK)}(\chi^s_{J^s})=\Ind_{I}^{\GL_2(\OK)}(\chi_{J})&\stackrel{\alpha}{\onto}\bigang{\!\GL_2(\OK)v_J}=\bigang{\!\GL_2(\OK)\smat{p&0\\0&1}v_{J^s}}\into\pi\\
        \phi&\mapsto v_J=\smat{0&1\\p&0}v_{J^s},
    \end{aligned}
    \end{equation}
    where $\phi\in\Ind_{I}^{\GL_2(\OK)}(\chi_{J})$ is supported on $I$ such that $\phi(\id)=1$. By \cite[Prop.~4.2]{Wang2}, the $\GL_2(\OK)$-subrepresentation $V\eqdef\bigang{\!\GL_2(\OK)v_J}\subseteq D_0(\rhobar)$ has constituents $\sigma_{\un{b}}$ with
    \begin{equation*}
    \begin{cases}
        b_j=\delta_{j\in J}(=\delta_{j\in J^{\ss}})&\text{if}~j\notin J^{\nss}\\
        b_j\in\set{0,(-1)^{\delta_{j+1\in J}}}&\text{if}~j\in J^{\nss}.
    \end{cases}
    \end{equation*}
    In particular, we have $V=I\bigbra{\sigma_{J^{\ss}},\sigma_{\un{c}}}$ in the notation of \cite[Lemma~5.1]{Wang2}(iii) with
    \begin{equation*}
        c_j\eqdef\delta_{j\in J^{\ss}}+\delta_{j\in J^{\nss}}(-1)^{\delta_{j+1\in J}}.
    \end{equation*} 
    Since $c_j=\delta_{j\in J^{\ss}}$ if and only if $j\notin J^{\nss}$, we deduce from \cite[Lemma~3.2(i)]{Wang2} (applied to $\lambda=\lambda_{J^s}$) that $V$ is isomorphic to the quotient $Q\bigbra{\chi^s_{J^s},(J^{\nss})^c-1}$ of $\Ind_{I}^{\GL_2(\OK)}(\chi^s_{J^s})$ in the notation of \cite[Lemma~3.2(iii)]{Wang2}. By Lemma \ref{Constants Lem compare sJ}(ii), for $j+1\notin J^{\nss}$ we have
    \begin{equation*}
        p-2-s^{J^s}_j+\delta_{j\in(J^{\nss})^c}=s^J_j-\delta_{j\in J^{\nss}}=s^{J^{\ss}}_j.
    \end{equation*}
    Then by \cite[Lemma~3.2(iii)]{Wang2} (applied to $\lambda=\lambda_{J^s}$), the LHS of \eqref{Constants Eq muJsJss} is nonzero in $\sigma_{J^{\ss}}$ and is the unique (up to scalar) $H$-eigencharacter in $\sigma_{J^{\ss}}$ killed by all $Y_j$. It follows that the LHS of \eqref{Constants Eq muJsJss} is a nonzero $I_1$-invariant of $\sigma_{J^{\ss}}$, hence is a scalar multiple of $v_{J^{\ss}}$.

    (ii). By the proof of (i) and using $J^{\nss}\neq\cJ$, the $\GL_2(\OK)$-equivariant surjection $\alpha$ in \eqref{Constants Eq Frob rec} is not an isomorphism, hence it maps $\sigma_{J}^s=\soc\bigbra{\Ind_{I}^{\GL_2(\OK)}(\chi_{J})}$ to zero. By \cite[Lemma~3.2(iii)(a)]{Wang2} (applied to $\lambda=\lambda_{J^s}$) we have
    \begin{equation*}
        \un{Y}^{\un{p}-\un{1}-\un{s}^{J^s}}\smat{p&0\\0&1}v_{J^s}+(-1)^{f-1+\un{s}^{J^s}+\un{t}^{J^s}}\bbra{\sprod_{j=0}^{f-1}\bra{s^{J^s}_j}!}v_{J^s}=0.
    \end{equation*}
    Since $\un{s}^{J^s}=\un{p}-\un{1}-\un{s}^J$, this proves (ii) with
    \begin{equation}\label{Constants Eq muJsJ}
        \mu_{J^s,J}=(-1)^{\un{1}+\un{s}^{J^s}+\un{t}^{J^s}}\bbra{\sprod_{j=0}^{f-1}\bra{p-1-s^{J}_j}!}\inv=(-1)^{\un{t}^{J^s}}\bbra{\sprod_{j=0}^{f-1}\bra{s^{J}_j}!},
    \end{equation}
    where the second equality uses
    \begin{equation}\label{Constants Eq factorial}
        \bigbra{\bra{p-1-r}!}\inv\equiv(-1)^{r+1}r!~\mod p~\forall\,0\leq r\leq p-1,
    \end{equation}
    and the uniqueness is clear.
\end{proof}

\begin{corollary}\label{Constants Cor muJ*}
    Suppose that $J_{\rhobar}\neq\cJ$ and $J_{\rhobar}\neq\emptyset$. Let $J^*$ be as in Lemma \ref{Constants Lem J*}. Then we have
    \begin{equation*}
        (-1)^{\abs{J^*\cap(J^*-1)^{\nss}}}\mu_{(J^*-1)^{\ss},J^*}\mu_{J^*,(J^*-1)^{\ss}}=1.
    \end{equation*}
\end{corollary}

\begin{proof}
    Note that $J_{\rhobar}\neq\emptyset$ implies $(J^*)^{\nss}\neq\cJ$. Since $\sigma_{(J^*-1)^{\ss}}=\sigma_{J^*}^s$ by Lemma \ref{Constants Lem J*}, we have $\un{s}^{(J-1)^{\ss}}=\un{s}^{J^s}$ and $\un{t}^{\bra{(J-1)^{\ss}}^s}=\un{t}^J$. Then by \eqref{Constants Eq muJsJ} we have
    \begin{equation*}
    \begin{aligned}
        &\mu_{(J^*-1)^{\ss},J^*}=\mu_{(J^*)^s,J^*}=(-1)^{\un{t}^{J^s}}\bbra{\sprod_{j=0}^{f-1}\bra{s^{J}_j}!};\\
        &\mu_{J^*,(J^*-1)^{\ss}}=\mu_{\bra{(J^*-1)^{\ss}}^s,(J^*-1)^{\ss}}=(-1)^{\un{t}^{J}}\bbra{\sprod_{j=0}^{f-1}\bra{s^{J^s}_j}!}.
    \end{aligned}
    \end{equation*}
    Hence we have
    \begin{equation*}
        \mu_{(J^*-1)^{\ss},J^*}\mu_{J^*,(J^*-1)^{\ss}}=(-1)^{\un{t}^{J^s}+\un{t}^J}\bbra{\sprod_{j=0}^{f-1}\bra{s^{J}_j}!\bra{s^{J^s}_j}!}=(-1)^{\un{t}^{J^s}+\un{t}^J+\un{s}^J+\un{1}}=(-1)^f,
    \end{equation*}
    where the second equality uses \eqref{Constants Eq factorial} and the last equality follows from Lemma \ref{Constants Lem t+t+s}. Moreover, from the structure of $J^*$ one can check that $|J^*\cap(J^*-1)^{\nss}|\equiv f~\mod2$, which completes the proof.
\end{proof}

\begin{lemma}\label{Constants Lem J0 Delta J}
    Let $J\subseteq\cJ$. We write $J_0\eqdef\bigbra{J^{\ss}\sqcup(J^c-1)^{\nss}}+1$ and $c'_j\eqdef2\delta_{j\in J_0\cap J^{\nss}}+p-1-s^{J_0}_j+\delta_{j\in J_0\Delta J}$ for all $j\in\cJ$. Then we have for all $j\in\cJ$
    \begin{enumerate}
        \item 
        $(-1)^{\delta_{j+1\notin J_0}}\bigbra{2\delta_{j\in J_0\cap J^{\nss}}+\delta_{j\in J^{\ss}}-\delta_{j\in J_0\Delta J^{\ss}}+\delta_{j\in J_0\Delta J}}=\delta_{j\in J^{\ss}}+\delta_{j\in J^{\nss}}(-1)^{\delta_{j+1\in J}}$;
        \item 
        $s^J_j+\delta_{j+1\notin J_0\Delta J}c'_j=\delta_{j+1\in J_0\Delta J}s^J_j+\delta_{j+1\notin J_0\Delta J}(p-1)$;
        \item 
        $\delta_{j+1\notin J^{\nss}}s^{J^{\ss}}_j+\delta_{j+1\in J^{\nss}}(p-1)+\delta_{j+1\notin J_0\Delta J}c'_j$\setlength{\parindent}
        {3cm}\par
        $=\delta_{j+1\in J_0\cap J^{\nss}}p+\delta_{j+1\in J_0\Delta J^{\ss}}s^{J^{\ss}}_j+\delta_{j+1\notin J_0\Delta J^{\ss}}(p-1)$.
    \end{enumerate}
\end{lemma}

\begin{proof}
    (i). We have
    \begin{equation*}
    \begin{aligned}
        &(-1)^{\delta_{j+1\notin J_0}}\bigbra{2\delta_{j\in J_0\cap J^{\nss}}+\delta_{j\in J^{\ss}}-\delta_{j\in J_0\Delta J^{\ss}}+\delta_{j\in J_0\Delta J}}\\
        &\hspace{1cm}=(-1)^{\delta_{j+1\notin J_0}}\big(2\delta_{j\in J_0}\delta_{j\in J^{\nss}}+\delta_{j\in J^{\ss}}\\
        &\hspace{3cm}-\bra{\delta_{j\in J_0}+\delta_{j\in J^{\ss}}-2\delta_{j\in J_0}\delta_{j\in J^{\ss}}}+\bra{\delta_{j\in J_0}+\delta_{j\in J}-2\delta_{j\in J_0}\delta_{j\in J}}\big)\\
        &\hspace{1cm}=(-1)^{\delta_{j\in(J^c)^{\ss}}+\delta_{j\in(J-1)^{\nss}}}\delta_{j\in J}=\delta_{j\in J^{\ss}}+\delta_{j\in J^{\nss}}(-1)^{\delta_{j+1\in J}},
    \end{aligned}
    \end{equation*}
    where the last equality is easy to check, separating the cases $j\in J^{\ss}$, $j\in J^{\nss}$, and $j\notin J$.
    
    (ii). It suffices to show that $s^{J_0}_j=s^J_j+2\delta_{j\in J_0\cap J^{\nss}}+\delta_{j\in J_0\Delta J}$
    for $j+1\notin J_0\Delta J$. This follows from a case-by-case examination similar to the proof of \cite[Lemma~D.1]{Wang2}.

    (iii). By the definition of $J_0$, we have that $j+1\notin J_0\Delta J$ implies $j\notin J^{\nss}$. Hence by (ii) and Lemma \ref{Constants Lem compare sJ}(ii) we have 
    \begin{equation*}
        \delta_{j+1\notin J_0\Delta J}c'_j=\delta_{j+1\notin J_0\Delta J}\bigbra{p-1-s^J_j}=
    \begin{cases}
        s^{J^{\ss}}_j+1&\text{if}~j+1\in J_0\cap J^{\nss}\\
        p-1-s^{J^{\ss}}_j&\text{if}~j+1\notin J_0\Delta J,~j+1\notin J_0\cap J^{\nss}.
    \end{cases}
    \end{equation*} 
    Consider the decomposition $\cJ=J_1\sqcup J_2\sqcup J_3\sqcup J_4\sqcup J_5$ with $J_1\eqdef \bra{J_0\setminus J}\sqcup\bra{J^{\ss}\setminus J_0}$, $J_2\eqdef J_0\cap J^{\ss}$, $J_3\eqdef J_0\cap J^{\nss}$, $J_4\eqdef J^{\nss}\setminus J_0$ and $J_5\eqdef (J_0)^c\cap J^c$. Then we can rewrite both sides of (iii) as
    \begin{equation*}
    \begin{aligned}
        \text{LHS}&=\bigbra{\delta_{j+1\in J_1}+\delta_{j+1\in J_2}+\delta_{j+1\in J_5}}s^{J^{\ss}}_j+\bigbra{\delta_{j+1\in J_3}+\delta_{j+1\in J_4}}(p-1)\\
        &\hspace{2cm}+\delta_{j+1\in J_3}\bigbra{s^{J^{\ss}}_j+1}+\bigbra{\delta_{j+1\in J_2}+\delta_{j+1\in J_5}}\bigbra{p-1-s^{J^{\ss}}_j};\\
        \text{RHS}&=\delta_{j+1\in J_3}p+\bigbra{\delta_{j+1\in J_1}+\delta_{j+1\in J_3}}s^{J^{\ss}}_j+\bigbra{\delta_{j+1\in J_2}+\delta_{j+1\in J_4}+\delta_{j+1\in J_5}}(p-1),
    \end{aligned}
    \end{equation*}
    from which it is easy to see that the equality holds.
\end{proof}

\begin{proposition}\label{Constants Prop ratio Js}
    For $J\subseteq\cJ$ such that $J^{\nss}\neq\cJ$, we have
    \begin{equation*}
        \frac{\mu_{J^s,J}}{\mu_{J^s,J^{\ss}}}=\frac{\mu_{*,J}}{\mu_{*,J^{\ss}}}.
    \end{equation*}
\end{proposition}

\begin{proof}
    Let $J_0\eqdef\bigbra{J^{\ss}\sqcup(J^c-1)^{\nss}}+1$. Then both the pairs $(J_0,J)$ and $(J_0,J^{\ss})$ satisfy \eqref{Constants Eq condition simple mu}. We consider the elements $B_1\eqdef\smat{p&0\\0&1}v_{J^s}\in\pi$ and
    \begin{equation*}
        B_2\eqdef\bbbra{\sprod_{j+1\notin J_0\Delta J}Y_j^{2\delta_{j\in J_0\cap J^{\nss}}+p-1-s^{J_0}_j+\delta_{j\in J_0\Delta J}}}\smat{p&0\\0&1}\bigbra{\un{Y}^{-\un{i}}v_{J_0}}\in\pi.
    \end{equation*}
    By Proposition \ref{Constants Prop mu for Js}(ii), we have $\un{Y}^{\un{s}^J}B_1=\mu_{J^s,J}v_J$. By Lemma \ref{Constants Lem J0 Delta J}(ii) and \eqref{Constants Eq muJJ'} applied to $(J_0,J)$, we also have
    \begin{equation*}
        \un{Y}^{\un{s}^J}B_2=\bbbra{\sprod_{j+1\in J_0\Delta J}Y_j^{s^{J'}_j}\sprod_{j+1\notin J_0\Delta J}Y_j^{p-1}}\smat{p&0\\0&1}\bbra{\un{Y}^{-\un{e}^{J_0\cap J^{\nss}}}v_J}=\mu_{J_0,J}v_J.
    \end{equation*}
    In particular, we deduce from \cite[Lemma~3.1(ii)]{Wang2} that $B_1$ and $B_2$ are $H$-eigenvectors with the same $H$-eigencharacter.
    
    Moreover, by the proof and the notation of Proposition \ref{Constants Prop mu for Js}(i), we have $B_1\in I\bigbra{\sigma_{J^{\ss}},\sigma_{\un{c}}}\cong Q\bigbra{\chi^s_{J^s},(J^{\nss})^c-1}$. By \cite[Prop.~5.7]{Wang2} applied to $\bra{J,J',\un{i}}=\bigbra{J_0,\bra{J_0\Delta J}-1,\un{e}^{J_0\cap J^{\nss}}}$ and using Lemma \ref{Constants Lem J0 Delta J}(i), we have $Y_{j'}B_2\in I\bigbra{\sigma_{J^{\ss}},\sigma_{\un{c}}}$ for all $j'\in\cJ$. Since $I\bigbra{\sigma_{J^{\ss}},\sigma_{\un{c}}}$ is multiplicity free as an $H$-representation by \cite[Prop.~Lemma~3.2(ii),(iii)]{Wang2} and using $J^{\nss}\neq\cJ$, we deduce that $Y_{j'}B_1=\bra{\mu_{J^s,J}/\mu_{J_0,J}}Y_{j'}B_2$ for all $j'\in\cJ$. Hence we have
    \begin{equation*}
    \begin{aligned}
        \mu_{J^s,J^{\ss}}v_{J^{\ss}}&=\bbbra{\sprod_{j+1\notin J^{\nss}}Y_j^{s^{J^{\ss}}_j}\sprod_{j+1\in J^{\nss}}Y_j^{p-1}}B_1\\
        &\hspace{-0.5cm}=\frac{\mu_{J^s,J}}{\mu_{J_0,J}}\bbbra{\sprod_{j+1\notin J^{\nss}}Y_j^{s^{J^{\ss}}_j}\sprod_{j+1\in J^{\nss}}Y_j^{p-1}}B_2\\
        &\hspace{-0.5cm}=\frac{\mu_{J^s,J}}{\mu_{J_0,J}}\bbbra{\sprod_{j+1\notin J_0\Delta J^{\ss}}Y_j^{s^{J^{\ss}}_j}\sprod_{j+1\in J_0\Delta J^{\ss}}Y_j^{p-1}}\un{Y}^{p\delta\bigbra{\un{e}^{J_0\cap J^{\nss}}}}\smat{p&0\\0&1}\bigbra{\un{Y}^{-\un{e}^{J_0\cap J^{\nss}}}v_{J_0}}\\
        &\hspace{-0.5cm}=\frac{\mu_{J^s,J}}{\mu_{J_0,J}}\bbbra{\sprod_{j+1\notin J_0\Delta J^{\ss}}Y_j^{s^{J^{\ss}}_j}\sprod_{j+1\in J_0\Delta J^{\ss}}Y_j^{p-1}}\smat{p&0\\0&1}v_{J_0}=\frac{\mu_{J^s,J}}{\mu_{J_0,J}}\mu_{J_0,J^{\ss}}v_{J^{\ss}},
    \end{aligned}
    \end{equation*}
    where the first equality follows from Proposition \ref{Constants Prop mu for Js}(i), the third equality follows from Lemma \ref{Constants Lem J0 Delta J}(iii), the fourth equality follows from \cite[Lemma~3.1(i)]{Wang2}, and the last equality follows from \eqref{Constants Eq muJJ'} applied to $(J_0,J^{\ss})$. Therefore, we have $\mu_{J^s,J^{\ss}}=\bra{\mu_{J^s,J}/\mu_{J_0,J}}\mu_{J_0,J^{\ss}}$, which completes the proof.
\end{proof}

\section{Kisin modules}\label{Constants Sec Kisin}

Let $\rhobar$ be as in \eqref{Constants Eq rhobar}. In particular, $\rhobar$ is generic in the sense of \cite[\S3.2.2]{DL21}. We study the Kisin modules following \cite{DL21} that can be used to describe the tamely potentially Barsotti--Tate deformation rings of $\rhobar$. When $\rhobar$ is non-semisimple, we define and compute the elements in the Kisin modules that form one part of the computation of the constants in the diagram $(\pi^{I_1}\into\pi^{K_1})$. The main result is Proposition \ref{Constants Prop UpJ}.

Up to enlarging $\FF$, we fix an $f$-th root $\beta\eqdef\sqrt[f]{\xi}\in\FF\x$ of $\xi$ (see \eqref{Constants Eq rhobar}). We let $T_K$ be the Lubin--Tate variable as in \cite[\S2]{Wang3}. By \cite[(44)]{Wang3}, the Lubin--Tate $(\varphi,\OK\x)$-module $D_K(\rhobar)$ associated to $\rhobar$ has the following form ($a\in\OK\x$):

\begin{equation*}
\left\{\begin{array}{cll}
    D_K(\rhobar)&=&\prod\nolimits_{j=0}^{f-1}D_{K,\sigma_j}(\rhobar)=\prod\nolimits_{j=0}^{f-1}\bbra{\FF\dbra{T_K}e_0^{(j)}\oplus\FF\dbra{T_K}e_1^{(j)}}\\
    \varphi\bra{e_0^{(j+1)}\ e_1^{(j+1)}}&=&(e_0^{(j)}\ e_1^{(j)})\Mat(\varphi^{(j)})\\
    a(e_0^{(j)}\ e_1^{(j)})&=&(e_0^{(j)}\ e_1^{(j)})\Mat(a^{(j)}),
\end{array}\right.
\end{equation*}
where 
\begin{equation}\label{Constants Eq LT dj}
    \Mat(\varphi^{(j)})=\pmat{\beta\,T_K^{-(q-1)(r_j+1)}&\beta\inv d_j\\0&\beta\inv}
\end{equation}
for some $d_j\in\FF$, and $\Mat(a^{(j)})\in I_2+\M_2\bigbra{T_K^{q-1}\FF\ddbra{T_K^{q-1}}}$ which uniquely determines $\Mat(a^{(j)})$. 

By \cite[Lemma~5.1]{Wang3}, the Fontaine--Laffaille module $\op{FL}(\rhobar)$ associated to $\rhobar$ (see \cite{FL82}) has the following form:
\begin{equation}\label{Constants Eq FL}
\left\{\begin{array}{cll}
    \op{FL}(\rhobar)&=&\prod\nolimits_{j=0}^{f-1}\op{FL}_{\sigma_j}(\rhobar)=\prod\nolimits_{j=0}^{f-1}\bbra{\FF e_0^{(j)}\oplus\FF e_1^{(j)}}\\
    \Fil^{k}\op{FL}_{\sigma_j}(\rhobar)&=&\FF e_0^{(j)}~\text{exactly for}~1\leq k\leq r_j+1\\
    \varphi_{r_{j+1}+1}(e_0^{(j+1)})&=&\beta\inv(e_0^{(j)}-d_{j+1}e_1^{(j)})\\
    \varphi(e_1^{(j+1)})&=&\beta\,e_1^{(j)},
\end{array}\right.
\end{equation}
where $d_j\in\FF$ is as in \eqref{Constants Eq LT dj}. In particular, by \cite[(18)]{Bre14} with $e^j=e_1^{(f-j)}$, $f^j=e_0^{(f-j)}$, $\alpha_j=\beta$, $\beta_j=\beta\inv$ and $\mu_j=d_{f+1-j}$ for all $j\in\cJ$ in \cite[(16)]{Bre14}, we deduce that $d_j=0$ if and only if $j\in J_{\rhobar}$.

We fix a compatible system $(p_n)_n$ of $p$-power roots of $(-p)$ in $\Qpbar$ and define $K_{\infty}\eqdef\bigcup_{n\geq0}K(p_n)$. Let $\ovl{\cM}$ be the \'etale $\varphi$-module over $k\dbra{v}\otimes_{\Fp}\FF$ in the sense of \cite[\S3.2.1]{DL21} such that $\VV^*(\ovl{\cM})\cong\rhobar|_{G_{K_{\infty}}}$, where $\VV^*$ is Fontaine's anti-equivalence of categories (see \cite{Fon90}). Then as in the proof of \cite[Prop.~3.3]{DL21} and using \eqref{Constants Eq FL}, we can take

\begin{equation*}
\left\{\begin{array}{cll}
    \ovl{\cM}&=&\prod_{j=0}^{f-1}\ovl{\cM}^{(j)}=\prod_{j=0}^{f-1}\bbra{\FF\dbra{v}e_0\j\oplus\FF\dbra{v}e_1\j}\\
    \varphi\bra{e_0^{(j)}\ e_1^{(j)}}&=&(e_0^{(j+1)}\ e_1^{(j+1)})\Mat(\varphi_{\ovl{\cM}}^{(j)})
\end{array}\right.
\end{equation*}
with
\begin{equation}\label{Constants Eq etale phi}
    \Mat(\varphi_{\ovl{\cM}}\jj)=\pmat{\beta\inv&0\\0&\beta}\pmat{v^{r_j+1}&0\\-\beta^{-2}d_jv^{r_j+1}&1},
\end{equation}
where $d_j\in\FF$ is as in \eqref{Constants Eq LT dj}.

\hspace{\fill}

We write $\un{W}\eqdef(S_2)^f$ and $X^*(\un{T})\eqdef(\ZZ^2)^f\cong\ZZ^{2f}$. For $w\in\un{W}$ and $\lambda\in X^*(\un{T})$, we write $w_j\in S_2$ and $\lambda_j\in\ZZ^2$ the corresponding $j$-th components and define $w\lambda\in X^*(\un{T})$ with $(w\lambda)_j=w_j(\lambda_j)$. We write $\tau(w,\lambda):I_K\to\GL_2(\cO)\subseteq\GL_2(E)$ the associated tame inertial type defined in \cite[\S2.3.2]{DL21}, which is a $2$-dimensional representation of $I_K$ over $E$ that factors through the tame inertial quotient and extends to $G_K$. If moreover $(w,\mu)$ is a good pair (see \cite[\S2.3.2]{DL21}), which will always be the case in this article, we write $R_w(\lambda)$ the associated tame type of $K$ defined in \cite[\S2.3.1]{DL21}, which is a smooth irreducible representation of $\GL_2(\Fq)$ over $E$. 

For $\chi:I\to\FF\x$ a character, we write $\theta^{\circ}(\chi)\eqdef\Ind_{I}^{\GL_2(\OK)}\bra{[\chi]}$, which is an $\cO$-lattice in the principal series type $\theta(\chi)\eqdef\theta^{\circ}(\chi)[1/p]$, where $[\chi]:I\to\cO\x$ is the Techm\"uller lift of $\chi$. We let $\varphi^{\chi}\in\theta^{\circ}(\chi)$ be the unique element supported on $I$ such that $\varphi^{\chi}(\id)=1$. If moreover $\chi\neq\chi^s$, we write $\sigma(\chi)$ for the unique Serre weight such that $\chi$ is the $I$-character acting on $\sigma(\chi)^{I_1}$. Then the cosocle of $\theta\cc(\chi)$ is isomorphic to $\sigma(\chi)$, and we denote the image of $\varphi^{\chi}$ in $\sigma(\chi)$ by $\varphi^{\chi}$ as well.

For $R$ an $\cO$-algebra and $\tau$ a $2$-dimensional tame inertial type, we define a \textbf{Kisin module} over $R$ of type $\tau$ and its eigenbasis as in \cite[\S3.1]{DL21}, with the caveat that we only consider modules of rank $2$. For each Kisin module $\fM$ with a fixed eigenbasis, we define the matrices $A^{(j)}\in\M_2(R\ddbra{v})$ for $0\leq j\leq f-1$ as in \cite[\S3.4]{DL21}.

From now on, we write $\mu\eqdef\bra{\un{r}+\un{1},\un{0}}\in X^*\bra{\un{T}}$ and $\eta\eqdef\bra{\un{1},\un{0}}\in X^*\bra{\un{T}}$. We let $w,w'\in\un{W}$ such that (see \cite[Prop.~3.11]{DL21})
\begin{enumerate}
    \item 
    $(w_j,w_j')\neq(\fw,\id)$ for all $j\in\cJ$;
    \item 
    if $(w_j,w_j')=(\id,\fw)$, then $j\in J_{\rhobar}$ (or equivalently, $d_j=0$),
\end{enumerate}
where $\fw$ is the unique non-trivial element in $S_2$. As in \cite[\S3.5]{DL21}, we consider the tame inertial type $\tau=\tau(w,\mu-w'\eta)$. We let $\ovl{\fM}^{\tau}$ be the Kisin module over $\FF$ of type $\tau$ given by the matrices $\ovl{A}^{(f-1-j)}=\Mat(\phi_{\ovl{\cM}}\jj)v^{-\bra{\mu_j-w'_j\eta_j}}\dot{w}_j\in\M_2\bra{\FF\ddbra{v}}$ for $j\in\cJ$, where $v^{(a,b)}\in\M_2\!\bigbra{\FF\dbra{v}}$ denotes the diagonal matrix $\smat{v^a&0\\0&v^b}$ for $(a,b)\in\ZZ^2$, and $\dot{w}_j\in\M_2(\FF)$ denotes the corresponding permutation matrix associated to $w_j$. 
Then by \cite[Prop.~3.2.1]{LLHLM20}, the \'etale $\varphi$-module over $k\dbra{v}\otimes_{\Fp}\FF$ associated to $\ovl{\fM}^{\tau}$ in the sense of \cite[\S3.2]{LLHLM20} is isomorphic to $\ovl{\cM}$. Concretely, from \eqref{Constants Eq etale phi} we have (compare with \cite[(14)]{DL21})
\begin{equation*}
    \ovl{A}^{(f-1-j)}=
\begin{cases}
    \pmat{\beta\inv&0\\0&\beta}\pmat{v&0\\-\beta^{-2}d_jv&1}&\text{if}~(w_j,w'_j)=(\id,\id)\\
    \pmat{\beta\inv&0\\0&\beta}\pmat{1&0\\0&v}&\text{if}~(w_j,w'_j)=(\id,\fw)\\
    \pmat{\beta\inv&0\\0&\beta}\pmat{0&1\\v&-\beta^{-2}d_j}&\text{if}~(w_j,w'_j)=(\fw,\fw).
\end{cases}
\end{equation*}

Let $R_0\eqdef\cO\ddbra{X_j,Y_j,Z_j,Z'_j}_{j=0}^{f-1}/\bra{f_j}_{j=0}^{f-1}$ with $f_j=Y_j$ if $j\notin J_{\rhobar}$ and $f_j=X_jY_j-p$ if $j\in J_{\rhobar}$. For $\tau=\tau(w,\mu-w'\eta)$ as above, we let $\fM^{\tau}$ be the Kisin module over $R_0$ given by the matrices $A^{(f-1-j)}\in\M_2\bra{R_0\ddbra{v}}$ for $j\in\cJ$ such that (compare with \cite[(15)]{DL21})
\begin{equation}\label{Constants Eq Af-1-j}
    A^{(f-1-j)}=D^{(f-1-j)}A^{\prime(f-1-j)}
\end{equation}
with
\begin{equation}\label{Constants Eq DA'f-1-j}
\begin{aligned}
    D^{(f-1-j)}&\eqdef\pmat{Z'_j+[\beta]\inv&0\\0&Z_j+[\beta]};\\
    A^{\prime(f-1-j)}&\eqdef
\begin{cases}
    \pmat{v+p&0\\\bigbra{X_j-[\beta^{-2}d_j]}v&1}&\text{if}~(w_j,w'_j)=(\id,\id)\\
    \pmat{1&-Y_j\\0&v+p}&\text{if}~(w_j,w'_j)=(\id,\fw)\\
    \pmat{-Y_j&1\\v&X_j}&\text{if}~(w_j,w'_j)=(\fw,\fw)~\text{and}~j\in J_{\rhobar}\\
    \pmat{-p\bigbra{X_j-[\beta^{-2}d_j]}\inv&1\\v&X_j-[\beta^{-2}d_j]}&\text{if}~(w_j,w'_j)=(\fw,\fw)~\text{and}~j\notin J_{\rhobar}.\\
\end{cases}
\end{aligned}
\end{equation}
In particular, $\fM^{\tau}$ is a deformation of $\ovl{\fM}^{\tau}$ to $R_0$. 

\hspace{\fill}

For each $J\subseteq\cJ$, we let $s^*,w'\in\un{W}$ be characterized by $\chi_J=\chi_{(s^*)\inv\bra{\mu-w'\eta}}$ (see \cite[Lemma~3.13]{DL21}). We also let $w\in\un{W}$ such that $w_j=s^*_js^*_{j-1}$. In particular, we have an isomorphism of tame types $\theta(\chi_J)\cong R_w\bra{\mu-w'\eta}$. Then we define $U_p(\chi_J)\eqdef\prod_{j=0}^{f-1}U_p(\chi_J)_j\in R_0$ with (compare with \cite[Prop.~3.22]{DL21})
\begin{equation}\label{Constants Eq UpchiJ}
    U_p(\chi_J)_j\eqdef\bbra{p\bbra{A^{(f-1-j)}_{s^*_{j-1}(1)s^*_{j-1}(1)}}\inv\mod v}\in R_0,
\end{equation}
where $A^{(f-1-j)}_{s^*_{j-1}(1)s^*_{j-1}(1)}$ is the $\bra{s^*_{j-1}(1),s^*_{j-1}(1)}$-th entry of the matrix $A^{(f-1-j)}$ in \eqref{Constants Eq Af-1-j} for the Kisin module $\fM^{\tau}$ over $R_0$ of type $\tau=\tau(w,\mu-w'\eta)$.

\begin{lemma}\label{Constants Lem s*ww'}
    Let $J\subseteq\cJ$ and $s^*,w,w'\in\un{W}$ be as above. Then for each $j\in\cJ$ we have
    \begin{equation}\label{Constants Eq Lem s*ww' statement}
    \begin{aligned}
        s^*_{j-1}=\id&\iff j\notin J\\
        w_j=\id&\iff j\notin J,~j+1\notin J~\text{or}~j\in J,~j+1\in J\\
        w'_j=\id&\iff j\notin J,~j+1\notin J~\text{or}~j\in J,~j+1\in J,~j\notin J_{\rhobar}.
    \end{aligned}
    \end{equation}
\end{lemma}

\begin{proof}
    It suffices to show that $\chi_J=\chi_{(s^*)\inv\bra{\mu-w'\eta}}$ for $s^*,w'\in\un{W}$ as in \eqref{Constants Eq Lem s*ww' statement}. Then the statement for $w_j$ follows from $w_j=s^*_js^*_{j-1}$. 
    
    Recall from \S\ref{Constants Sec sw} that $\chi_J=\chi_{\lambda_J}$ with $\lambda_J=(\un{s}^J+\un{t}^J,\un{t}^J)\in X^*(\un{T})$. Concretely, by \eqref{Constants Eq sJ} and \eqref{Constants Eq tJ} we have
    \begin{equation}\label{Constants Eq Lem s*ww' 1}
        \bra{\lambda_J}_j=
    \begin{cases}
        \bra{r_j,0}&\text{if}~j\notin J,~j+1\notin J\\
        \bra{r_j,-1}&\text{if}~j\in J,~j+1\notin J\\
        \bra{p-1,r_j+1}&\text{if}~j\notin J,~j+1\in J\\
        \bra{p-2,r_j+1}&\text{if}~j\in J,~j+1\in J,~j\in J_{\rhobar}\\
        \bra{p-1,r_j}&\text{if}~j\in J,~j+1\in J,~j\notin J_{\rhobar}.
    \end{cases}
    \end{equation}
    By \eqref{Constants Eq Lem s*ww' statement} we also have
    \begin{equation}\label{Constants Eq Lem s*ww' 2}
    \begin{aligned}
        &\bigbra{(s^*)\inv\bra{\mu-w'\eta}}_j=\\
    &\hspace{1cm}\left\{\begin{array}{lll}
        \bra{r_j,0} & \text{if}~s^*_j=\id,~w'_j=\id, & \text{equivalently},~j\notin J,~j+1\notin J\\
        \bra{r_j+1,-1} & \text{if}~s^*_j=\id,~w'_j=\fw, & \text{equivalently},~j\in J,~j+1\notin J\\
        \bra{-1,r_j+1} & \text{if}~s^*_j=\fw,~w'_j=\fw, & \text{equivalently},~
    \begin{aligned}
        &j\notin J,~j+1\notin J\\
        &\text{or}~j\in J,~j+1\in J,~j\in J_{\rhobar}
    \end{aligned}\\
        \bra{0,r_j} & \text{if}~s^*_j=\fw,~w'_j=\id, & \text{equivalently},~j\in J,~j+1\in J,~j\notin J_{\rhobar}.
    \end{array}\right.
    \end{aligned}
    \end{equation}
    Combining \eqref{Constants Eq Lem s*ww' 1} and \eqref{Constants Eq Lem s*ww' 2} we have 
    \begin{equation*}
        \lambda_J-(s^*)\inv\bra{\mu-w'\eta}=\bra{p\un{e}^{J-1}-\un{e}^{J},\un{0}}\quad\text{in}~X^*(\un{T}).
    \end{equation*}
    Since $x^{p\un{e}^{J-1}-\un{e}^{J}}=1$ for all $x\in\Fq\x$, we deduce the equality $\chi_J=\chi_{(s^*)\inv\bra{\mu-w'\eta}}$.
\end{proof}

\begin{lemma}\label{Constants Lem Up'}
    Let $J\subseteq\cJ$. Then $U_p(\chi_J)$ is a product of a $1$-unit of $R_0$, an integer power of $p$, the scalar $[\beta]^{|J^c|-|J|}$, and the quantity $U_p(\chi_J)'\eqdef\sprod_{i=0}^{f-1}U_p(\chi_J)'_j\in R_0$ with
    \begin{equation}\label{Constants Eq Lem Up' statement}
        U_p(\chi_J)'_j=
    \begin{cases}
        1&\text{if}~j\notin J,~j+1\notin J~\text{or}~j\in J,~j+1\in J\\
        Y_j&\text{if}~j\in J,~j+1\notin J,~j\in J_{\rhobar}\\
        -X_j&\text{if}~j\notin J,~j+1\in J,~j\in J_{\rhobar}\\
        -[\beta^{-2}d_j]\inv&\text{if}~j\in J,~j+1\notin J,~j\notin J_{\rhobar}\\
        [\beta^{-2}d_j]&\text{if}~j\notin J,~j+1\in J,~j\notin J_{\rhobar}.
    \end{cases}
    \end{equation}
    In particular, we have $U_p(\chi_J)\in R_0[1/p]\x$. Here, a $1$-unit means an element of $1+\fm_0$, where $\fm_0$ is the maximal ideal of $R_0$.
\end{lemma}

\begin{proof}
    Let $s^*,w,w'\in\un{W}$ be as in Lemma \ref{Constants Lem s*ww'}. By \eqref{Constants Eq DA'f-1-j} we have 
    \begin{equation*}
        D^{(f-1-j)}_{s^*_{j-1}(1)s^*_{j-1}(1)}=
    \begin{cases}
        Z_j'+[\beta]\inv\in[\beta]\inv(1+\fm_0)&\text{if}~s^*_{j-1}=\id,~\text{equivalently},~j\notin J\\
        Z_j+[\beta]\in[\beta](1+\fm_0)&\text{if}~s^*_{j-1}=\fw,~\text{equivalently},~j\in J.
    \end{cases}
    \end{equation*}
    Hence we have
    \begin{equation}\label{Constants Eq Lem Up' 1}
        \bbra{\sprod_{j=0}^{f-1}D^{(f-1-j)}_{s^*_{j-1}(1)s^*_{j-1}(1)}}\inv\in[\beta]^{|J^c|-|J|}(1+\fm_0).
    \end{equation}
    By \eqref{Constants Eq DA'f-1-j} we also have
    \begin{equation}\label{Constants Eq Lem Up' 2}
        \bbra{A^{\prime(f-1-j)}_{s^*_{j-1}(1)s^*_{j-1}(1)}\mod v}=
    \begin{cases}
        1~\text{or}~p&\text{if}~w_j=\id\\
        X_j&\text{if}~\bra{w_j,w'_j,s^*_{j-1}}\!=\!(\fw,\fw,\fw)~\text{and}~j\in J_{\rhobar}\\
        -Y_j&\text{if}~\bra{w_j,w'_j,s^*_{j-1}}\!=\!(\fw,\fw,\id)~\text{and}~j\in J_{\rhobar}\\
        X_j-[\beta^{-2}d_j]&\text{if}~\bra{w_j,w'_j,s^*_{j-1}}\!=\!(\fw,\fw,\fw)~\text{and}~j\notin J_{\rhobar}\\
        -p\bigbra{\!X_j\!-\![\beta^{-2}d_j]}\inv&\text{if}~\bra{w_j,w'_j,s^*_{j-1}}\!=\!(\fw,\fw,\id)~\text{and}~j\notin J_{\rhobar}.
    \end{cases}
    \end{equation}
    Since $X_j-[\beta^{-2}d_j]\in-[\beta^{-2}d_j](1+\fm_0)$ for all $j\in\cJ$ and  $X_jY_j=p$ in $R_0$ for $j\in J_{\rhobar}$, from  \eqref{Constants Eq Lem Up' 2} and \eqref{Constants Eq Lem s*ww' statement} we deduce that 
    \begin{equation}\label{Constants Eq Lem Up' 3}
        \bbra{p\bbra{A^{\prime(f-1-j)}_{s^*_{j-1}(1)s^*_{j-1}(1)}}\inv\mod v}\in p^{\NNN}U_p(\chi_J)'_j(1+\fm_0)\quad\text{in}~R_0
    \end{equation}
    for $U_p(\chi_J)'_j$ as in \eqref{Constants Eq Lem Up' statement}. The result is then a combination of \eqref{Constants Eq Lem Up' 1} and \eqref{Constants Eq Lem Up' 3}.
\end{proof}

In the rest of this section we suppose that $J_{\rhobar}\neq\cJ$. Then for each $J\subseteq\cJ$, there exists $i\geq0$ such that $\delta_{\ss}^i(J)=\emptyset$ (see \eqref{Constants Eq various J} for $\delta_{\ss}$) and we define $\ell(J)\eqdef\min\set{i\geq0:\delta_{\ss}^i(J)=\emptyset}$. We then define
\begin{equation*}
    \wt{U}_p(J)\eqdef\frac{\sprod_{i=0}^{\ell(J)-1}U_p\bbra{\chi_{\delta_{\ss}^i(J)}}}{\sprod_{i=0}^{\ell(J^{\ss})-1}U_p\bbra{\chi_{\delta_{\ss}^i(J^{\ss})}}}\in R_0[1/p]
\end{equation*}
Since $J_{\rhobar}\neq\cJ$, there is a unique decomposition of $J_{\rhobar}$ into a disjoint union of intervals (in $\ZZ/f\ZZ$) not adjacent to each other $J_{\rhobar}=J_1\sqcup\ldots\sqcup J_t$. For each $1\leq i\leq t$, we write $J_i=\set{j_i,j_i+1,\ldots,j_i+k_i}$ with $j_i\in\cJ$ and $k_i\geq0$. Then we define
\begin{equation}\label{Constants Eq AssJ}
    A^{\ss}(J)\eqdef\ssum_{i=1}^{t}\bigbra{\delta_{j_i+k_i\in\partial(J^c)}(k_i+1)}\in\ZZ.
\end{equation}

\begin{proposition}\label{Constants Prop UpJ}
    Suppose that $J_{\rhobar}\neq\cJ$. Let $J\subseteq\cJ$. Then $\wt{U}_p(J)$ is a product of a $1$-unit of $R_0$, an integer power of $p$, and the scalar $[U_p(J)]$ with 
    \begin{equation}\label{Constants Eq UpJ}
        U_p(J)=(-1)^{A(J)}\beta^{B(J)}d(J)\in\FF\x,
    \end{equation}
    where 
    \begin{equation}\label{Constants Eq ABd}
    \begin{aligned}
        A(J)&\eqdef A^{\ss}(J)+\ssum_{j\notin J_{\rhobar}}\delta_{j\in\partial J}\in\ZZ;\\
        B(J)&\eqdef\ssum_{i=0}^{\ell(J)-1}\bigbra{\abs{\delta_{\ss}^i(J)^c}-\abs{\delta_{\ss}^i(J)}}-\ssum_{i=0}^{\ell(J^{\ss})-1}\bigbra{\abs{\delta_{\ss}^i(J^{\ss})^c}-\abs{\delta_{\ss}^i(J^{\ss})}}\in\ZZ;\\
        d(J)&\eqdef\bbra{\sprod_{j\in J^{\nss}}d_j}\inv\!\bbbra{\frac{\sprod_{i=0}^{\ell(J)-1}\sprod_{j\in\bra{\delta_{\ss}^i(J)-1}^{\nss}}d_j}{\sprod_{i=0}^{\ell(J^{\ss})-1}\sprod_{j\in\bra{\delta_{\ss}^i(J^{\ss})-1}^{\nss}}d_j}}\in\FF\x.
    \end{aligned}
    \end{equation}
\end{proposition}

\begin{proof}
    We write $\wt{U}_p(J)'\eqdef\prod_{j=0}^{f-1}\wt{U}_p(J)'_j\in R_0[1/p]$ with
    \begin{equation*}
        \wt{U}_p(J)'_j\eqdef\frac{\sprod_{i=0}^{\ell(J)-1}U_p\bbra{\chi_{\delta_{\ss}^i(J)}}'_j}{\sprod_{i=0}^{\ell(J^{\ss})-1}U_p\bbra{\chi_{\delta_{\ss}^i(J^{\ss})}}'_j}=\frac{\sprod_{i\geq0}U_p\bbra{\chi_{\delta_{\ss}^i(J)}}'_j}{\sprod_{i\geq0}U_p\bbra{\chi_{\delta_{\ss}^i(J^{\ss})}}'_j},
    \end{equation*}
    where each $U_p\bbra{\chi_{J'}}'_j$ is defined in \eqref{Constants Eq Lem Up' statement} and the equality uses $U_p\bbra{\chi_{\emptyset}}'_j=1$ for all $j\in\cJ$. By Lemma \ref{Constants Lem Up'}, it suffices to show that $\wt{U}_p(J)'\in p^{\ZZ}(-1)^{A(J)}[d(J)]$ for $A(J)\in\ZZ$ and $d(J)\in\FF\x$ as in \eqref{Constants Eq ABd}.

    We fix $j\in\cJ$ and compute $\wt{U}_p(J)'_j$. By definition, for $i\geq0$ we have
    \begin{equation}\label{Constants Eq j in deltass}
    \begin{aligned}
        &j\in\delta_{\ss}^i(J)\iff\bbra{j+i\in J,~\text{and}~j+i'\in J_{\rhobar}~\text{for}~0\leq i'\leq i-1}\\
        &j\in\delta_{\ss}^i(J^{\ss})\iff\bbra{j+i\in J,~\text{and}~j+i'\in J_{\rhobar}~\text{for}~0\leq i'\leq i}.
    \end{aligned}
    \end{equation}
    We let $0\leq k\leq f-1$ be the unique integer such that $j+i\in J_{\rhobar}$ for $1\leq i\leq k$, and $j+k+1\notin J_{\rhobar}$. We also write it as $k(j)$ to emphasize its dependence on $j$. We separate the following two cases.

    \hspace{\fill}
    
    \noindent\textbf{Case 1.}~Assume that $j\in J_{\rhobar}$. Then we write
    \begin{align*}
    &\begin{aligned}
        S_1(J)_j&\eqdef\#\bigset{i\geq0:U_p\bbra{\chi_{\delta_{\ss}^i(J)}}'_j=Y_j}\\
        &=\#\bigset{i\geq0:j\in\delta_{\ss}^i(J),~j+1\notin\delta_{\ss}^i(J)}=\bbra{\ssum_{i'=0}^{k}\delta_{j+i'\in\partial J}}+\delta_{j+k+1\in J};
    \end{aligned}\\
    &\begin{aligned}
        S_2(J)_j&\eqdef\#\bigset{i\geq0:U_p\bbra{\chi_{\delta_{\ss}^i(J)}}'_j=-X_j}\\&=\#\bigset{i\geq0:j\notin\delta_{\ss}^i(J),~j+1\in\delta_{\ss}^i(J)}=\ssum_{i'=0}^{k}\delta_{j+i'\in\partial(J^c)};
    \end{aligned}\\
    &\begin{aligned}
        S_1(J^{\ss})_j&\eqdef\#\bigset{i\geq0:U_p\bbra{\chi_{\delta_{\ss}^i(J^{\ss})}}'_j=Y_j}\\
        &=\#\bigset{i\geq0:j\in\delta_{\ss}^i(J^{\ss}),~j+1\notin\delta_{\ss}^i(J^{\ss})}=\bbra{\ssum_{i'=0}^{k-1}\delta_{j+i'\in\partial J}}+\delta_{j+k\in J};
    \end{aligned}\\
    &\begin{aligned}
        S_2(J^{\ss})_j&\eqdef\#\bigset{i\geq0:U_p\bbra{\chi_{\delta_{\ss}^i(J^{\ss})}}'_j=-X_j}\\&=\#\bigset{i\geq0:j\notin\delta_{\ss}^i(J^{\ss}),~j+1\in\delta_{\ss}^i(J^{\ss})}=\ssum_{i'=0}^{k-1}\delta_{j+i'\in\partial(J^c)},
    \end{aligned}
    \end{align*}
    where in each formula the first equality follows from \eqref{Constants Eq Lem Up' statement} and the second equality follows from \eqref{Constants Eq j in deltass}. In particular, we have
    \begin{multline*}
        \bigbra{S_1(J)_j-S_1(J^{\ss})_j}-\bigbra{S_2(J)_j-S_2(J^{\ss})_j}=\bigbra{\delta_{j+k\in\partial J}+\delta_{j+k+1\in J}-\delta_{j+k\in J}}-\delta_{j+k\in\partial(J^c)}\\
        =\bigbra{\delta_{j+k\in J}\bra{1-\delta_{j+k+1\in J}}+\delta_{j+k+1\in J}-\delta_{j+k\in J}}-\bra{1-\delta_{j+k\in J}}\delta_{j+k+1\in J}=0.
    \end{multline*}
    Then using $X_jY_j=p$ in $R_0$ since $j\in J_{\rhobar}$, we deduce that
    \begin{equation}\label{Constants Eq Prop UpJ 1}
        \wt{U}_p(J)'_j=Y_j^{S_1(J)_j-S_1(J^{\ss})_j}(-X_j)^{S_2(J)_j-S_2(J^{\ss})_j}\in(-1)^{\delta_{j+k\in\partial(J^c)}}p^{\ZZ}.
    \end{equation}

    \hspace{\fill}

    \noindent\textbf{Case 2.}~Assume that $j\notin J_{\rhobar}$. Then we write
    \begin{align*}
    &\begin{aligned}
        N_1(J)_j&\eqdef\#\bigset{i\geq0:U_p\bbra{\chi_{\delta_{\ss}^i(J)}}'_j=-[\beta^{-2}d_j]\inv}\\
        &=\#\bigset{i\geq0:j\in\delta_{\ss}^i(J),~j+1\notin\delta_{\ss}^i(J)}=\delta_{j\in\partial J};
    \end{aligned}\\
    &\begin{aligned}
        N_2(J)_j&\eqdef\#\bigset{i\geq0:U_p\bbra{\chi_{\delta_{\ss}^i(J)}}'_j=[\beta^{-2}d_j]}\\&=\#\bigset{i\geq0:j\notin\delta_{\ss}^i(J),~j+1\in\delta_{\ss}^i(J)}=\delta_{j\in\partial(J^c)}+\ssum_{i'=2}^{k+1}\delta_{j+i'\in J};
    \end{aligned}\\
    &\begin{aligned}
        N_1(J^{\ss})_j&\eqdef\#\bigset{i\geq0:U_p\bbra{\chi_{\delta_{\ss}^i(J^{\ss})}}'_j=-[\beta^{-2}d_j]\inv}\\
        &=\#\bigset{i\geq0:j\in\delta_{\ss}^i(J^{\ss}),~j+1\notin\delta_{\ss}^i(J^{\ss})}=0;
    \end{aligned}\\
    &\begin{aligned}
        N_2(J^{\ss})_j&\eqdef\#\bigset{i\geq0:U_p\bbra{\chi_{\delta_{\ss}^i(J^{\ss})}}'_j=[\beta^{-2}d_j]}\\&=\#\bigset{i\geq0:j\notin\delta_{\ss}^i(J^{\ss}),~j+1\in\delta_{\ss}^i(J^{\ss})}=\ssum_{i'=1}^{k}\delta_{j+i'\in J},
    \end{aligned}
    \end{align*}
    where in each formula the first equality follows from \eqref{Constants Eq Lem Up' statement} and the second equality follows from \eqref{Constants Eq j in deltass}. In particular, we have
    \begin{equation*}
    \begin{aligned}
        N(J)_j&\eqdef\bigbra{N_2(J)_j-N_2(J^{\ss})_j}-\bigbra{N_1(J)_j-N_1(J^{\ss})_j}\\
        &=\bigbra{\delta_{j\in\partial(J^c)}+\delta_{j+k+1\in J}-\delta_{j+1\in J}}-\delta_{j\in\partial J}\\
        &=\bigbra{1-\delta_{j\in J}}\delta_{j+1\in J}+\delta_{j+k+1\in J}-\delta_{j+1\in J}+\delta_{j\in J}\bigbra{1-\delta_{j+1\in J}}=\delta_{j+k+1\in J}-\delta_{j\in J}.
    \end{aligned}
    \end{equation*}
    Hence we have
    \begin{equation}\label{Constants Eq Prop UpJ 2}
        \wt{U}_p(J)'_j=(-1)^{N_1(J)_j+N_1(J^{\ss})_j}[\beta^{-2}d_j]^{N(J)_j}=(-1)^{\delta_{j\in\partial J}}[\beta^{-2}d_j]^{N(J)_j}.
    \end{equation}

    \hspace{\fill}

    By the definition of $d(J)$, we have $d(J)=\sprod_{j\notin J_{\rhobar}}d_j^{M(J)_j}$ with (for each $j\notin J_{\rhobar}$)
    \begin{equation}\label{Constants Eq Prop UpJ 3}
    \begin{aligned}
        M(J)_j&=-\delta_{j\in J}+\bbra{\ssum_{i=0}^{\ell(J)-1}\delta_{j+1\in\delta_{\ss}^i(J)}}-\bbra{\ssum_{i=0}^{\ell(J^{\ss})-1}\delta_{j+1\in\delta_{\ss}^i(J^{\ss})}}\\
        &=-\delta_{j\in J}+\bbra{\ssum_{i=0}^{k(j)}\delta_{j+i+1\in J}}-\bbra{\ssum_{i=0}^{k(j)-1}\delta_{j+i+1\in J}}\\
        &=\delta_{j+k(j)+1\in J}-\delta_{j\in J}=N(J)_j.
    \end{aligned}
    \end{equation}
    By the definition of $k(j)$, we have that $j+k(j)+1$ is the first place after $j$ that is not in $J_{\rhobar}$, hence we have
    \begin{equation}\label{Constants Eq Prop UpJ 4}
        \ssum_{j\notin J_{\rhobar}}N(J)_j=\ssum_{j\notin J_{\rhobar}}\bbra{\delta_{j+k(j)+1\in J}-\delta_{j\in J}}=0.
    \end{equation}
    Combining \eqref{Constants Eq Prop UpJ 1}, \eqref{Constants Eq Prop UpJ 2}, \eqref{Constants Eq Prop UpJ 3} and \eqref{Constants Eq Prop UpJ 4}, we deduce that
    \begin{equation*}
    \begin{aligned}
        \wt{U}_p(J)'&=\sprod_{j=0}^{f-1}\wt{U}_p(J)'_j\\
        &\in p^{\ZZ}(-1)^{\sum_{j\in J_{\rhobar}}\delta_{j+k(j)\in\partial(J^c)}+\sum_{j\notin J_{\rhobar}}\delta_{j\in\partial J}}[\beta]^{-2\sum_{j\notin J_{\rhobar}}N(J)_j}\!\bbra{\sprod_{j\notin J_{\rhobar}}[d_j]^{N(J)_j}}\\
        &=p^{\ZZ}(-1)^{A(J)}[d(J)]
    \end{aligned}
    \end{equation*}
    for $A(J)\in\ZZ$ and $d(J)\in\FF\x$ as in \eqref{Constants Eq ABd}, which completes the proof.
\end{proof}

\section{Constants in the Diamond diagrams}\label{Constants Sec Diamond diagram}

We review the strategy of \cite{DL21} to compute the constants in the diagram $(\pi^{I_1}\into\pi^{K_1})$. When $\rhobar$ is non-semisimple, we specialize the general formula to some particular constants that are needed to prove the main result (Theorem \ref{Constants Thm main}), see Example \ref{Constants Ex formula gammaJ}. 

For $0\leq i\leq q-1$, we define (with the convention that $0^0\eqdef1$)
\begin{equation}\label{Constants Eq operator S}
\begin{aligned}
    S_i&\eqdef\sum\limits_{\lambda\in\Fq\x}[\lambda]^i\pmat{\lambda&1\\1&0}\in\cO[\GL_2(\Fq)];\\
    S^+_i&\eqdef\sum\limits_{\lambda\in\Fq\x}[\lambda]^i\pmat{1&0\\\lambda&1}\in\cO[\GL_2(\Fq)].
\end{aligned}
\end{equation}
As in \cite[Def.~4.1]{DL21}, we let $R:D_0(\rhobar)^{I_1}\to\bigbra{\soc_{\GL_2(\OK)}D_0(\rhobar)}^{I_1}$ be the unique map defined as follows: if $\chi:I\to\FF\x$ is an $I$-character such that $D_0(\rhobar)^{I_1}[\chi]\neq0$, then $R|_{D_0(\rhobar)^{I_1}[\chi]}$ is given by $S_{i(\chi)}$ for some unique $0\leq i(\chi)\leq q-1$, except when $\chi$ appears in $\bigbra{\soc_{\GL_2(\OK)}D_0(\rhobar)}^{I_1}$, in which case $R|_{D_0(\rhobar)^{I_1}[\chi]}$ is the identity. Given $\chi$, we write $R\chi$ for the $I$-character such that $R\bigbra{D_0(\rhobar)^{I_1}[\chi]}=D_0(\rhobar)^{I_1}[R\chi]$. Then we define $g_{\chi}:D_0(\rhobar)^{I_1}[R\chi]\to D_0(\rhobar)^{I_1}[R\chi^s]$ by the formula $g_{\chi}\bra{R(v)}=R\bigbra{\!\!\smat{0&1\\p&0}v}$ for $v\in D_0(\rhobar)^{I_1}[\chi]$. In particular, we have $g_{\chi^s}=g_{\chi}\inv$ for all $\chi$.

\begin{example}\label{Constants Ex define gammaJ}
    Suppose that $J_{\rhobar}\neq\cJ$. We let $J\subseteq\cJ$ such that $J\neq J^*$ (see Lemma \ref{Constants Lem J*} for $J^*$, which implies $J^{\nss}\neq\cJ$). By \cite[Lemma~4.1(iii)]{Wang2} we have $R\chi_J=\chi_{J^{\ss}}$ and $R\chi_J^s=\chi_{(J-1)^{\ss}}$. In particular, we have $R\chi_J=\chi_J$ if and only if $J=J^{\ss}$ and we have $R\chi_J^s\neq\chi_J^s$ by Lemma \ref{Constants Lem J*}. 

    By \eqref{Constants Eq muJJ'} with $J'=(J-1)^{\ss}$ and  \cite[Lemma~3.2.2.5(i)]{BHHMS2}, we have
    \begin{equation}\label{Constants Eq SvJs}
        S_{i(\chi_J^s)}v_{J^s}=(-1)^{f-1}P_1(J)\mu_{J,(J-1)^{\ss}}v_{(J-1)^{\ss}}
    \end{equation}
    with
    \begin{equation}\label{Constants Eq P1}
    \begin{aligned}
        i(\chi_J^s)&=\ssum_{j+1\in J\Delta(J-1)^{\ss}}\bigbra{p-1-s^{(J-1)^{\ss}}_j}p^j;\\
        P_1(\chi_J)&\eqdef\sprod_{j+1\in J\Delta(J-1)^{\ss}}\bigbra{p-1-s^{(J-1)^{\ss}}_j}!\in\FF\x.
    \end{aligned}
    \end{equation}
    If moreover $J\nsubseteq J_{\rhobar}$, then by Proposition \ref{Constants Prop mu for Js}(i) and \cite[Lemma~3.2.2.5(i)]{BHHMS2} we have
    \begin{equation}\label{Constants Eq SvJ}
        S_{i(\chi_J)}v_J=(-1)^{f-1}P_2(J)\mu_{J^s,J^{\ss}}v_{J^{\ss}}
    \end{equation}
    with
    \begin{equation}\label{Constants Eq P2}
    \begin{aligned}
        i(\chi_J)&=\ssum_{j+1\notin J^{\nss}}\bigbra{p-1-s^{J^{\ss}}_j}p^j;\\
        P_2(\chi_J)&\eqdef\sprod_{j+1\notin J^{\nss}}\bigbra{p-1-s^{J^{\ss}}_j}!\in\FF\x.
    \end{aligned}
    \end{equation}
    Combining \eqref{Constants Eq SvJs} and \eqref{Constants Eq SvJ} we get
    \begin{equation}\label{Constants Eq gchiJ}
        g_{\chi_J}(v_{J^{\ss}})=
    \left\{\begin{aligned}
        (-1)^{f-1}P_1(\chi_J)\mu_{J,(J-1)^{\ss}}v_{(J-1)^{\ss}}&\quad\text{if}~J\subseteq J_{\rhobar}\\
        \frac{(-1)^{f-1}P_1(\chi_J)\mu_{J,(J-1)^{\ss}}}{(-1)^{f-1}P_2(\chi_J)\mu_{J^s,J^{\ss}}}v_{(J-1)^{\ss}}&\quad\text{if}~J\nsubseteq J_{\rhobar}.
    \end{aligned}\right.
    \end{equation}
    
    For $J\nsubseteq J_{\rhobar}$ we define 
    \begin{equation}\label{Constants Eq gammaJ}
        \gamma(J)\eqdef(-1)^{\abs{J\cap(J-1)^{\nss}}}\frac{\mu_{*,J}}{\mu_{*,J^{\ss}}}\bbbra{\frac{\sprod_{i=0}^{\ell(J)-1}(-1)^{f-1}\mu_{\delta_{\ss}^i(J),\delta_{\ss}^{i+1}(J)}}{\sprod_{i=0}^{\ell(J^{\ss})-1}(-1)^{f-1}\mu_{\delta_{\ss}^i(J^{\ss}),\delta_{\ss}^{i+1}(J^{\ss})}}}.
    \end{equation}
    Consider the following two maps:
    \begin{equation*}
    \begin{aligned}
        &D_0(\rhobar)^{I_1}[R\chi_J]\xrightarrow{g_{\chi_J}} D_0(\rhobar)^{I_1}[R\chi_{\delta_{\ss}(J)}]\xrightarrow{g_{\chi_{\delta_{\ss}(J)}}}\cdots\rightarrow D_0(\rhobar)^{I_1}[R\chi_{\emptyset}];\\
        &D_0(\rhobar)^{I_1}[R\chi_J]=D_0(\rhobar)^{I_1}[R\chi_{J^{\ss}}]\xrightarrow{g_{\chi_{J^{\ss}}}} D_0(\rhobar)^{I_1}[R\chi_{\delta_{\ss}(J^{\ss})}]\xrightarrow{g_{\chi_{\delta_{\ss}(J^{\ss})}}}\cdots\rightarrow D_0(\rhobar)^{I_1}[R\chi_{\emptyset}].
    \end{aligned}
    \end{equation*}    
    Suppose that the composition
    \begin{equation*}
        \bbra{\sprod^{0}_{i=\ell(J^{\ss})-1}g_{\chi_{\delta_{\ss}^i(J^{\ss})}}}\inv\circ\bbra{\sprod^{0}_{i=\ell(J)-1}g_{\chi_{\delta_{\ss}^i(J)}}}:D_0(\rhobar)^{I_1}[R\chi_J]\to D_0(\rhobar)^{I_1}[R\chi_J]
    \end{equation*}
    is given by the scalar $g(J)\in\FF\x$. Then by \eqref{Constants Eq gchiJ} and Proposition \ref{Constants Prop ratio Js} we have
    \begin{equation}\label{Constants Eq gammaJ and gJ}
        \gamma(J)=(-1)^{f-1+\abs{J\cap(J-1)^{\nss}}}\mu_{J^s,J}\bigbra{P_2(J)/P_1(J)}g(J),
    \end{equation}
    where $P_1(J)\eqdef\bigbra{\sprod_{i=0}^{\ell(J)-1}P_1\bigbra{\chi_{\delta_{\ss}^i(J)}}}/\bigbra{\sprod_{i=0}^{\ell(J^{\ss})-1}P_1\bigbra{\chi_{\delta_{\ss}^i(J^{\ss})}}}$ and $P_2(J)\eqdef P_2(\chi_J)$.
\end{example}

\begin{lemma}\label{Constants Lem muempty}
    We have $\mu_{\emptyset,\emptyset}=(-1)^{f-1}\xi$ (see \eqref{Constants Eq rhobar} for $\xi$).
\end{lemma}

\begin{proof}
    By the proof of \cite[Lemma~4.17]{DL21}, the map $g_{\chi_{\emptyset}}$ is given by the reduction modulo $\fm_0$ of $U_p(\chi_{\emptyset})\in R_0$ (see \eqref{Constants Eq UpchiJ}), which equals $\xi$ by Lemma \ref{Constants Lem Up'} (and its proof). Then we conclude using \eqref{Constants Eq gchiJ} with $J=\emptyset$.
\end{proof}

\hspace{\fill}

We let $M_{\infty}$ be the $\cO[\GL_2(K)]$-module as in \cite[\S6.2]{DL21}. Then $M_{\infty}$ has a minimal arithmetic action of $R_{\infty}$ in the sense of \cite[\S4.2]{DL21}, where $R_{\infty}$ is a suitable power series ring over $R_{\rhobar}$, the universal framed $\cO$-deformation ring for $\rhobar$. In particular, we have $M_{\infty}^{\vee}[\fm_{\infty}]\cong\pi$ for $\pi$ as in \eqref{Constants Eq local factor}, where $(-)^{\vee}$ is the Pontrjagin dual and $\fm_{\infty}$ is the maximal ideal of $R_{\infty}$. We let $M_{\infty}(-)$ be the corresponding patching functor.

Let $\theta$ be a non-scalar tame type. For each Serre weight $\sigma\in\JH\bra{\ovl{\theta}}$ where $\ovl{\theta}$ is the reduction modulo $\varpi$ of any $\cO$-lattice in $\theta$, we fix a lattice $\theta^{\sigma}$ in $\theta$ with irreducible cosocle $\sigma$. Such a lattice is unique up to homothety and we rescale it when necessary. For $\chi:I\to\FF\x$ a character such that $\chi\neq\chi^s$, we write $\theta^{\chi}$ for $\theta^{\sigma(\chi)}$. We let $\op{pr}_{\chi^s}:\theta(\chi^s)^{R\chi^s}\to\sigma(R\chi^s)$ and $\op{pr}_{\chi}:\theta(\chi^s)^{R\chi}\to\sigma(R\chi)$ be the normalized surjections as in \cite[(23),(24)]{DL21}.

We fix a tame inertial type $\tau_0\eqdef\tau(w,\mu-w\eta)$ with associated tame type $\theta_0\eqdef R_w(\mu-w\eta)$ for some fixed $w\in\un{W}$ satisfying $w_j=\fw$ for $j\in J_{\rhobar}$ and $w_0w_1\cdots w_{f-1}=\fw$ if $J_{\rhobar}\neq\cJ$. Then we have $W(\rhobar)\subseteq\JH(\ovl{\theta}_0)$ by \cite[Prop.~3.11]{DL21} and $\theta_0$ is a cuspidal type if $J_{\rhobar}\neq\cJ$. Moreover, the ring $R_0$ defined above \eqref{Constants Eq Af-1-j} is a power series ring over $R_{\rhobar}^{\tau_0}$ (see \cite[\S3.5.1]{DL21}), where $R_{\rhobar}^{\tau_0}$ is the quotient of $R_{\rhobar}$ parametrizing potentially crystalline lifts of $\rhobar$ with Hodge--Tate weights $(1,0)$ in each embedding and inertial type $\tau_0$. In particular, all the arguments of \cite[\S4]{DL21} still hold, replacing the so-called central type $\theta=R_{\un{\fw}}\bra{\mu-\un{\fw}\eta}$ with the type $\theta_0$.

For any $\cO$-lattice $\theta_0\cc$ in $\theta_0$, the patched module $M_{\infty}(\theta_0\cc)$ is supported on $R_{\infty}(\tau_0)\eqdef R_{\infty}\otimes_{R_{\rhobar}}R_{\rhobar}^{\tau_0}$. We let $Q(\chi^s)^{R\chi^s}$ (resp.\,$Q(\chi^s)^{R\chi}$) be the quotient of $\theta(\chi^s)^{R\chi^s}/\varpi$ (resp.\,$\theta(\chi^s)^{R\chi}/\varpi$) as in \cite[Prop.~4.18]{DL21}. Then the surjection $\op{pr}_{\chi^s}$ (resp.\,$\op{pr}_{\chi}$) factors through $Q(\chi^s)^{R\chi^s}$ (resp.\,$Q(\chi^s)^{R\chi}$). If we fix a surjection $\alpha:\theta_0^{R\chi^s}\to Q(\chi^s)^{R\chi^s}$ which induces a surjection $\alpha:\theta_0^{R\chi}\to Q(\chi^s)^{R\chi}$, then as in \cite[(29)]{DL21} there is a commutative diagram
\begin{equation*}
\begin{tikzcd}
    M_{\infty}\bra{\theta_0^{R\chi}}\arrow[hook]{r}{\iota}\arrow{dr}[swap]{\wt{U}_p(\chi)}&M_{\infty}\bra{\theta_0^{R\chi^s}}\arrow{r}{\op{pr}_{\chi^s}\circ\alpha}\arrow[d,"\wt{h}_{\chi}","\cong"']&M_{\infty}\bra{\sigma(R\chi^s)}/\fm_{\infty}\arrow[d,"\ovl{h}_{\chi}","\cong"']\\
    &M_{\infty}\bra{\theta_0^{R\chi}}\arrow{r}{\op{pr}_{\chi}\circ\alpha}&M_{\infty}\bra{\sigma(R\chi^s)}/\fm_{\infty},
\end{tikzcd}
\end{equation*}
where we refer to \cite[\S4.4]{DL21} for the maps $\iota$, $\wt{h}_{\chi}$, $\ovl{h}_{\chi}$ and the element $\wt{U}_p(\chi)\in R_{\infty}(\tau_0)$. Moreover, by \cite[Lemma~5.5]{DL21} and the definition of $\wt{U}'_p(\chi)$ in \cite[Def.~3.22]{DL21}, we deduce that $\wt{U}_p(\chi)\wt{U}_p(\chi^s)$ is a product of an integer power of $p$ and a $1$-unit of $R_{\infty}(\tau_0)$ (i.e.\,an element of $1+\fm_{\infty}(\tau_0)$, where $\fm_{\infty}(\tau_0)$ is the maximal ideal of $R_{\infty}(\tau_0)$). In particular, $\wt{U}_p(\chi)\in R_{\infty}(\tau_0)[1/p]\x$.

For our purposes, we consider a cycle of characters but in a different order from that in \cite[\S4.5]{DL21}. Namely, we consider the $I$-characters $\psi_0,\psi_1,\ldots,\psi_n$ and $\psi'_0,\psi'_1,\ldots,\psi'_m$ appearing in $D_0(\rhobar)^{I_1}$ such that $R\psi_0=R\psi'_0$, $R\psi_n^s=R\psi_m^{\prime s}$, $R\psi_i^s=R\psi_{i+1}$ for $0\leq i\leq n-1$ and $R\psi_i^{\prime s}=R\psi'_{i+1}$ for $0\leq i\leq m-1$. Here we allow $m=-1$, in which case we are reduced to the situation considered in \cite[\S4.5]{DL21}. We fix a surjection $\alpha_0:\theta_0^{R\psi_0^s}\to Q(\psi_0^s)^{R\psi_0^s}$. We define the surjection $\alpha'_0:\theta_0^{R\psi_0^{\prime s}}\to Q(\psi_0^{\prime s})^{R\psi_0^{\prime s}}$ by the commutative diagram (if $m\geq0$)

\begin{equation*}
\begin{tikzcd}  
    \theta_0^{R\psi_{0}}\arrow{r}{\alpha_0}\arrow[equal]{d}&Q(\psi_{0}^s)^{R\psi_{0}}\arrow{r}{\op{pr}_{\psi_{0}}}&\sigma\bra{R\psi_{0}}\arrow[equal]{d}\\
    \theta_0^{R\psi_{0}'}\arrow{r}{\alpha_0'}&Q(\psi_0^{\prime s})^{R\psi_0'}\arrow{r}{\op{pr}_{\psi_{0}'}}&\sigma\bra{R\psi_{0}'}.
\end{tikzcd}
\end{equation*}
Then we define the surjections $\alpha_i:\theta_0^{R\psi_i^s}\to Q(\psi_i^s)^{R\psi_i^s}$ for $1\leq i\leq n$ inductively by the commutative diagram
\begin{equation*}
\begin{tikzcd}
    \theta_0^{R\psi_{i+1}}\arrow{r}{\alpha_{i+1}}\arrow[equal]{d}&Q(\psi_{i+1}^s)^{R\psi_{i+1}}\arrow{r}{\op{pr}_{\psi_{i+1}}}&\sigma\bra{R\psi_{i+1}}\arrow[equal]{d}\\
    \theta_0^{R\psi_{i}^s}\arrow{r}{\alpha_i}&Q(\psi_i^s)^{R\psi_i^s}\arrow{r}{\op{pr}_{\psi_{i}^s}}&\sigma\bra{R\psi_{i}^s},
\end{tikzcd}
\end{equation*}
and we define the surjections $\alpha_i':\theta_0^{R\psi_i^{\prime s}}\to Q(\psi_i^{\prime s})^{R\psi_i^{\prime s}}$ for $1\leq i\leq m$ inductively in a similar way. 

Analogous to the picture in \cite[\S4.5]{DL21}, we give a picture for $n=1$ and $m=0$.

\begin{equation*}
\begin{tikzcd}
    M_{\infty}\bra{\theta_0^{R\psi_0}}\arrow[hook]{r}\arrow{dr}[swap]{\wt{U}_p(\psi_1)}&M_{\infty}\bra{\theta_0^{R\psi_0^s}}\arrow[hook]{r}\arrow{dr}[swap]{\wt{U}_p(\psi_1)}&M_{\infty}\bra{\theta_0^{R\psi_1^s}}\arrow{r}{\op{pr}_{\psi_1^s}\circ\alpha_1}\arrow{d}{\wt{h}_{\psi_1}}&M_{\infty}\bra{\sigma(R\psi_1^s)}/\fm_{\infty}\arrow{d}{\ovl{h}_{\psi_1}}\\
    &M_{\infty}\bra{\theta_0^{R\psi_0}}\arrow[hook]{r}\arrow{dr}[swap]{\wt{U}_p(\psi_0)}&M_{\infty}\bra{\theta_0^{R\psi_0^s}}\arrow{r}{\op{pr}_{\psi_1}\circ\alpha_1}[swap]{\op{pr}_{\psi_0^s}\circ\alpha_0}\arrow{d}{\wt{h}_{\psi_0}}&M_{\infty}\bra{\sigma(R\psi_0^s)}/\fm_{\infty}\arrow{d}{\ovl{h}_{\psi_0}}\\
    & &M_{\infty}\bra{\theta_0^{R\psi_0}}\arrow{r}{\op{pr}_{\psi_0}\circ\alpha_0}[swap]{\op{pr}_{\psi_0'}\circ\alpha_0'}&M_{\infty}\bra{\sigma(R\psi_0)}/\fm_{\infty}\\
    &M_{\infty}\bra{\theta_0^{R\psi_0'}}\arrow[hook]{r}\arrow{ur}{\wt{U}_p(\psi_0')}&M_{\infty}\bra{\theta_0^{R\psi_0^{\prime s}}}\arrow{r}{\op{pr}_{\psi_0^{\prime s}}\circ\alpha_0'}\arrow{u}[swap]{\wt{h}_{\psi_0'}}&M_{\infty}\bra{\sigma(R\psi_0^{\prime s})}/\fm_{\infty}\arrow{u}[swap]{\ovl{h}_{\psi_0'}}.
\end{tikzcd}
\end{equation*}
Suppose that $\op{pr}_{\psi_n^s}\circ\alpha_n=c(\psi,\psi')\op{pr}_{\psi_m^{\prime s}}\circ\alpha_m'$ for $c(\psi,\psi')\in\FF\x$. Suppose that the composition
\begin{equation*}
    \ovl{h}_{\psi_n}\inv\!\circ\!\cdots\!\circ\!\ovl{h}_{\psi_0}\inv\!\circ\!\ovl{h}_{\psi_0'}\!\circ\!\cdots\!\circ\!\ovl{h}_{\psi_m'}:M_{\infty}\bra{\sigma(R\psi_n^s)}/\fm_{\infty}=M_{\infty}\bra{\sigma(R\psi_m^{\prime s})}/\fm_{\infty}\to M_{\infty}\bra{\sigma(R\psi_n^s)}/\fm_{\infty}
\end{equation*}
is given by multiplication by $h(\psi,\psi')\in\FF\x$. Analogous to \cite[(34)]{DL21}, there exists $\nu\in\ZZ$ such that the element
\begin{equation*}
    p^{-\nu}\bbra{\sprod_{i=0}^n\wt{U}_p(\psi_i)}\inv\!\bbra{\sprod_{i=0}^m\wt{U}_p(\psi_i')}\in R_{\infty}(\tau_0)[1/p]\x
\end{equation*}
lies in $R_{\infty}(\tau_0)$ and reduces to $h(\psi,\psi')c(\psi,\psi')\inv$ modulo $\fm_{\infty}(\tau_0)$.

\begin{example}\label{Constants Ex formula gammaJ}
    Suppose that $J_{\rhobar}\neq\cJ$. We let $J\subseteq\cJ$ such that $J\nsubseteq J_{\rhobar}$ and $J\neq J^*$ (see Lemma \ref{Constants Lem J*} for $J^*$). Then we take $n=\ell(J)-1$, $m=\ell(J^{\ss})-1$, $\psi_i=\chi_{\delta_{\ss}^{n-i}(J)}^s$ for $0\leq i\leq n$ and $\psi_i'=\chi_{\delta_{\ss}^{m-i}(J^{\ss})}^s$ for $0\leq i\leq m$, and write $c(J)$ for $c(\psi,\psi')\in\FF\x$. From the previous paragraph using \cite[Prop.~4.16]{DL21} and \cite[Lemma~5.5]{DL21} we deduce that $g(J)=U_p(J)c(J)$ (see \eqref{Constants Eq gammaJ and gJ} for $g(J)$ and \eqref{Constants Eq UpJ} for $U_p(J)$). Combining with \eqref{Constants Eq gammaJ and gJ} we conclude that (see \eqref{Constants Eq gammaJ} for $\gamma(J)$)
    \begin{equation}\label{Constants gamma Up c'}
        \gamma(J)=U_p(J)c'(J)
    \end{equation}
    with $c'(J)\eqdef(-1)^{f-1+\abs{J\cap(J-1)^{\nss}}}\mu_{J^s,J}\bigbra{P_2(J)/P_1(J)}c(J).$
\end{example}

\section{Computation of constants}\label{Constants Sec c'J}

Throughout this section, we suppose that $J_{\rhobar}\neq\cJ$, equivalently, $\rhobar$ is non-semisimple. We compute $c'(J)$ defined in \eqref{Constants gamma Up c'} for $J\nsubseteq J_{\rhobar}$ and $J\neq J^*$ (see Lemma \ref{Constants Lem J*} for $J^*$). The main results are Propostion \ref{Constants Prop c'J simple} and Proposition \ref{Constants Prop c'J new}. Together with the results of \S\ref{Constants Sec Kisin}, we finish the computation of all the constants in the diagram $(\pi^{I_1}\into\pi^{K_1})$ that we need.

\subsection{Relation between $S$-operators}

For $a\in\ZZ$, we define $a_j\in\set{0,1,\ldots,p-1}$ for $0\leq j\leq f-1$ by writing $a=\sum_{j=0}^{f-1}a_jp^j+Q(q-1)$ for some $Q\in\ZZ$ and we define $a_{q}\eqdef\sum_{j=0}^{f-1}a_jp^j\in\set{0,1,\ldots,q-2}$. If $(q-1)\nmid a$, we write $S_a$ (resp.\,$S^+_a$) for the operators $S_{a_q}$ (resp.\,$S^+_{a_q}$) defined in \eqref{Constants Eq operator S}. For $0\neq b\in K$, we let $u\in\ZZ$ be such that $b\in p^u\OK\x$, then we define the \textbf{leading term} of $b$ to be the element of $\Fq\x$ that is the reduction modulo $p$ of $p^{-u}b\in\OK\x$. 

For $a,b\in\ZZ$, we define
\begin{equation}\label{Constants Eq Jab}
\begin{aligned}
    u(a,b)&\eqdef(p-1)\inv\ssum_{j=0}^{f-1}\bigbra{p-1-\bra{a_j+b_j-(a+b)_j}}\in\ZZ;\\
    J(a,b)&\eqdef(-1)^{f-1+u(a,b)}\sprod_{j=0}^{f-1}\bbra{a_j!b_j!\bigbra{(a+b)_j!}\inv}\in\Fq\x.
\end{aligned}  
\end{equation}
More generally, for $a_1,\ldots,a_n\in\ZZ$, we define
\begin{equation}\label{Constants Eq Ja1an}
\begin{aligned}
    u(a_1,\ldots,a_n)&\eqdef(p-1)\inv\ssum_{j=0}^{f-1}\Bigbra{\ssum_{i=1}^n\bigbra{p-1-(a_i)_j}-\bigbra{p-1-\bra{\ssum_{i=1}^na_i}_j}}\in\ZZ;\\
    J\bra{a_1,\ldots,a_n}&\eqdef(-1)^{(n-1)(f-1)+u(a_1,\ldots,a_n)}\sprod_{j=0}^{f-1}\bbra{\bigbra{\sprod_{i=1}^n(a_i)_j!}\bigbra{\bra{\ssum_{i=1}^na_i}_j!}\inv}\in\Fq\x.
\end{aligned}
\end{equation}

\begin{lemma}\label{Constants Lem S+S+S+}
\begin{enumerate}
    \item 
    Let $0<a,b<q-1$ such that $a+b\neq q-1$. Then we have
    \begin{equation*}
        S^+_{a}S^+_{b}=\wt{J}(a,b)S^+_{a+b};\quad
        S_{a}S^+_{b}=\wt{J}(a,b)S_{a+b},
    \end{equation*}
    where $0\neq\wt{J}(a,b)\in\OK$ has leading term $J(a,b)$. 
    \item
    Let $0<a_1,\ldots,a_n<q-1$ such that $(q-1)$ does not divide $\sum_{i=1}^ka_i$ for all $1\leq k\leq n-1$. Then we have
    \begin{equation*}
    \begin{aligned}
        S^+_{a_1}\ldots S^+_{a_n}&=\wt{J}(a_1,\ldots,a_n)S^+_{a_1+\cdots+a_n}&\quad\text{if}~(q-1)\nmid\ssum_{i=1}^na_i,\\
        S^+_{a_1}\ldots S^+_{a_n}&=J'(a_1,\ldots,a_n)S^+_0+\wt{J}(a_1,\ldots,a_n)\smat{1&0\\0&1}&\quad\text{if}~(q-1)\mid\ssum_{i=1}^na_i,
    \end{aligned}
    \end{equation*}
    where $0\neq\wt{J}(a_1,\ldots,a_n)\in\OK$ has leading term $J(a_1,\ldots,a_n)$, and $J'(a_1,\ldots,a_n)\in\OK$.
\end{enumerate}
\end{lemma}

\begin{proof}
    (i) is \cite[Lemma~2.3]{DL21} and \cite[Lemma~2.4]{DL21}. The first formula of (ii) follows from (i) by induction, and the second formula of (ii) is \cite[Prop.~2.5]{DL21}.
\end{proof}

\begin{lemma}\label{Constants Lem S+v}
    Let $v$ be an $H$-eigenvector in an $\cO[\GL_2(\Fq)]$-module with $H$-eigencharacter $\chi$. Then $S_iv$ has $H$-eigencharacter $\chi^s\alpha^{-i}$ and $S^+_iv$ has $H$-eigencharacter $\chi\alpha^i$, where $\alpha$ is the $H$-character $\smat{a&0\\0&d}\mapsto ad\inv$.
\end{lemma}

\begin{proof}
    This is \cite[Lemma~2.7]{DL21}.
\end{proof}

\begin{lemma}\label{Constants Lem SSv}
    Let $J\subseteq\cJ$ and $v$ be an $I$-eigenvector in an $\cO[\GL_2(\Fq)]$-module with $I$-character $[\chi_J]$. Let $a,b\in\ZZ$ such that $(q-1)$ does not divide $a,a-b,a-b-\un{s}^J$ (see \eqref{Constants Eq sJ} for $\un{s}^J$, which also denotes the integer $\ssum_{j=0}^{f-1}s^J_jp^j\in\ZZ$). Then we have (see \eqref{Constants Eq tJ} for $\un{t}^J$)
    \begin{equation*}
        S_aS_bv=(-1)^{\un{t}^J}\wt{J}\bra{a,-b-\un{s}^J}S_{a-b-\un{s}^J}v,
    \end{equation*}
    where $0\neq\wt{J}(a,-b-\un{s}^J)\in\OK$ has leading term $J(a,-b-\un{s}^J)$.
\end{lemma}

\begin{proof}
    This is \cite[Prop.~2.8]{DL21} with $\chi=\chi_J$.
\end{proof}

\begin{lemma}\label{Constants Lem Jab}
    Suppose that $J\neq J^*$. Then we have (see \eqref{Constants Eq P1} for $i(\chi_J^s)$ and see \eqref{Constants Eq various J} for $J^{\delta}$)
    \begin{equation*}
        J\bigbra{i(\chi_J^s),-\un{s}^J}=(-1)^{1+\sum_{j+1\in J^{\delta}}s^{(J-1)^{\ss}}_j}\bbbra{\frac{\prod_{j+1\in J^{\delta},j\notin J^{\delta}}(-1)\bigbra{s^{(J-1)^{\ss}}_j+1}}{\prod_{j+1\notin J^{\delta},j\in J^{\delta}}\bigbra{s^{(J-1)^{\ss}}_j+1}}}.
    \end{equation*}
\end{lemma}

\begin{proof}
    We write $a\eqdef i(\chi_J^s)\in\ZZ$ and $b\eqdef-\un{s}^J\in\ZZ$ so that $a_j=\delta_{j+1\in J^{\delta}}\bigbra{p-1-s^{(J-1)^{\ss}}_j}$ and $b_j=p-1-s^J_j$ for each $j\in\cJ$. By Lemma \ref{Constants Lem compare sJ}(ii) we have
    \begin{equation}\label{Constants Eq Lem Jab 1}
        a_j+b_j=
    \begin{cases}
        \bigbra{p-1-s^{(J-1)^{\ss}}_j}+\bigbra{p-1-s^J_j}=p-\delta_{j\in J^{\delta}}&\text{if}~j+1\in J^{\delta}\\
        p-1-s^J_j=p-1-s^{(J-1)^{\ss}}_j-\delta_{j\in J^{\delta}}&\text{if}~j+1\notin J^{\delta},
    \end{cases}
    \end{equation}
    hence we have
    \begin{equation}\label{Constants Eq Lem Jab 2}
    \begin{aligned}
        a+b&\equiv\ssum_{j=0}^{f-1}(a_j+b_j)p^j\\
        &=\ssum_{j+1\in J^{\delta}}\bigbra{p-\delta_{j\in J^{\delta}}}p^j+\ssum_{j+1\notin J^{\delta}}\bigbra{p-1-s^{(J-1)^{\ss}}_j-\delta_{j\in J^{\delta}}}p^j\\
        &=\bbra{\ssum_{j+1\notin J^{\delta}}\bigbra{p-1-s^{(J-1)^{\ss}}_j}p^j}+\bbra{\ssum_{j+1\in J^{\delta}}p^{j+1}-\ssum_{j\in J^{\delta}}p^j}\\
        &\equiv\ssum_{j+1\notin J^{\delta}}\bigbra{p-1-s^{(J-1)^{\ss}}_j}p^j\quad\mod~(q-1),
    \end{aligned}
    \end{equation}
    which implies $(a+b)_j=0$ if $j+1\in J^{\delta}$ and $(a+b)_j=p-1-s^{(J-1)^{\ss}}_j$ if $j+1\notin J^{\delta}$.
    
    Then by \eqref{Constants Eq Jab} using \eqref{Constants Eq Lem Jab 1} and \eqref{Constants Eq Lem Jab 2} we have
    \begin{align*}
        u(a,b)&=(p-1)\inv\ssum_{j=0}^{f-1}\bigbra{p-1-(a_j+b_j-(a+b)_j)}\\
        &=(p-1)\inv\bbra{\ssum_{j+1\notin J^{\delta}}(p-1+\delta_{j\in J^{\delta}})+\ssum_{j+1\in J^{\delta}}\bigbra{\!-\!1+\delta_{j\in J^{\delta}}}}\\
        &=(p-1)\inv\bbra{(p-1)\#\set{j:j+1\notin J^{\delta}}+\#\set{j:j\in J^{\delta}}-\#\set{j:j+1\in J^{\delta}}}\\
        &=f-\bigabs{J^{\delta}},
    \end{align*}
    and
    \begin{align*}
        &J\bigbra{i(\chi_J^s),-\un{s}^J}=(-1)^{f-1+u(a,b)}\sprod_{j=0}^{f-1}\bbra{a_j!b_j!\bigbra{(a+b)_j!}\inv}\\
        &\hspace{1cm}=(-1)^{f-1+u(a,b)}\bbbra{\frac{\sprod_{j+1\in J^{\delta},j\in J^{\delta}}\bigbra{a_j!(p-1-a_j)!}\sprod_{j+1\in J^{\delta},j\notin J^{\delta}}\bigbra{a_j!(p-a_j)!}}{\sprod_{j+1\notin J^{\delta},j\in J^{\delta}}(a+b)_j}}\\
        &\hspace{1cm}=(-1)^{f-1+f-|J^{\delta}|+\sum_{j+1\in J^{\delta}}\bra{s^{(J-1)^{\ss}}_j+1}}\bbbra{\frac{\sprod_{j+1\in J^{\delta},j\notin J^{\delta}}\bigbra{s^{(J-1)^{\ss}}_j+1}}{\sprod_{j+1\notin J^{\delta},j\in J^{\delta}}\bigbra{p-1-s^{(J-1)^{\ss}}_j}}}\\
        &\hspace{1cm}=(-1)^{1+\sum_{j+1\in J^{\delta}}s^{(J-1)^{\ss}}_j}\bbbra{\frac{\prod_{j+1\in J^{\delta},j\notin J^{\delta}}(-1)\bigbra{s^{(J-1)^{\ss}}_j+1}}{\prod_{j+1\notin J^{\delta},j\in J^{\delta}}\bigbra{s^{(J-1)^{\ss}}_j+1}}},
    \end{align*}
    where the third equality follows from \eqref{Constants Eq factorial} and the last equality uses $\#\set{j:j+1\in J^{\delta},j\notin J^{\delta}}=\#\set{j:j+1\notin J^{\delta},j\in J^{\delta}}$.
\end{proof}

\begin{lemma}\label{Constants Lem ichiJ+sJ}
    Let $J\nsubseteq J_{\rhobar}$. Then for each $j\in\cJ$ we have (see \eqref{Constants Eq P2} for $i(\chi_J)$)
    \begin{equation*}
        \bigbra{i(\chi_J)+\un{s}^J}_j=\delta_{j+1\in J^{\nss}}\bigbra{p-1-s^{J^{\ss}}_j}.
    \end{equation*}
\end{lemma}

\begin{proof}
    This follows from the computation
    \begin{align*}
        i(\chi_J)+\un{s}^J&\equiv\ssum_{j=0}^{f-1}\bigbra{i(\chi_J)_j+s^J_j}p^j\\
        &=\ssum_{j+1\notin J^{\nss}}\bigbra{p-1-s^{J^{\ss}}_j+s^J_j}p^j+\ssum_{j+1\in J^{\nss}}s^J_jp^j\\
        &=\ssum_{j+1\notin J^{\nss}}\bigbra{p-\delta_{j\notin J^{\nss}}}p^j+\ssum_{j+1\in J^{\nss}}\bigbra{p-1-s^{J^{\ss}}_j-\delta_{j\notin J^{\nss}}}p^j\\
        &=\bbra{\ssum_{j+1\in J^{\nss}}\bigbra{p-1-s^{J^{\ss}}_j}p^j}+\bbra{\ssum_{j+1\notin J^{\nss}}p^{j+1}-\ssum_{j\notin J^{\nss}}p^j}\\
        &\equiv\ssum_{j+1\in J^{\nss}}\bigbra{p-1-s^{J^{\ss}}_j}p^j\quad\mod~(q-1),
    \end{align*}
    where the second equality follows from Lemma \ref{Constants Lem compare sJ}(ii).
\end{proof}

\subsection{The case \texorpdfstring{$(J-1)^{\ss}=J^{\ss}$}{.}}

\begin{lemma}\label{Constants Lem pr Hu}
    Let $\emptyset\neq J\subseteq\cJ$ such that $J\neq J^*$ and $(J-1)^{\ss}=J^{\ss}$ (which implies $J\nsubseteq J_{\rhobar}$ and $J^{\nss}\neq\cJ$). Then we have (see \S\ref{Constants Sec Kisin} for $\theta\cc(\chi_J)$ and $\varphi^{\chi_J}$)
    \begin{equation*}
        S_{i(\chi_J^s)}S_0\varphi^{\chi_J}=\wt{c}(\chi_J)S_{i(\chi_J)}\varphi^{\chi_J}\quad\text{in}~\theta\cc(\chi_J),
    \end{equation*}
    where $0\neq\wt{c}(\chi_J)\in\OK$ has leading term $c(\chi_J)\eqdef(-1)^{\un{t}^J}J\bigbra{i(\chi_J^s),-\un{s}^J}$.
\end{lemma}

\begin{proof}
    Since $(J-1)^{\ss}=J^{\ss}$, we deduce from \eqref{Constants Eq Lem Jab 2} that 
        $i(\chi_J^s)-\un{s}^J\equiv i(\chi_J)~\mod\,(q-1).$
    Moreover, our assumption implies $J^{\nss}\notin\set{\emptyset,\cJ}$, hence $(q-1)\nmid i(\chi_J^s)$ and $(q-1)\nmid i(\chi_J)$. Then we conclude using Lemma \ref{Constants Lem SSv}.
\end{proof}

Recall that $c'(J)=(-1)^{f-1+\abs{J\cap(J-1)^{\nss}}}\mu_{J^s,J}\bigbra{P_2(J)/P_1(J)}c(J)$ for $J\nsubseteq J_{\rhobar}$ and $J\neq J^*$ with $\mu_{J^s,J}$ defined in Proposition \ref{Constants Prop mu for Js}(ii), $P_i(J)$ defined in Example \ref{Constants Ex define gammaJ} for $i\in\set{1,2}$ and $c(J)$ defined in Example \ref{Constants Ex formula gammaJ}. 

\begin{proposition}\label{Constants Prop c'J simple}
    Let $\emptyset\neq J\subseteq\cJ$ such that $J\neq J^*$ and $(J-1)^{\ss}=J^{\ss}$. Then we have (see \eqref{Constants gamma Up c'} for $c'(J)$ and see \eqref{Constants Eq ABd} for $A(J)$)
    \begin{equation*}
        c'(J)=(-1)^{A(J)}.
    \end{equation*}
\end{proposition}

\begin{proof}
    By \eqref{Constants Eq muJsJ} using \eqref{Constants Eq factorial} and Lemma \ref{Constants Lem compare sJ}(ii) we have
    \begin{equation}\label{Constants Eq Prop c'J simple 1}
        \mu_{J^s,J}=(-1)^{\un{t}^{J^s}+\sum_{j+1\in J^{\nss}}\bra{s^{J^{\ss}}_j+\delta_{j\in J^{\nss}}}}\bbbra{\frac{\sprod_{j+1\notin J^{\nss}}\bra{s^{J^{\ss}}_j+\delta_{j\in J^{\nss}}}!}{\sprod_{j+1\in J^{\nss}}\bra{s^{J^{\ss}}_j+\delta_{j\notin J^{\nss}}}!}}.
    \end{equation}
    Since $(J-1)^{\ss}=J^{\ss}$, we have $\delta_{\ss}^{i+1}(J)=\delta_{\ss}^{i}(J^{\ss})$ for all $i\geq0$, hence we have 
    \begin{equation}\label{Constants Eq Prop c'J simple 2}
    \begin{aligned}
        P_2(J)/P_1(J)&=P_2(\chi_J)/P_1(\chi_J)\\
        &=\frac{\sprod_{j+1\notin J^{\nss}}\bigbra{p-1-s^{J^{\ss}}_j}!}{\sprod_{j+1\in J^{\nss}}\bigbra{p-1-s^{J^{\ss}}_j}!}
        =(-1)^{\un{s}^{J^{\ss}}+\un{1}}\bbbra{\frac{\sprod_{j+1\in J^{\nss}}\bigbra{s^{J^{\ss}}_j}!}{\sprod_{j+1\notin J^{\nss}}\bigbra{s^{J^{\ss}}_j}!}}.
    \end{aligned}
    \end{equation}
    Recall from Example \ref{Constants Ex formula gammaJ} that $c(J)=c(\psi,\psi')$ with $\psi_i=\chi_{\delta_{\ss}^{n-i}(J)}^s$ for $0\leq i\leq n\eqdef\ell(J)-1$ and $\psi_i'=\chi_{\delta_{\ss}^{m-i}(J^{\ss})}^s$ for $0\leq i\leq m\eqdef\ell(J^{\ss})-1$. Since $(J-1)^{\ss}=J^{\ss}$, we have $n=m+1$ and $\psi_{i+1}=\psi'_i$ for $0\leq i\leq m$, hence $c(J)$ is also equal to $c(\psi,\psi')$ with $n=0$, $m=-1$ and $\psi_0=\chi_J^s$ by definition. Then we deduce from Lemma \ref{Constants Lem Jab} and Lemma \ref{Constants Lem pr Hu} that
    \begin{equation}\label{Constants Eq Prop c'J simple 3}
        c(J)=c(\chi_J)=(-1)^{1+\un{t}^J+\sum_{j+1\in J^{\nss}}s^{J^{\ss}}_j}\bbbra{\frac{\prod_{j+1\in J^{\nss},j\notin J^{\nss}}(-1)\bigbra{s^{J^{\ss}}_j+1}}{\prod_{j+1\notin J^{\nss},j\in J^{\nss}}\bigbra{s^{J^{\ss}}_j+1}}}.
    \end{equation}
    By the definition of $c'(J)$ and combining \eqref{Constants Eq Prop c'J simple 1}, \eqref{Constants Eq Prop c'J simple 2} and \eqref{Constants Eq Prop c'J simple 3}, we deduce that $c'(J)=(-1)^{d}$ with
    \begin{align*}
        d&=
    \begin{aligned}[t]
        &\bigbra{f-1+\abs{J\cap(J-1)^{\nss}}}+\bbra{\un{t}^{J^s}+\ssum_{j+1\in J^{\nss}}\bra{s^{J^{\ss}}_j+\delta_{j\in J^{\nss}}}}\\
        &\hspace{1.5cm}+\bbra{\un{s}^{J^{\ss}}+\un{1}}+\bbra{1+\un{t}^J+\ssum_{j+1\in J^{\nss}}s^{J^{\ss}}_j+\#\set{j:j+1\in J^{\nss},j\notin J^{\nss}}}
    \end{aligned}\\
        &\equiv\un{t}^J+\un{t}^{J^s}+\un{s}^{J^{\ss}}+\#\set{j:j+1\in J^{\nss},j\notin J^{\nss}}\\
        &\equiv\#\set{j:j+1\notin J^{\nss},j\in J^{\nss}}=\#\set{j:j\in J^{\nss},j+1\notin J}=A(J)\quad\mod~2,
    \end{align*}
    where the last congruence follows from Lemma \ref{Constants Lem t+t+s} and the last equality follows from the definition of $A(J)$ using $(J-1)^{\ss}=J^{\ss}$. This completes the proof.
\end{proof}

\begin{remark}
    When $(J-1)^{\ss}=J^{\ss}$, we are in the situation of \cite{Hu16}, and the constant $g(J)=U_p(J)c(J)$ (hence the constant $\gamma(J)=U_p(J)c'(J)$) can be computed by \cite[Thm.~4.7]{Hu16}. Here we remark that the term $(-1)^{e(\tau)(r_0,\ldots,r_{f-1})}$ is missing in the formula of \cite[Thm.~4.5]{Hu16} and \cite[Thm.~4.7]{Hu16}. 
\end{remark}

\hspace{\fill}

\subsection{The case \texorpdfstring{$(J-1)^{\ss}\neq J^{\ss}$}{.}}

\begin{lemma}\label{Constants Lem pr DL}
    Let $\emptyset\neq J\subseteq J_{\rhobar}$ (which implies $J\neq J^*$ and $(J-1)^{\ss}\neq J^{\ss}$). Then we have
    \begin{equation*}
        S_{i(\chi_J^s)}S_0\varphi^{\chi_J}=\wt{c}(\chi_J)S^+_{i^+\bra{\chi_J}}\varphi^{\chi_J}\quad\text{in}~\theta\cc(\chi_J),
    \end{equation*}
    where $i^+\bra{\chi_J}\eqdef q-1-i(\chi_J^s)$ and $0\neq\wt{c}(\chi_J)\in\OK$ has leading term $c(\chi_J)\eqdef J\bigbra{i(\chi_J^s),-\un{s}^J}$.
\end{lemma}

\begin{proof}
    The proof is the same as \cite[Lemma~5.11]{DL21} and \cite[Lemma~5.28]{DL21}.
\end{proof}

\begin{lemma}\label{Constants Lem pr new}
    Let $J\subseteq\cJ$ such that $J\nsubseteq J_{\rhobar}$, $J\neq J^*$ and $(J-1)^{\ss}\neq J^{\ss}$. Then we have
    \begin{equation}\label{Constants Eq Lem pr new statement}
        S_{i(\chi_J^s)}S_0\varphi^{\chi_J}=\wt{c}(\chi_J)S^+_{i^+\bra{\chi_J}}S_{i(\chi_J)}\varphi^{\chi_J}\quad\text{in}~\theta(\chi_J),
    \end{equation}
    where $i^+\bra{\chi_J}\eqdef i(\chi_J)-i(\chi_J^s)+\un{s}^J$, and $0\neq\wt{c}(\chi_J)\in K$ has leading term
    \begin{equation*}
        c(\chi_J)\eqdef(-1)^{\un{t}^J}\frac{J\bigbra{i(\chi_J^s),-\un{s}^J}}{J\bigbra{i^+(\chi_J),-i(\chi_J)-\un{s}^J}}.
    \end{equation*}
\end{lemma}

\begin{proof}
    The assumption $J\neq J^*$ implies $R\chi_J^s\neq\chi_J^s$. The assumption $J\nsubseteq J_{\rhobar}$ implies $R\chi_J^s\neq\chi_J$. The assumption $(J-1)^{\ss}\neq J^{\ss}$ implies $R\chi_J\neq R\chi_{J^s}$. Then as in the proof of \cite[Lemma~5.11]{DL21} (using Lemma \ref{Constants Lem S+v} and \cite[Lemma~2.11]{DL21}), the equality \eqref{Constants Eq Lem pr new statement} holds for some $0\neq\wt{c}(\chi_J)\in K$.

    Next we compute $\wt{c}(\chi_J)\in K$. The assumption $J\nsubseteq J_{\rhobar}$ implies $J^{\delta}\neq\emptyset$, hence $(q-1)\nmid i(\chi_J^s)$. The assumption $J\neq J^*$ implies $J^{\delta}\neq\cJ$, hence $(q-1)\nmid\bigbra{i(\chi_J^s)-\un{s}^J}$ by \eqref{Constants Eq Lem Jab 2}. Then by Lemma \ref{Constants Lem SSv} we have 
    \begin{equation}\label{Constants Eq Lem pr new 1}
        S_{i(\chi_J^s)}S_0\varphi^{\chi_J}=\wt{J}_1S_{i(\chi_J^s)-\un{s}^J}\varphi^{\chi_J},
    \end{equation}
    where $0\neq\wt{J}_1\in\OK$ has leading term $J_1\eqdef(-1)^{\un{t}^J}J\bigbra{i(\chi_J^s),-\un{s}^J}$.

    We choose $0<z<q-1$ such that none of the following numbers
    \begin{equation*}
        z-i(\chi_J^s)+\un{s}^J,~z-i(\chi_J^s),~z+i^+\bra{\chi_J},~z+i^+\bra{\chi_J}-i(\chi_J)
    \end{equation*}
    are multiples of $(q-1)$, which is possible since $q\geq7$. On one hand, by Lemma \ref{Constants Lem SSv} we have 
    \begin{equation}\label{Constants Eq Lem pr new 2}
        S_zS_{i(\chi_J^s)-\un{s}^J}\varphi^{\chi_J}=\wt{J}_2S_{z-i(\chi_J^s)}\varphi^{\chi_J}, 
    \end{equation}
    where $0\neq\wt{J}_2\in\OK$ has leading term $J_2\eqdef(-1)^{\un{t}^J}J\bigbra{z,-i(\chi_J^s)}$. On the other hand, by Lemma \ref{Constants Lem S+S+S+}(i) and Lemma \ref{Constants Lem SSv} we have
    \begin{equation}\label{Constants Eq Lem pr new 3}
        S_zS^+_{i^+\bra{\chi_J}}S_{i(\chi_J)}\varphi^{\chi_J}=\wt{J}_3S_{z+i^+\bra{\chi_J}}S_{i(\chi_J)}\varphi^{\chi_J}=\wt{J}_3\wt{J}_4S_{z-i(\chi_J^s)}\varphi^{\chi_J}, 
    \end{equation}
    where $0\neq\wt{J}_3\in\OK$ has leading term $J_3\eqdef J\bigbra{z,i^+\bra{\chi_J}}$ and $0\neq\wt{J}_4\in\OK$ has leading term $J_4\eqdef(-1)^{\un{t}^J}J\bigbra{z+i^+\bra{\chi_J},-i(\chi_J)-\un{s}^J}$. Combining \eqref{Constants Eq Lem pr new statement}, \eqref{Constants Eq Lem pr new 1}, \eqref{Constants Eq Lem pr new 2} and \eqref{Constants Eq Lem pr new 3}, we deduce that $0\neq\wt{c}(\chi_J)=\bra{\wt{J}_1\wt{J}_2}/\bra{\wt{J}_3\wt{J}_4}\in K$ has leading term
    \begin{equation*}
        c(\chi_J)=\frac{J_1J_2}{J_3J_4}=\frac{(-1)^{\un{t}^J}J\bigbra{i(\chi_J^s),-\un{s}^J}J\bigbra{z,-i(\chi_J^s)}}{J\bigbra{z,i^+\bra{\chi_J}}J\bigbra{z+i^+\bra{\chi_J},-i(\chi_J)-\un{s}^J}}=(-1)^{\un{t}^J}\frac{J\bigbra{i(\chi_J^s),-\un{s}^J}}{J\bigbra{i^+(\chi_J),-i(\chi_J)-\un{s}^J}},
    \end{equation*}
    where the last equality follows from the formula (applied with $a=z,~b=i^+\bra{\chi_J},~c=-i(\chi_J)-\un{s}^J$)
    \begin{equation*}
        J(a,b)J(a+b,c)=J(a,b+c)J(b,c)\quad\text{for}~a,b,c\in\ZZ,
    \end{equation*}
     which can be deduced from the explicit formula \eqref{Constants Eq Jab}.
\end{proof}

For $J\subseteq\cJ$ such that $J\neq J^*$ and $(J-1)^{\ss}\neq J^{\ss}$, we write
\begin{equation*}
\begin{aligned}
    i^+(J)&\eqdef\ssum_{i=0}^{\ell(J)-1}i^+\bigbra{\chi_{\delta_{\ss}^i(J)}}\in\ZZ;\\
    \beta(J)&\eqdef J\bbra{i^+\bra{\chi_J},i^+\bigbra{\chi_{\delta_{\ss}(J)}},\ldots,i^+\bigbra{\chi_{\delta_{\ss}^{\ell(J)-1}(J)}}}\in\FF\x,
\end{aligned}
\end{equation*}
where each $i^+\bra{\chi_{J'}}$ is defined in either Lemma \ref{Constants Lem pr DL} or Lemma \ref{Constants Lem pr new}.

\begin{proposition}\label{Constants Prop cJ}
    Let $J\subseteq\cJ$ such that $J\nsubseteq J_{\rhobar}$, $J\neq J^*$ and $(J-1)^{\ss}\neq J^{\ss}$. Then we have 
    \begin{equation*}
        c(J)=\frac{\beta(J)}{\beta(J^{\ss})}\frac{\sprod_{i=0}^{\ell(J)-1}c\bigbra{\chi_{\delta_{\ss}^i(J)}}}{\sprod_{i=0}^{\ell(J^{\ss})-1}c\bigbra{\chi_{\delta_{\ss}^i(J^{\ss})}}},
    \end{equation*}
    where $c(J)$ is defined in Example \ref{Constants Ex formula gammaJ} and each $c(\chi_{J'})$ is defined in either Lemma \ref{Constants Lem pr DL} or Lemma \ref{Constants Lem pr new}.
\end{proposition}

\begin{proof}
    Recall from Example \ref{Constants Ex formula gammaJ} that $c(J)=c(\psi,\psi')$ with $\psi_i=\chi_{\delta_{\ss}^{n-i}(J)}^s$ for $0\leq i\leq n\eqdef\ell(J)-1$ and $\psi_i'=\chi_{\delta_{\ss}^{m-i}(J^{\ss})}^s$ for $0\leq i\leq m\eqdef\ell(J^{\ss})-1$. Recall from \S\ref{Constants Sec Diamond diagram} the maps $\op{pr}_{\chi}$, $\alpha_i$ and $\alpha'_i$. If $x\in\theta_0^{R\psi_n^s}$ such that $\op{pr}_{\psi_n^s}\alpha_n(x)=\varphi^{R\psi_n^s}\in\sigma(R\psi_n^s)$, then the proof of \cite[Prop.~5.14]{DL21} (using Lemma \ref{Constants Lem pr DL}, Lemma \ref{Constants Lem pr new} and the cuspidality of $\theta_0$) and the argument that follows show that 
    \begin{equation}\label{Constants Eq Prop cJ 1}
        \op{pr}_{\psi_0}\alpha_0\bbra{\sprod_{i=n}^0\bigbbbra{c\bra{\chi_{\delta_{\ss}^{i}(J)}}}S^+_{i^+\bra{\chi_{\delta_{\ss}^{i}(J)}}}(x)}=\varphi^{R\psi_0}=\varphi^{\chi_{\emptyset}}\in\sigma(\chi_{\emptyset}).
    \end{equation}
    Similarly, if $x'\in\theta_0^{R\psi_m^{\prime s}}$ such that $\op{pr}_{\psi_m^{\prime s}}\alpha_m'(x')=\varphi^{R\psi_m^{\prime s}}\in\sigma(R\psi_m^{\prime s})$, then we have 
    \begin{equation}\label{Constants Eq Prop cJ 2}
        \op{pr}_{\psi_0'}\alpha_0'\bbra{\sprod_{i=m}^0\bigbbbra{c\bra{\chi_{\delta_{\ss}^{i}(J^{\ss})}}}S^+_{i^+\bra{\chi_{\delta_{\ss}^{i}(J^{\ss})}}}(x')}=\varphi^{R\psi_0'}=\varphi^{\chi_{\emptyset}}\in\sigma(\chi_{\emptyset}).
    \end{equation}

    We compare $H$-eigencharacters of \eqref{Constants Eq Prop cJ 1} and \eqref{Constants Eq Prop cJ 2} using Lemma \ref{Constants Lem S+v}. Since both $x$ and $x'$ have $H$-eigencharacters $R\psi_n^s=R\psi_m^{\prime s}=\chi_{J^{\ss}}$,  we deduce that $i^+(J)\equiv i^+(J^{\ss})~\mod~(q-1)$. Since $R\chi_{\delta_{\ss}^i(J)}\neq\chi_{\emptyset}$ for $0\leq i\leq n-1$ and $R\chi_{\delta_{\ss}^i(J^{\ss})}\neq\chi_{\emptyset}$ for $0\leq i\leq m-1$, we deduce that $(q-1)\nmid i^+\bra{\delta_{\ss}^i(J)}$ for $1\leq i\leq n$ and $(q-1)\nmid i^+\bra{\delta_{\ss}^i(J^{\ss})}$ for $1\leq i\leq m$. Moreover, since $\theta_0$ is a cuspidal type, we have $S^+_0(x)=S^+_0(x')=0$ by \cite[Lemma~2.9]{DL21}. Hence by Lemma \ref{Constants Lem S+S+S+}(ii) we have
    \begin{equation}\label{Constants Eq Prop cJ 3}
    \begin{aligned}
        \bbra{\sprod_{i=n}^0S^+_{i^+\bra{\chi_{\delta_{\ss}^{i}(J)}}}}(x)&=\wt{\beta}(J)S^+_{i^+\bra{J}}(x);\\
        \bbra{\sprod_{i=m}^0S^+_{i^+\bra{\chi_{\delta_{\ss}^{i}(J^{\ss})}}}}(x)&=\wt{\beta}(J^{\ss})S^+_{i^+\bra{J^{\ss}}}(x),
    \end{aligned}
    \end{equation}
    where $0\neq\wt{\beta}(J)\in\OK$ (resp.\,$0\neq\wt{\beta}(J^{\ss})\in\OK$) has leading term $\beta(J)$ (resp.\,$\beta(J^{\ss})$). Combining \eqref{Constants Eq Prop cJ 1}, \eqref{Constants Eq Prop cJ 2}, \eqref{Constants Eq Prop cJ 3} and the argument following the proof of \cite[Prop.~5.14]{DL21} we deduce that
    \begin{equation*}
        c(J)=\frac{\beta(J)}{\beta(J^{\ss})}\frac{\sprod_{i=0}^{n}c\bigbra{\chi_{\delta_{\ss}^i(J)}}}{\sprod_{i=0}^{m}c\bigbra{\chi_{\delta_{\ss}^i(J^{\ss})}}},
    \end{equation*}
    which completes the proof.
\end{proof}

\begin{proposition}\label{Constants Prop c'J new}
    Let $J\subseteq\cJ$ such that $J\nsubseteq J_{\rhobar}$, $J\neq J^*$ and $(J-1)^{\ss}\neq J^{\ss}$. Then we have 
    \begin{equation*}
        c'(J)=(-1)^{A(J)}.
    \end{equation*}
\end{proposition}

\begin{proof} 
By the definition $c'(J)$ and Proposition \ref{Constants Prop cJ} we have
\begin{equation}\label{Constants Eq Prop c'J new 1}
    c'(J)=(-1)^{f-1+\abs{J\cap(J-1)^{\nss}}}\mu_{J^s,J}\frac{P_2(J)}{P_1(J)}\frac{\beta(J)}{\beta(J^{\ss})}\frac{\sprod_{i=0}^{\ell(J)-1}c\bigbra{\chi_{\delta_{\ss}^i(J)}}}{\sprod_{i=0}^{\ell(J^{\ss})-1}c\bigbra{\chi_{\delta_{\ss}^i(J^{\ss})}}}.
\end{equation}
Recall from the proof of Proposition \ref{Constants Prop cJ} that $i^+(J)\equiv i^+(J^{\ss})~\mod~(q-1)$, hence by \eqref{Constants Eq Ja1an} we have
\begin{equation}\label{Constants Eq Prop c'J new 2}
    \frac{\beta(J)}{\beta(J^{\ss})}=\frac{(-1)^{u(J)}}{(-1)^{u(J^{\ss})}}\frac{\bbra{(-1)^{f-1}\sprod_{j=0}^{f-1}\bigbra{i^+\bra{\chi_J}_j}!}\sprod_{i=1}^{\ell(J)-1}\!\bbra{(-1)^{f-1}\sprod_{j=0}^{f-1}\bigbra{i^+\bra{\chi_{\delta_{\ss}^i(J)}}_j}!}}{\sprod_{i=0}^{\ell(J^{\ss})-1}\bbra{(-1)^{f-1}\sprod_{j=0}^{f-1}\!\bigbra{i^+\bra{\chi_{\delta_{\ss}^i(J^{\ss})}}_j}!}},
\end{equation}
where $u(J)\eqdef u\bigbra{i^+\bra{\chi_J},i^+\bra{\chi_{\delta_{\ss}(J)}},\ldots,i^+\bra{\chi_{\delta_{\ss}^{\ell(J)-1}(J)}}}\in\ZZ$, and similar for $u(J^{\ss})$. Moveover, by Lemma \ref{Constants Lem pr DL}, Lemma \ref{Constants Lem pr new} and \eqref{Constants Eq Jab} we have
\begin{equation}\label{Constants Eq Prop c'J new 3}
\begin{aligned}
    \sprod_{j=0}^{f-1}\bigbra{i^+\bra{\chi_{J'}}_j}!&=\sprod_{j=0}^{f-1}\bigbra{p-1-i(\chi_{J'}^s)_j}!\hspace{1cm}\text{for}~J'\subseteq J_{\rhobar};\\
    \frac{\sprod_{j=0}^{f-1}\bigbra{i^+\bra{\chi_J}_j}!}{J\bigbra{i^+\bra{\chi_J},-i(\chi_J)-\un{s}^J}}&=(-1)^{f-1+u'(J)}\frac{\sprod_{j=0}^{f-1}\bigbra{p-1-i(\chi_J^s)_j}!}{\sprod_{j=0}^{f-1}\bigbra{(-i(\chi_J)-\un{s}^J)_j}!},
\end{aligned}
\end{equation}
where $u'(J)\eqdef u\bigbra{i^+\bra{\chi_J},-i(\chi_J)-\un{s}^J}$. Combining \eqref{Constants Eq Prop c'J new 1}, \eqref{Constants Eq Prop c'J new 2}, \eqref{Constants Eq Prop c'J new 3} and the definition of each $c(\chi_J)$ in Lemma \ref{Constants Lem pr DL} and Lemma \ref{Constants Lem pr new}, we deduce that 
\begin{equation*}
    c'(J)=(-1)^{U(J)}\alpha'(J)\alpha(J),
\end{equation*}
where
\begin{equation}\label{Constants Eq Prop c'J new 4}
\begin{aligned}
    U(J)&\eqdef u(J)-u(J^{\ss})-u'(J)\in\ZZ;\\
    \alpha'(J)&\eqdef(-1)^{|J\cap(J-1)^{\nss}|+\un{t}^J}\mu_{J^s,J}P_2(\chi_J)/\bbra{\sprod_{j=0}^{f-1}\bigbra{(-i(\chi_J)-\un{s}^J)_j}!}\in\FF\x;\\
    \alpha(J)&\eqdef\bbra{\sprod_{i=0}^{\ell(J)-1}\alpha\bigbra{\chi_{\delta_{\ss}^i(J)}}}/\bbra{\sprod_{i=0}^{\ell(J^{\ss})-1}\alpha\bigbra{\chi_{\delta_{\ss}^i(J^{\ss})}}}\in\FF\x,
\end{aligned}
\end{equation}
with (for each $J'\neq J^*$)
\begin{equation*}
    \alpha(\chi_{J'})\eqdef J\bigbra{i(\chi_{J'}^s),-\un{s}^{J'}}\bbra{(-1)^{f-1}\sprod_{j=0}^{f-1}\bigbra{p-1-i(\chi_{J'}^s)}!}/P_1(\chi_{J'})\in\FF\x.
\end{equation*}
Then the proposition follows from an explicit computation of the constants $U(J)$, $\alpha'(J)$ and $\alpha(J)$, which is given in Lemma \ref{Constants Lem Prop c'J new} below.
\end{proof}

\subsection{Explicit computations}

We prove Lemma \ref{Constants Lem Prop c'J new} below, which will finish the proof of Proposition \ref{Constants Prop c'J new}. To state the result, for $J\subseteq\cJ$ and $j\in\cJ$ we define (in $\FF\x$)
\begin{equation*}
    \alpha'(J)_j\eqdef
\begin{cases}
    1&\text{if}~j\notin J^{\nss},~j+1\notin J^{\nss}\\
    r_j+1&\text{if}~j\in J^{\nss},~j+1\notin J\\
    p-1-r_j&\text{if}~j\in J^{\nss},~j+1\in J^{\ss}\\
    -\bigbra{r_j!(r_j+1)!}\inv&\text{if}~j\notin J,~j+1\in J^{\nss}\\
    -\bigbra{(r_j+1)!(r_j+2)!}\inv&\text{if}~j\in J^{\ss},~j+1\in J^{\nss}\\
    (r_j!)^{-2}&\text{if}~j\in J^{\nss},~j+1\in J^{\nss}.
\end{cases}
\end{equation*}
    
\begin{lemma}\label{Constants Lem Prop c'J new}
    Let $J\subseteq\cJ$ such that $J\nsubseteq J_{\rhobar}$, $J\neq J^*$ and $(J-1)^{\ss}\neq J^{\ss}$. Let $U(J)\in\ZZ$, $\alpha'(J),\alpha(J)\in\FF\x$ be as in \eqref{Constants Eq Prop c'J new 4}. Then we have
\begin{enumerate}
    \item 
    $U(J)=A^{\ss}(J)$ (see \eqref{Constants Eq AssJ} for $A^{\ss}(J)$);
    \item 
    $\alpha'(J)=(-1)^{\abs{\bra{\partial J}^{\nss}}}\sprod_{j=0}^{f-1}\alpha'(J)_j$;
    \item
    $\alpha(J){\sprod_{j=0}^{f-1}\alpha'(J)_j}=1$.
\end{enumerate}
\end{lemma}

\hspace{\fill}

In the rest of this subsection, we prove Lemma \ref{Constants Lem Prop c'J new}. We start with some more notation that are needed in the proof. For $J\subseteq\cJ$ and $j\in\cJ$, we define
\begin{equation*}
\begin{aligned}
    I(J)^{1}_j&\eqdef\sset{i\geq0:j+1\in\delta_{\ss}^i(J),~j+1\notin\delta_{\ss}^{i+1}(J)};\\
    I(J)^{2}_j&\eqdef\sset{i\geq0:j+1\notin\delta_{\ss}^i(J),~j+1\in\delta_{\ss}^{i+1}(J)};\\
    I(J)^{3}_j&\eqdef\sset{i\geq0:j+1\in\delta_{\ss}^i(J),~j+1\in\delta_{\ss}^{i+1}(J)};\\
    I(J)^{4}_j&\eqdef\sset{i\geq0:j+1\notin\delta_{\ss}^i(J),~j+1\notin\delta_{\ss}^{i+1}(J)}.
\end{aligned}
\end{equation*}
Since $j\in\delta_{\ss}^{i+1}(J)$ if and only if $(j+1\in\delta_{\ss}^{i}(J)~\text{and}~j\in J_{\rhobar})$, by \eqref{Constants Eq sJ} we have
\begin{equation}\label{Constants Eq s for deltass}
    s^{\delta_{\ss}^{i+1}(J)}_j=
\begin{cases}
    r_j+\delta_{j\in J_{\rhobar}}&\text{if}~i\in I(J)^{1}_j\\
    p-2-r_j&\text{if}~i\in I(J)^{2}_j\\
    p-2-r_j-\delta_{j\in J_{\rhobar}}&\text{if}~i\in I(J)^{3}_j\\
    r_j&\text{if}~i\in I(J)^{4}_j.
\end{cases}
\end{equation}
Then we define
\begin{equation*}
\begin{aligned}
    I(J)^{t+4}_j&\eqdef I(J)^t_j\cap\bigset{i\geq0:j\notin\bra{\delta_{\ss}^i(J)}^{\delta}}&\text{for}~t\in\set{1,2};\\
    I(J)^{t+4}_j&\eqdef I(J)^t_j\cap\bigset{i\geq0:j\in\bra{\delta_{\ss}^i(J)}^{\delta}}&\text{for}~t\in\set{3,4}.
\end{aligned}    
\end{equation*}
We also define
\begin{equation*}
\begin{aligned}
    i(J)^{t}_j&\eqdef\bigabs{I(J)^{t}_j}-\bigabs{I(J^{\ss})^{t}_j}&\text{for}~t\in\set{1,2};\\
    i(J)^{t}_j&\eqdef\bbra{\bigabs{I(J)^{t+2}_j}-\bigabs{I(J^{\ss})^{t+2}_j}}-\bbra{\bigabs{I(J)^{t+4}_j}-\bigabs{I(J^{\ss})^{t+4}_j}}&\text{for}~t\in\set{3,4}.
\end{aligned}    
\end{equation*}
Finally, for $t\in\set{1,2,5,6,7,8}$ we define
\begin{equation*}
    I\cc(J)^t_j\eqdef I(J)^t_j\cap\set{0}.
\end{equation*}
If moreover $j+1\in J_{\rhobar}$, we let $k\geq0$ be such that $j+i+1\in J_{\rhobar}$ for $0\leq i\leq k$ and $j+k+2\notin J_{\rhobar}$. Then we define
\begin{equation*}
    I^{\geq}(J)^t_j\eqdef\set{i\geq0:i+k\in I(J)^t_j}.
\end{equation*}

\hspace{\fill}

\proof[Proof of Lemma \ref{Constants Lem Prop c'J new}(i)] 

Recall that $U(J)=u(J)-u(J^{\ss})-u'(J)$ with $u(J)$ and $u(J^{\ss})$ defined in \eqref{Constants Eq Prop c'J new 2} and $u'(J)$ defined in \eqref{Constants Eq Prop c'J new 3}. By \eqref{Constants Eq Jab} we have
\begin{equation}\label{Constants Eq Lem Prop 1-1}
\begin{aligned}
    u'(J)&=(p-1)\inv\ssum_{j=0}^{f-1}\Bigbra{p\!-\!1\!-\!\bigbra{i^+\bra{\chi_J}_j\!+\!\bra{-i(\chi_J)\!-\!\un{s}^J}_j\!-\!\bra{-i(\chi_J^s)}_j}\!}\\
    &=(p-1)\inv\ssum_{j=0}^{f-1}\Bigbra{\!\bigbra{p\!-\!1\!-\!i^+\bra{\chi_J}_j}\!-\!\bigbra{p\!-\!1\!-\!\bra{i(\chi_J)\!+\!\un{s}^J}_j}\!+\!\bigbra{p\!-\!1\!-\!i(\chi_J^s)_j}\!}.
    \end{aligned}
\end{equation}
By Lemma \ref{Constants Lem pr DL} we have $p-1-i^+(\chi_{J'})_j=i(\chi_{J'}^s)_j$ for $\emptyset\neq J'\subseteq J_{\rhobar}$ and $j\in\cJ$. Hence by \eqref{Constants Eq Ja1an} we have
\begin{equation}\label{Constants Eq Lem Prop 1-2}
\begin{aligned}
    u(J)&=(p-1)\inv\ssum_{j=0}^{f-1}\bbra{\!\bigbra{p\!-\!1\!-\!i^+\bra{\chi_J}_j}\!+\!\bigbra{\ssum_{i=1}^{\ell(J)-1}i\bra{\chi_{\delta_{\ss}^i(J)}^s}_j}\!-\!\bigbra{p\!-\!1\!-\!i^+(J)_j}\!};\\
    u(J^{\ss})&=(p-1)\inv\ssum_{j=0}^{f-1}\bbra{\!\bigbra{\ssum_{i=0}^{\ell(J^{\ss})-1}i\bra{\chi_{\delta_{\ss}^i(J^{\ss})}^s}_j}\!-\!\bigbra{p\!-\!1\!-\!i^+(J^{\ss})_j}\!}.
\end{aligned}
\end{equation}
Moreover, from the proof of Proposition \ref{Constants Prop cJ} we have $i^+(J)\equiv i^+(J^{\ss})~\mod~(q-1)$. Combining \eqref{Constants Eq Lem Prop 1-1} and \eqref{Constants Eq Lem Prop 1-2} we deduce that
\begin{equation*}
    U(J)=u(J)-u(J^{\ss})-u'(J)=(p-1)\inv\ssum_{j=0}^{f-1}U(J)_j,
\end{equation*}
where (for each $j\in\cJ$)
\begin{equation}\label{Constants Eq Lem Prop 1-4}
    U(J)_j\eqdef\ssum_{i=0}^{\ell(J)-1}i\bigbra{\chi_{\delta_{\ss}^i(J)}^s}_j-\ssum_{i=0}^{\ell(J^{\ss})-1}i\bigbra{\chi_{\delta_{\ss}^i(J^{\ss})}^s}_j-\bigbra{i(\chi_J)+\un{s}^J}_j\in\ZZ.
\end{equation}
By definition and using \eqref{Constants Eq s for deltass}, we have
\begin{equation}\label{Constants Eq Lem Prop 1-5}
\begin{aligned}
    \ssum_{i=0}^{\ell(J)-1}i\bigbra{\chi_{\delta_{\ss}^i(J)}^s}_j&=\ssum_{i\geq0}\delta_{j\in\bra{\delta_{\ss}^i(J)}^{\delta}}\bigbra{p\!-\!1\!-\!s^{\delta_{\ss}^{i+1}(J)}_j}=\ssum_{i\in I(J)^1_j\sqcup I(J)^2_j}\bigbra{p\!-\!1\!-\!s^{\delta_{\ss}^{i+1}(J)}_j}\\
    &=\babs{I(J)^1_j}\bigbra{p-1-r_j-\delta_{j\in J_{\rhobar}}}+\babs{I(J)^2_j}\bra{r_j+1}.
\end{aligned}
\end{equation}
Similarly, we have
\begin{equation}\label{Constants Eq Lem Prop 1-6}
    \ssum_{i=0}^{\ell(J^{\ss})-1}i\bigbra{\chi_{\delta_{\ss}^i(J^{\ss})}^s}_j=\babs{I(J^{\ss})^1_j}\bigbra{p-1-r_j-\delta_{j\in J_{\rhobar}}}+\babs{I(J^{\ss})^2_j}\bra{r_j+1}.
\end{equation}
Moreover, by Lemma \ref{Constants Lem ichiJ+sJ} and \eqref{Constants Eq sJ} we have
\begin{equation}\label{Constants Eq Lem Prop 1-7}
    \bigbra{i(\chi_J)+\un{s}^J}_j=\delta_{j+1\in J^{\nss}}\bigbra{p-1-s^{J^{\ss}}_j}=\delta_{j+1\in J^{\nss}}\bigbra{p-1-r_j-\delta_{j\in J^{\ss}}}.
\end{equation}
Combining \eqref{Constants Eq Lem Prop 1-4}, \eqref{Constants Eq Lem Prop 1-5}, \eqref{Constants Eq Lem Prop 1-6} and \eqref{Constants Eq Lem Prop 1-7} we deduce that
\begin{equation}\label{Constants Eq Lem Prop 1-8}
    U(J)_j=i(J)^1_j\bigbra{p-1-r_j-\delta_{j\in J_{\rhobar}}}+i(J)^2_j\bra{r_j+1}-\delta_{j+1\in J^{\nss}}\bigbra{p-1-r_j-\delta_{j\in J^{\ss}}}.
\end{equation}
To compute each $U(J)_j$ explicitly, we separate the following cases.

\hspace{\fill}

\noindent\textbf{Case 1.} $j+1\in J_{\rhobar}$.

Let $k\geq0$ such that $j+i+1\in J_{\rhobar}$ for $0\leq i\leq k$ and $j+k+2\notin J_{\rhobar}$. By \eqref{Constants Eq j in deltass} we have
\begin{equation}\label{Constants Eq Lem Prop 1-9}
\begin{aligned}
    j+1\in\delta_{\ss}^i(J)&\iff\bra{0\leq i\leq k+1~\text{and}~j+i+1\in J};\\
    j+1\in\delta_{\ss}^i(J^{\ss})&\iff\bra{0\leq i\leq k~\text{and}~j+i+1\in J}.
\end{aligned}
\end{equation}
In particular, for $t\in\set{1,2}$ we have $i\in I(J)^{t}_j\iff i\in I(J^{\ss})^{t}_j$ for $0\leq i\leq k-1$, hence
\begin{equation}\label{Constants Eq Lem Prop 1-10}
    \bigabs{I(J)^t_j}-\bigabs{I(J^{\ss})^t_j}=\bigabs{I^{\geq}(J)^t_j}-\bigabs{I^{\geq}(J^{\ss})^t_j}.
\end{equation}
We denote $\op{ch}^1_J\eqdef\bigbra{\delta_{j+k+1\in J},\delta_{j+k+2\in J}}\in\set{0,1}^2$. Combining \eqref{Constants Eq Lem Prop 1-8}, \eqref{Constants Eq Lem Prop 1-9}, \eqref{Constants Eq Lem Prop 1-10} and a case-by-case examination we get the following table.
\begin{figure}[H]
    \centering
\begin{tabular}{|c|c|c|c|c|}
    \hline
    $\op{ch}^1_J$&$I^{\geq}(J)^{1,2}_j$&$I^{\geq}(J^{\ss})^{1,2}_j$&$i(J)^{1,2}_j$&$U(J)_j$\\
    \hline
    $(1,1)$&$\set{1},\emptyset$&$\set{0},\emptyset$&$0,0$&$0$\\
    \hline
    $(1,0)$&$\set{0},\emptyset$&$\set{0},\emptyset$&$0,0$&$0$\\
    \hline
    $(0,1)$&$\set{1},\set{0}$&$\emptyset,\emptyset$&$1,1$&$p-\delta_{j\in J_{\rhobar}}$\\
    \hline
    $(0,0)$&$\emptyset,\emptyset$&$\emptyset,\emptyset$&$0,0$&$0$\\
    \hline
\end{tabular}
\end{figure}

\hspace{\fill}

\noindent\textbf{Case 2.} $j+1\notin J_{\rhobar}$.

In this case, we have $j+1\notin\delta_{\ss}^i(J)$ for $i\geq1$ and $j+1\notin\delta_{\ss}^i(J^{\ss})$ for $i\geq0$. Hence $I(J)^1_j=\set{0}$ if $j+1\in J$, $I(J)^1_j=\emptyset$ if $j+1\notin J$, and $I(J)^2_j=I(J^{\ss})^1_j=I(J^{\ss})^2_j=\emptyset$. By \eqref{Constants Eq Lem Prop 1-8} we have
\begin{equation*}
\begin{aligned}
    U(J)_j&=\delta_{j+1\in J}\bigbra{p-1-r_j-\delta_{j\in J_{\rhobar}}}-\delta_{j+1\in J^{\nss}}\bigbra{p-1-r_j-\delta_{j\in J^{\ss}}}\\
    &=-\delta_{j+1\in J}\bigbra{\delta_{j\in J_{\rhobar}}-\delta_{j\in J^{\ss}}}=-\delta_{j+1\in J^{\nss},j\in J_{\rhobar}\setminus J}.
\end{aligned}
\end{equation*}

\vspace{0.2cm}

As in \S\ref{Constants Sec Kisin}, we decomposition $J_{\rhobar}$ into a disjoint union of intervals (in $\ZZ/f\ZZ$) not adjacent to each other $J_{\rhobar}=J_1\sqcup\ldots\sqcup J_t$, and for each $1\leq i\leq t$ we write $J_i=\set{j_i,j_i+1,\ldots,j_i+k_i}$ with $j_i\in\cJ$ and $k_i\geq0$. Combining Case 1 and Case 2, we get
\begin{equation*}
\begin{aligned}
    U(J)&=(p-1)\inv\ssum_{j=0}^{f-1}U(J)_j=(p-1)\inv\ssum_{i=1}^{t}\bigbra{\ssum_{j=j_i}^{j_i+k_i+1}U(J)_j}\\
    &=(p-1)\inv\ssum_{i=1}^{t}\Bigbra{\delta_{j_i+k_i\in\partial(J^c)}\bigbra{p+k_i(p-1)-1}}=\ssum_{i=1}^{t}\bigbra{\delta_{j_i+k_i\in\partial(J^c)}(k_i+1)},
\end{aligned}
\end{equation*}
which equals $A^{\ss}(J)$ by \eqref{Constants Eq AssJ}.\qed

\hspace{\fill}

\proof[Proof of Lemma \ref{Constants Lem Prop c'J new}(ii)]

As in \eqref{Constants Eq Prop c'J simple 1} we have
\begin{equation}\label{Constants Eq Lem Prop 2-1}
    \mu_{J^s,J}=(-1)^{\un{t}^{J^s}+\sum_{j+1\in J^{\nss}}\bra{s^{J^{\ss}}_j+\delta_{j\in J^{\nss}}}}\bbbra{\frac{\sprod_{j+1\notin J^{\nss}}\bra{s^{J^{\ss}}_j+\delta_{j\in J^{\nss}}}!}{\sprod_{j+1\in J^{\nss}}\bra{s^{J^{\ss}}_j+\delta_{j\notin J^{\nss}}}!}}.
\end{equation}
By \eqref{Constants Eq P2} and using \eqref{Constants Eq factorial}, we have
\begin{equation}\label{Constants Eq Lem Prop 2-2}
    P_2(J)=P_2(\chi_J)=(-1)^{\sum_{j+1\notin J^{\nss}}\bra{s^{J^{\ss}}_j+1}}\bbra{\sprod_{j+1\notin J^{\nss}}\bigbra{s^{J^{\ss}}_j}!}\inv.
\end{equation}
By Lemma \ref{Constants Lem ichiJ+sJ} and using $(p-1)!\equiv-1~\mod~p$, we have
\begin{equation}\label{Constants Eq Lem Prop 2-3}
    \sprod_{j=0}^{f-1}\bigbra{(-i(\chi_J)-\un{s}^J)_j}!=(-1)^{f-\abs{J^{\nss}}}\sprod_{j+1\in J^{\nss}}\bigbra{s^{J^{\ss}}_j}!
\end{equation}
Combining \eqref{Constants Eq Lem Prop 2-1}, \eqref{Constants Eq Lem Prop 2-2} and \eqref{Constants Eq Lem Prop 2-3} we deduce that
\begin{equation*}
\begin{aligned}
    \alpha'(J)&=(-1)^{|J\cap(J-1)^{\nss}|+\un{t}^J}\mu_{J^s,J}P_2(J)/\bbra{\sprod_{j=0}^{f-1}\bigbra{(-i(\chi_J)-\un{s}^J)_j}!}\\
    &=(-1)^{d}\bbbra{\frac{\sprod_{j\notin J^{\nss},j+1\notin J^{\nss}}(1)~\sprod_{j\in J^{\nss},j+1\notin J^{\nss}}\bigbra{s^{J^{\ss}}_j+1}}{\sprod_{j\notin J^{\nss},j+1\in J^{\nss}}\bigbra{\!-\!\bra{s^{J^{\ss}}_j}!\bra{s^{J^{\ss}}_j+1}!}~\sprod_{j\in J^{\nss},j+1\in J^{\nss}}\bigbra{\bra{s^{J^{\ss}}_j}!}^2}},
\end{aligned}
\end{equation*}
where (using Lemma \ref{Constants Lem t+t+s})
\begin{equation*}
\begin{aligned}
    d&=\abs{J\cap(J-1)^{\nss}}+\un{t}^J+\un{t}^{J^s}+\un{s}^{J^{\ss}}+\un{1}+f-\abs{J^{\nss}}\\
    &\equiv\abs{J^{\nss}}-\abs{J\cap(J-1)^{\nss}}=\abs{\bra{\partial J}^{\nss}}\quad\mod~2.
\end{aligned}
\end{equation*}
Then a case-by-case examination using \eqref{Constants Eq sJ} shows that $\alpha'(J)=(-1)^{\abs{\bra{\partial J}^{\nss}}}\sprod_{j=0}^{f-1}\alpha'(J)_j$.\qed

\hspace{\fill}

\proof[Proof of Lemma \ref{Constants Lem Prop c'J new}(iii)]

For $J'\neq J^*$, by \eqref{Constants Eq P1} and using \eqref{Constants Eq factorial}, we have 
\begin{equation}\label{Constants Eq Lem Prop 3-1}
\begin{aligned}
    &\sprod_{j=0}^{f-1}\bigbra{p-1-i(\chi_{J'}^s)_j}!=(-1)^{f-\abs{J^{\prime\delta}}}\bbra{\sprod_{j+1\in J^{\prime\delta}}\bigbra{s^{(J'-1)^{\ss}}_j}!};\\
    &P_1(\chi_{J'})=(-1)^{\sum_{j+1\in J^{\prime\delta}}\bigbra{s^{(J'-1)^{\ss}}_j+1}}\bbra{\sprod_{j+1\in J^{\prime\delta}}\bigbra{s^{(J'-1)^{\ss}}_j}!}\inv.
\end{aligned}
\end{equation}
Combining \eqref{Constants Eq Lem Prop 3-1} with Lemma \ref{Constants Lem Jab}, we deduce that
\begin{equation*}
\begin{aligned}
    \alpha(\chi_{J'})&=J\bigbra{i(\chi_{J'}^s),-\un{s}^J}\bbra{(-1)^{f-1}\sprod_{j=0}^{f-1}\bigbra{p-1-i(\chi_{J'}^s)}!}/P_1(\chi_{J'})\\
    &=\bbbra{\frac{\sprod_{j+1\in J^{\prime\delta},j\notin J^{\prime\delta}}(-1)\bigbra{s^{(J'-1)^{\ss}}_j\!+1}}{\sprod_{j+1\notin J^{\prime\delta},j\in J^{\prime\delta}}\bigbra{s^{(J'-1)^{\ss}}_j\!+1}}}\bbra{\sprod_{j+1\in J^{\prime\delta}}\bigbra{s^{(J'-1)^{\ss}}_j}!}^2.
\end{aligned}
\end{equation*}
Then we have (using \eqref{Constants Eq s for deltass})
\begin{equation*}
\begin{aligned}
    \sprod_{i=0}^{\ell(J)-1}\alpha\bigbra{\chi_{\delta_{\ss}^i(J)}}=\bbbra{\frac{\sprod_{i\in I(J)^{5}_j\sqcup I(J)^{6}_j}(-1)\bigbra{s^{\delta_{\ss}^{i+1}(J)}_j+1}}{\sprod_{i\in I(J)^{7}_j\sqcup I(J)^{8}_j}\bigbra{s^{\delta_{\ss}^{i+1}(J)}_j+1}}}\bbra{\sprod_{i\in I(J)^{1}_j\sqcup I(J)^{2}_j}\bigbra{s^{\delta_{\ss}^{i+1}(J)}_j}!}^2\\
    =\bbbra{\frac{\bigbra{p\!-\!1\!-\!r_j\!-\!\delta_{j\in J_{\rhobar}}}^{\abs{I(J)^{5}_j}}\bra{r_j\!+\!1}^{\abs{I(J)^{6}_j}}}{\bigbra{p\!-\!1\!-\!r_j\!-\!\delta_{j\in J_{\rhobar}}}^{\abs{I(J)^{7}_j}}\bra{r_j\!+\!1}^{\abs{I(J)^{8}_j}}}}\!\bigbra{\bra{r_j\!+\!\delta_{j\in J_{\rhobar}}}!}^{2\abs{I(J)^{1}_j}}\bigbra{\bra{p\!-\!2\!-\!r_j}!}^{2\abs{I(J)^{2}_j}}.
\end{aligned}
\end{equation*}
Similar formula holds with each $J$ replaced with $J^{\ss}$. Hence we have
\begin{equation}\label{Constants Eq Lem Prop 3-2}
    \alpha(J)=\bbra{\sprod_{i=0}^{\ell(J)-1}\alpha\bigbra{\chi_{\delta_{\ss}^i(J)}}}/\bbra{\sprod_{i=0}^{\ell(J^{\ss})-1}\alpha\bigbra{\chi_{\delta_{\ss}^i(J^{\ss})}}}=\sprod_{j=0}^{f-1}\alpha(J)_j,
\end{equation}
where (for each $j\in\cJ$)
\begin{equation}\label{Constants Eq Lem Prop 3-3}
    \alpha(J)_j\eqdef\bigbra{\bra{r_j\!+\!\delta_{j\in J_{\rhobar}}}!}^{2i(J)^{1}_j}\bigbra{\bra{p\!-\!2\!-\!r_j}!}^{2i(J)^{2}_j}\bigbra{p\!-\!1\!-\!r_j\!-\!\delta_{j\in J_{\rhobar}}}^{i(J)^{3}_j}\bra{r_j\!+\!1}^{i(J)^{4}_j}\in\FF\x.
\end{equation}
To compute each $\alpha(J)_j$ explicitly, we separate the following cases.

\hspace{\fill}

\noindent\textbf{Case 1.} $j+1\in J_{\rhobar},~j\in J_{\rhobar}$.

Let $k\geq0$ such that $j+i+1\in J_{\rhobar}$ for $0\leq i\leq k$ and $j+k+2\notin J_{\rhobar}$. By \eqref{Constants Eq j in deltass} we have
\begin{equation}\label{Constants Eq Lem Prop 3-4}
\begin{aligned}
    j+1\in\delta_{\ss}^i(J)&\iff\bra{0\leq i\leq k+1~\text{and}~j+i+1\in J}\\
    j\in\delta_{\ss}^i(J)&\iff\bra{0\leq i\leq k+2~\text{and}~j+i\in J}\\
    j+1\in\delta_{\ss}^i(J^{\ss})&\iff\bra{0\leq i\leq k~\text{and}~j+i+1\in J}\\
    j\in\delta_{\ss}^i(J^{\ss})&\iff\bra{0\leq i\leq k+1~\text{and}~j+i\in J}.
\end{aligned}
\end{equation}
In particular, for $t\in\set{1,2,5,6,7,8}$ we have $i\in I(J)^{t}_j\iff i\in I(J^{\ss})^{t}_j$ for $0\leq i\leq k-1$, hence
\begin{equation}\label{Constants Eq Lem Prop 3-5}
    \bigabs{I(J)^t_j}-\bigabs{I(J^{\ss})^t_j}=\bigabs{I^{\geq}(J)^t_j}-\bigabs{I^{\geq}(J^{\ss})^t_j}.
\end{equation}
We denote $\op{ch}^2_J\eqdef\bigbra{\delta_{j+k\in J},\delta_{j+k+1\in J},\delta_{j+k+2\in J}}\in\set{0,1}^3$. Combining \eqref{Constants Eq Lem Prop 3-3}, \eqref{Constants Eq Lem Prop 3-4}, \eqref{Constants Eq Lem Prop 3-5} and a case-by-case examination we get the following table.
\begin{figure}[H]
    \centering
\begin{tabular}{|c|c|c|c|c|}
    \hline
    $\op{ch}^2_J$&$I^{\geq}(J)^{1,2,5,6,7,8}_j$&$I^{\geq}(J^{\ss})^{1,2,5,6,7,8}_j$&$i(J)^{1,2,3,4}_j$&$\alpha(J)_j$\\
    \hline
    $(1,1,1)$&$\set{1},\emptyset,\set{1},\emptyset,\emptyset,\set{2}$&$\set{0},\emptyset,\set{0},\emptyset,\emptyset,\set{1}$&$0,0,0,0$&$1$\\
    \hline
    $(0,1,1)$&$\set{1},\emptyset,\set{1},\emptyset,\set{0},\set{2}$&$\set{0},\emptyset,\emptyset,\emptyset,\emptyset,\set{1}$&$0,0,0,0$&$1$\\
    \hline
    $(1,1,0)$&$\set{0},\emptyset,\set{0},\emptyset,\emptyset,\set{1}$&$\set{0},\emptyset,\set{0},\emptyset,\emptyset,\set{1}$&$0,0,0,0$&$1$\\
    \hline
    $(0,1,0)$&$\set{0},\emptyset,\emptyset,\emptyset,\emptyset,\set{1}$&$\set{0},\emptyset,\emptyset,\emptyset,\emptyset,\set{1}$&$0,0,0,0$&$1$\\
    \hline
    $(1,0,1)$&$\set{1},\set{0},\emptyset,\emptyset,\emptyset,\set{2}$&$\emptyset,\emptyset,\emptyset,\emptyset,\emptyset,\set{0}$&$1,1,0,0$&$1$\\
    \hline
    $(0,0,1)$&$\set{1},\set{0},\emptyset,\set{0},\emptyset,\set{2}$&$\emptyset,\emptyset,\emptyset,\emptyset,\emptyset,\emptyset$&$1,1,0,0$&$1$\\
    \hline
    $(1,0,0)$&$\emptyset,\emptyset,\emptyset,\emptyset,\emptyset,\set{0}$&$\emptyset,\emptyset,\emptyset,\emptyset,\emptyset,\set{0}$&$0,0,0,0$&$1$\\
    \hline
    $(0,0,0)$&$\emptyset,\emptyset,\emptyset,\emptyset,\emptyset,\emptyset$&$\emptyset,\emptyset,\emptyset,\emptyset,\emptyset,\emptyset$&$0,0,0,0$&$1$\\
    \hline
\end{tabular}
\end{figure}
\noindent Here we use $\bigbra{(r_j+1)!(p-2-r_j)!}^2=1$ in $\FF$. In particular, we have $\alpha(J)_j\alpha'(J)_j=1$.

\hspace{\fill}

\noindent\textbf{Case 2.} $j+1\in J_{\rhobar},~j\notin J_{\rhobar}$.

Let $k\geq0$ such that $j+i+1\in J_{\rhobar}$ for $0\leq i\leq k$ and $j+k+2\notin J_{\rhobar}$. By \eqref{Constants Eq j in deltass} we have
\begin{equation}\label{Constants Eq Lem Prop 3-6}
\begin{aligned}
    j+1\in\delta_{\ss}^i(J)&\iff\bra{0\leq i\leq k+1~\text{and}~j+i+1\in J}\\
    j\in\delta_{\ss}^i(J)&\iff\bra{i=0~\text{and}~j\in J}\\
    j+1\in\delta_{\ss}^i(J^{\ss})&\iff\bra{0\leq i\leq k~\text{and}~j+i+1\in J}\\
    j\in\delta_{\ss}^i(J^{\ss})&\iff\bra{\text{impossible}}.
\end{aligned}
\end{equation}
We denote $\op{ch}^3_J\eqdef\bigbra{\delta_{j\in J},\delta_{j+1\in J},\delta_{j+2\in J}}\in\set{0,1}^3$.

If $k=0$, then combining \eqref{Constants Eq Lem Prop 3-3} and \eqref{Constants Eq Lem Prop 3-6} we get the following table.
\begin{figure}[H]
    \centering
\begin{tabular}{|c|c|c|c|c|}
    \hline
    $\op{ch}^3_J$&$I(J)^{1,2,5,6,7,8}_j$&$I(J^{\ss})^{1,2,5,6,7,8}_j$&$i(J)^{1,2,3,4}_j$&$\alpha(J)_j$\\
    \hline
    $(1,1,1)$&$\set{1},\emptyset,\set{1},\emptyset,\set{0},\emptyset$&$\set{0},\emptyset,\set{0},\emptyset,\emptyset,\emptyset$&$0,0,-1,0$&$(p-1-r_j)\inv$\\
    \hline
    $(0,1,1)$&$\set{1},\emptyset,\set{1},\emptyset,\emptyset,\emptyset$&$\set{0},\emptyset,\set{0},\emptyset,\emptyset,\emptyset$&$0,0,0,0$&$1$\\
    \hline
    $(1,1,0)$&$\set{0},\emptyset,\emptyset,\emptyset,\emptyset,\emptyset$&$\set{0},\emptyset,\set{0},\emptyset,\emptyset,\emptyset$&$0,0,-1,0$&$(p-1-r_j)\inv$\\
    \hline
    $(0,1,0)$&$\set{0},\emptyset,\set{0},\emptyset,\emptyset,\emptyset$&$\set{0},\emptyset,\set{0},\emptyset,\emptyset,\emptyset$&$0,0,0,0$&$1$\\
    \hline
    $(1,0,1)$&$\set{1},\set{0},\set{1},\emptyset,\emptyset,\emptyset$&$\emptyset,\emptyset,\emptyset,\emptyset,\emptyset,\emptyset$&$1,1,1,0$&$(p-1-r_j)\inv$\\
    \hline
    $(0,0,1)$&$\set{1},\set{0},\set{1},\set{0},\emptyset,\emptyset$&$\emptyset,\emptyset,\emptyset,\emptyset,\emptyset,\emptyset$&$1,1,1,1$&$-1$\\
    \hline
    $(1,0,0)$&$\emptyset,\emptyset,\emptyset,\emptyset,\emptyset,\set{0}$&$\emptyset,\emptyset,\emptyset,\emptyset,\emptyset,\emptyset$&$0,0,0,-1$&$(r_j+1)\inv$\\
    \hline
    $(0,0,0)$&$\emptyset,\emptyset,\emptyset,\emptyset,\emptyset,\emptyset$&$\emptyset,\emptyset,\emptyset,\emptyset,\emptyset,\emptyset$&$0,0,0,0$&$1$\\
    \hline
\end{tabular}
\end{figure}
\noindent Here we use $\bigbra{r_j!(p-2-r_j)!}^2=(r_j+1)^{-2}$ in $\FF$.

If $k\geq1$, then for $t\in\set{1,2,5,6,7,8}$ we have $i\in I(J)^{t}_j\iff i\in I(J^{\ss})^{t}_j$ for $1\leq i\leq k-1$, hence
\begin{equation}\label{Constants Eq Lem Prop 3-7}
    \bigabs{I(J)^t_j}-\bigabs{I(J^{\ss})^t_j}=\Bigbra{\bigabs{I\cc(J)^t_j}-\bigabs{I\cc(J^{\ss})^t_j}}+\Bigbra{\bigabs{I^{\geq}(J)^t_j}-\bigabs{I^{\geq}(J^{\ss})^t_j}}.
\end{equation}
By \eqref{Constants Eq Lem Prop 3-6} and a case-by-case examination, we have
\begin{equation}\label{Constants Eq Lem Prop 3-8}
\begin{aligned}
    I^{\geq}(J)^{1,2,5,6,7,8}_j&=
\begin{cases}
    \set{1},\set{0},\set{1},\set{0},\emptyset,\emptyset&\text{if}~j+k+1\notin J,~j+k+2\in J\\
    \emptyset,\emptyset,\emptyset,\emptyset,\emptyset,\emptyset&\text{otherwise}.
\end{cases}\\
    I^{\geq}(J^{\ss})^{1,2,5,6,7,8}_j&=\emptyset,\emptyset,\emptyset,\emptyset,\emptyset,\emptyset.
\end{aligned}
\end{equation}
If $\bigbra{\delta_{j+k+1\in J},\delta_{j+k+2\in J}}\neq(0,1)$, then combining \eqref{Constants Eq Lem Prop 3-3}, \eqref{Constants Eq Lem Prop 3-7} and \eqref{Constants Eq Lem Prop 3-8} we get the following table.
\begin{figure}[H]
    \centering
\begin{tabular}{|c|c|c|c|c|}
    \hline
    $\op{ch}^3_J$&$I\cc(J)^{1,2,5,6,7,8}_j$&$I\cc(J^{\ss})^{1,2,5,6,7,8}_j$&$i(J)^{1,2,3,4}_j$&$\alpha(J)_j$\\
    \hline
    $(1,1,1)$&$\emptyset,\emptyset,\emptyset,\emptyset,\set{0},\emptyset$&$\emptyset,\emptyset,\emptyset,\emptyset,\emptyset,\emptyset$&$0,0,-1,0$&$(p-1-r_j)\inv$\\
    \hline
    $(0,1,1)$&$\emptyset,\emptyset,\emptyset,\emptyset,\emptyset,\emptyset$&$\emptyset,\emptyset,\emptyset,\emptyset,\emptyset,\emptyset$&$0,0,0,0$&$1$\\
    \hline
    $(1,1,0)$&$\set{0},\emptyset,\emptyset,\emptyset,\emptyset,\emptyset$&$\set{0},\emptyset,\set{0},\emptyset,\emptyset,\emptyset$&$0,0,-1,0$&$(p-1-r_j)\inv$\\
    \hline
    $(0,1,0)$&$\set{0},\emptyset,\set{0},\emptyset,\emptyset,\emptyset$&$\set{0},\emptyset,\set{0},\emptyset,\emptyset,\emptyset$&$0,0,0,0$&$1$\\
    \hline
    $(1,0,1)$&$\emptyset,\set{0},\emptyset,\emptyset,\emptyset,\emptyset$&$\emptyset,\set{0},\emptyset,\set{0},\emptyset,\emptyset$&$0,0,0,-1$&$(r_j+1)\inv$\\
    \hline
    $(0,0,1)$&$\emptyset,\set{0},\emptyset,\set{0},\emptyset,\emptyset$&$\emptyset,\set{0},\emptyset,\set{0},\emptyset,\emptyset$&$0,0,0,0$&$1$\\
    \hline
    $(1,0,0)$&$\emptyset,\emptyset,\emptyset,\emptyset,\emptyset,\set{0}$&$\emptyset,\emptyset,\emptyset,\emptyset,\emptyset,\emptyset$&$0,0,0,-1$&$(r_j+1)\inv$\\
    \hline
    $(0,0,0)$&$\emptyset,\emptyset,\emptyset,\emptyset,\emptyset,\emptyset$&$\emptyset,\emptyset,\emptyset,\emptyset,\emptyset,\emptyset$&$0,0,0,0$&$1$\\
    \hline
\end{tabular}
\end{figure}
\noindent If $\bigbra{\delta_{j+k+1\in J},\delta_{j+k+2\in J}}=(0,1)$, then the result of $\alpha(J)_j$ above should be multiplied by $(-1)$, which comes from \eqref{Constants Eq Lem Prop 3-8} using the equality $\bigbra{r_j!(p-2-r_j)!}^2(r_j+1)(p-1-r_j)=-1$ in $\FF$.

To conclude, in both cases (either $k=0$ or $k\geq1$) we have $\alpha(J)_j\alpha'(J)_j=-1$ if $j+k+1\notin J,~j+k+2\in J$, and $\alpha(J)_j\alpha'(J)_j=1$ otherwise.

\hspace{\fill}

\noindent\textbf{Case 3.} $j+1\notin J_{\rhobar},~j\notin J_{\rhobar}$.

In this case, by \eqref{Constants Eq j in deltass} we have $j,j+1\notin\delta_{\ss}^i(J)$ for all $i\geq1$ and $j,j+1\notin\delta_{\ss}^i(J^{\ss})$ for all $i\geq0$. We denote $\op{ch}^4_J\eqdef\bigbra{\delta_{j\in J},\delta_{j+1\in J}}\in\set{0,1}^2$. Then by \eqref{Constants Eq Lem Prop 3-3} we get the following table.

\begin{figure}[H]
    \centering
\begin{tabular}{|c|c|c|c|c|}
    \hline
    $\op{ch}^4_J$&$I(J)^{1,2,5,6,7,8}_j$&$I(J^{\ss})^{1,2,5,6,7,8}_j$&$i(J)^{1,2,3,4}_j$&$\alpha(J)_j$\\
    \hline
    $(1,1)$&$\set{0},\emptyset,\emptyset,\emptyset,\emptyset,\emptyset$&$\emptyset,\emptyset,\emptyset,\emptyset,\emptyset,\emptyset$&$1,0,0,0$&$(r_j!)^2$\\
    \hline
    $(1,0)$&$\emptyset,\emptyset,\emptyset,\emptyset,\emptyset,\set{0}$&$\emptyset,\emptyset,\emptyset,\emptyset,\emptyset,\emptyset$&$0,0,0,-1$&$(r_j+1)\inv$\\
    \hline
    $(0,1)$&$\set{0},\emptyset,\set{0},\emptyset,\emptyset,\emptyset$&$\emptyset,\emptyset,\emptyset,\emptyset,\emptyset,\emptyset$&$1,0,1,0$&$(r_j!)^2(p-1-r_j)$\\
    \hline
    $(0,0)$&$\emptyset,\emptyset,\emptyset,\emptyset,\emptyset,\emptyset$&$\emptyset,\emptyset,\emptyset,\emptyset,\emptyset,\emptyset$&$0,0,0,0$&$1$\\
    \hline
\end{tabular}
\end{figure}
\noindent In particular, in this case we have $\alpha(J)_j\alpha'(J)_j=1$.

\hspace{\fill}

\noindent\textbf{Case 4.} $j+1\notin J_{\rhobar},~j\in J_{\rhobar}$.

In this case, by \eqref{Constants Eq j in deltass} we have 
\begin{equation}\label{Constants Eq Lem Prop 3-9}
\begin{aligned}
    j+1\in\delta_{\ss}^i(J)&\iff\bra{i=0~\text{and}~j+1\in J}\\
    j\in\delta_{\ss}^i(J)&\iff\bra{i\in\set{0,1}~\text{and}~j+i\in J}\\
    j+1\in\delta_{\ss}^i(J^{\ss})&\iff\bra{\text{impossible}}\\
    j\in\delta_{\ss}^i(J^{\ss})&\iff\bra{i=0~\text{and}~j\in J}.
\end{aligned}
\end{equation}
Combining \eqref{Constants Eq Lem Prop 3-3} and \eqref{Constants Eq Lem Prop 3-9} we get the following table.
\begin{figure}[H]
    \centering
\begin{tabular}{|c|c|c|c|c|}
    \hline
    $\op{ch}^4_J$&$I(J)^{1,2,5,6,7,8}_j$&$I(J^{\ss})^{1,2,5,6,7,8}_j$&$i(J)^{1,2,3,4}_j$&$\alpha(J)_j$\\
    \hline
    $(1,1)$&$\set{0},\emptyset,\set{0},\emptyset,\emptyset,\set{1}$&$\emptyset,\emptyset,\emptyset,\emptyset,\emptyset,\set{0}$&$1,0,1,0$&$\bra{(r_j+1)!}^2(p-2-r_j)$\\
    \hline
    $(1,0)$&$\emptyset,\emptyset,\emptyset,\emptyset,\emptyset,\set{0}$&$\emptyset,\emptyset,\emptyset,\emptyset,\emptyset,\set{0}$&$0,0,0,0$&$1$\\
    \hline
    $(0,1)$&$\set{0},\emptyset,\emptyset,\emptyset,\emptyset,\set{1}$&$\emptyset,\emptyset,\emptyset,\emptyset,\emptyset,\emptyset$&$1,0,0,-1$&$\bra{(r_j+1)!}^2(r_j+1)\inv$\\
    \hline
    $(0,0)$&$\emptyset,\emptyset,\emptyset,\emptyset,\emptyset,\emptyset$&$\emptyset,\emptyset,\emptyset,\emptyset,\emptyset,\emptyset$&$0,0,0,0$&$1$\\
    \hline
\end{tabular}
\end{figure}
\noindent In particular, in this case we have $\alpha(J)_j\alpha'(J)_j=-1$ if $j\notin J,~j+1\in J$, and $\alpha(J)_j\alpha'(J)_j=1$ otherwise.

\hspace{\fill}

Combining \eqref{Constants Eq Lem Prop 3-2} and the explicit computation of $\alpha(J)_j$ in Case 1-Case 4, we deduce that 
\begin{equation*}
    \alpha(J){\sprod_{j=0}^{f-1}\alpha'(J)_j}=\sprod_{j=0}^{f-1}\bigbra{\alpha(J)_j\alpha'(J)_j}=(-1)^{2\#\set{j:j+1\notin J_{\rhobar},~j+1\in J,~j\in J_{\rhobar},~j\notin J}}=1,
\end{equation*}
which completes the proof.\qed

\section{The main result}\label{Constants Sec main result}

We combine the results of the previous sections and the results of \cite{Wang2} and \cite{Wang3} to finish the proof of Theorem \ref{Constants Thm main}.

We recall the definition of the ring $A$. We let $\fm_{N_0}$ be the maximal ideal of $\FF\ddbra{N_0}$. Then we have $\FF\ddbra{N_0}=\FF\ddbra{Y_0,\ldots,Y_{f-1}}$ and $\fm_{N_0}=(Y_0,\ldots,Y_{f-1})$. Consider the multiplicative subset $S\eqdef\sset{(Y_0\cdots Y_{f-1})^n:n\geq0}$ of $\FF\ddbra{N_0}$. Then $A\eqdef\wh{\FF\ddbra{N_0}_S}$ is the completion of the localization $\FF\ddbra{N_0}_S$ with respect to the $\fm_{N_0}$-adic filtration 
\begin{equation*}
    F_n\bbra{\FF\ddbra{N_0}_S}=\bigcup\limits_{k\geq0}\frac{1}{(Y_0\cdots Y_{f-1})^k}\fm_{N_0}^{kf-n},
\end{equation*}
where $\fm_{N_0}^m\eqdef\FF\ddbra{N_0}$ if $m\leq0$. We denote by $F_nA$ ($n\in\ZZ$) the induced filtration on $A$ and endow $A$ with the associated topology. 
There is an $\FF$-linear action of $\OK\x$ on $\FF\ddbra{N_0}$ given by multiplication on $N_0\cong\OK$, and an $\FF$-linear Frobenius $\varphi$ on $\FF\ddbra{N_0}$ given by multiplication by $p$ on $N_0\cong\OK$. They extend canonically by continuity to commuting continuous $\FF$-linear actions of $\varphi$ and $\OK\x$ on $A$. Then an \'etale $(\varphi,\OK\x)$-module over $A$ is by definition a finite free $A$-module $D$ endowed with a semi-linear Frobenius $\varphi$ and a commuting continuous semi-linear action of $\OK\x$ such that the image of $\varphi$ generates $D$ over $A$.

\hspace{\fill}

Let $\rhobar$ be as in \eqref{Constants Eq rhobar}. We refer to \cite{BHHMS3} for the definition of the \'etale $(\varphi,\OK\x)$-module $D_A^{\otimes}(\rhobar)$ over $A$. By \cite[(46)]{Wang3}, $D_A^{\otimes}(\rhobar)$ has rank $2^f$ and is equipped with an $A$-basis such that
\begin{enumerate}
    \item 
    the corresponding matrix $\Mat(\varphi)\in\GL_{2^f}(A)$ (with its rows and colomns indexed by the subsets of $\cJ$) for the $\varphi$-action is given by 
    \begin{equation}\label{Constants Eq Mat phi}
        \Mat(\varphi)_{J',J+1}=
    \begin{cases}
        \nu_{J+1,J'}\sprod_{j\notin J}\bra{Y_j/\varphi(Y_j)}^{r_j+1}&\text{if}~J'\subseteq J\\
        0&\text{if}~J'\nsubseteq J
    \end{cases}
    \end{equation}
    with (see \eqref{Constants Eq LT dj} for $\beta$ and $d_j$)
    \begin{equation*}
        \nu_{J,J'}\eqdef\beta^{|J^c|-|J|}\sprod_{j\in(J-1)\setminus J'}d_j~\text{for}~J'\subseteq J-1;
    \end{equation*}
    \item
    the corresponding matrices $\Mat(a)$ for the $\OK\x$-action satisfy $\Mat(a)\in\op{I}_{2^f}+\M_{2^f}(F_{1-p}A)$ for all $a\in\OK\x$.
\end{enumerate}
In particular, since $d_j=0$ if and only if $j\in J_{\rhobar}$, we deduce that $\nu_{J,J'}\neq0$ if and only if $(J-1)^{\ss}\subseteq J'\subseteq J-1$. We then extend the definition of $\nu_{J,J'}$ to all $J,J'\subseteq\cJ$ such that $(J-1)^{\ss}=(J')^{\ss}$ by the formula
\begin{equation}\label{Constants Eq nuJJ'}
    \nu_{J,J'}\eqdef\beta^{|J^c|-|J|}\bbra{\sprod_{j\in(J-1)^{\nss}}d_j}/\bbra{\sprod_{j\in(J')^{\nss}}d_j}.
\end{equation}
Then for $J_1,J_2,J_3,J_4\subseteq\cJ$ such that $(J_1-1)^{\ss}=(J_2-1)^{\ss}=J_3^{\ss}=J_4^{\ss}$ we have $\nu_{J_1,J_3}/\nu_{J_1,J_4}=\nu_{J_2,J_3}/\nu_{J_2,J_4}$. We define $\nu_{*,J}/\nu_{*,J'}$ for $J^{\ss}=(J')^{\ss}$ in a similar way as $\mu_{*,J}/\mu_{*,J'}$.

If moreover $J_{\rhobar}\neq\cJ$, then for $J\subseteq\cJ$ we define
\begin{equation}\label{Constants Eq nuJ}
    \nu(J)\eqdef\frac{\nu_{*,J}}{\nu_{*,J^{\ss}}}\bbbra{\frac{\sprod_{i=0}^{\ell(J)-1}\nu_{\delta_{\ss}^i(J),\delta_{\ss}^{i+1}(J)}}{\sprod_{i=0}^{\ell(J^{\ss})-1}\nu_{\delta_{\ss}^i(J^{\ss}),\delta_{\ss}^{i+1}(J^{\ss})}}}.
\end{equation}
By definition, we have $\nu(J)=\beta^{B(J)}d(J)$ for $B(J)\in\ZZ$ and $d(J)\in\FF\x$ as in \eqref{Constants Eq ABd}.

\hspace{\fill}

Let $\pi$ be as in \eqref{Constants Eq local factor}. We refer to \cite{BHHMS2} for the definition of the \'etale $(\varphi,\OK\x)$-module $D_A(\pi)$ over $A$. By \cite[Prop.~C.3(i),(iii)]{Wang2} and \cite[Cor.~C.4]{Wang2}, the twisted dual \'etale $(\varphi,\OK\x)$-module $\Hom_A(D_A(\pi),A)(1)$ has rank $2^f$ and is equipped with an $A$-basis such that
\begin{enumerate}
    \item 
    the corresponding matrix $\Mat(\varphi)'\in\GL_{2^f}(A)$ for the $\varphi$-action is given by 
    \begin{equation}\label{Constants Eq Mat phi'}
        \Mat(\varphi)'_{J',J+1}=
        \begin{cases}
            \gamma_{J+1,J'}\sprod_{j\notin J}\bra{Y_j/\varphi(Y_j)}^{r_j+1}&\text{if}~J^{\ss}\subseteq J'\subseteq J\\
            0&\text{otherwise},
        \end{cases}
    \end{equation}
    where for $J,J'\subseteq\cJ$ such that $(J-1)^{\ss}=(J')^{\ss}$ we define
    \begin{equation}\label{Constants Eq gammaJJ'}
        \gamma_{J,J'}\eqdef(-1)^{f-1+\delta_{(J')^{\nss}=\cJ}+\abs{\bra{J'\cap(J'-1)}^{\nss}}}\mu_{J,J'};
    \end{equation}
    \item 
    the corresponding matrices $\Mat(a)'$ for the $\OK\x$-action satisfy $\Mat(a)'_{J,J}\in1+F_{1-p}A$ for all $a\in\OK\x$ and $J\subseteq\cJ$, which uniquely determines $\Mat(a)'$.
\end{enumerate}
Note that when $J_{\rhobar}\neq\cJ$, $J\nsubseteq J_{\rhobar}$ and $J\neq J^*$ (see Lemma \ref{Constants Lem J*} for $J^*$) we have (see \eqref{Constants Eq gammaJ} for $\gamma(J)$)
\begin{equation}\label{Constants Eq gammaJ new}
    \gamma(J)=\frac{\gamma_{*,J}}{\gamma_{*,J^{\ss}}}\bbbra{\frac{\sprod_{i=0}^{\ell(J)-1}\gamma_{\delta_{\ss}^i(J),\delta_{\ss}^{i+1}(J)}}{\sprod_{i=0}^{\ell(J^{\ss})-1}\gamma_{\delta_{\ss}^i(J^{\ss}),\delta_{\ss}^{i+1}(J^{\ss})}}},
\end{equation}
where $\gamma_{*,J}/\gamma_{*,J^{\ss}}$ is defined in a similar way as $\mu_{*,J}/\mu_{*,J^{\ss}}$. 

\begin{lemma}\label{Constants Lem conjugation}
    Suppose that $J_{\rhobar}\neq\cJ$. Let $B\in\M_{2^f}(\FF)$ with its rows and columns indexed by the subsets of $\cJ$ such that
    \begin{enumerate}
        \item 
        $B_{J,J'}\in\FF\x$ if and only if $(J-1)^{\ss}=(J')^{\ss}$;
        \item 
        $B_{J_1,J_3}/B_{J_1,J_4}=B_{J_2,J_3}/B_{J_2,J_4}$ for all $J_1,J_2,J_3,J_4\subseteq\cJ$ such that $(J_1-1)^{\ss}=(J_2-1)^{\ss}=J_3^{\ss}=J_4^{\ss}$.
    \end{enumerate}
    We define  $B_{*,J}/B_{*,J^{\ss}}$ in a similar way as $\mu_{*,J}/\mu_{*,J^{\ss}}$. Then up to conjugation by diagonal matrices, $B$ is uniquely determined by the quantities 
    \begin{equation*}
        \left\{
    \begin{aligned}
        &B_{\emptyset,\emptyset}\\
        &B(J^*)\eqdef B_{(J^*-1)^{\ss},J^*}B_{J^*,(J^*-1)^{\ss}}\\
        &B(J)\eqdef\frac{B_{*,J}}{B_{*,J^{\ss}}}\bbbra{\frac{\sprod_{i=0}^{\ell(J)-1}B_{\delta_{\ss}^i(J),\delta_{\ss}^{i+1}(J)}}{\sprod_{i=0}^{\ell(J^{\ss})-1}B_{\delta_{\ss}^i(J^{\ss}),\delta_{\ss}^{i+1}(J^{\ss})}}}~\text{for}~J\nsubseteq J_{\rhobar}~\text{and}~J\neq J^*.
    \end{aligned}\right.
    \end{equation*}
\end{lemma}

\begin{proof}
    First, it is easy to check that conjugation by a diagonal matrix does not change these quantities. 
    
    Next, given such a matrix $B$, after conjugation we may assume that $B_{J,\delta_{\ss}(J)}=1$ for all $J\neq\emptyset$. Indeed, if we let $Q\in\GL_{2^f}(\FF)$ be the diagonal matrix with $Q_{J,J}=\prod_{i=0}^{\ell(J)-1}B_{\delta_{\ss}^i(J),\delta_{\ss}^{i+1}(J)}$, then $Q\inv BQ$ satisfies this property. In particular, this determines the entries $B_{J,J'}$ with $J'\subseteq J_{\rhobar}$.
    
    Then for $J,J'$ such that $(J-1)^{\ss}=(J')^{\ss}$ and $J'\nsubseteq J_{\rhobar}$, the entry $B_{J,J'}$ is determined by 
    \begin{equation*}
        B_{J,J'}=
    \begin{cases}
        B(J')B_{J,(J')^{\ss}}&\text{if}~J'\neq J^*\\
        B(J^*)B_{J,(J-1)^{\ss}}/B_{(J^*-1)^{\ss},(J-1)^{\ss}}&\text{if}~J'=J^*.
    \end{cases}
    \end{equation*}
    This completes the proof.
\end{proof}

Suppose that the matrices $\bra{\gamma_{J,J'}},\bra{\nu_{J,J'}}\in\M_{2^f}(\FF)$ are conjugated by the diagonal matrix $Q$, then the matrices $\bigbra{\gamma_{J,J'}\delta_{(J-1)^{\ss}\subseteq J'\subseteq J-1}}$ and $\bigbra{\nu_{J,J'}\delta_{(J-1)^{\ss}\subseteq J'\subseteq J-1}}$ are also conjugated by $Q$.

\begin{proof}[Proof of the main result]
    The case $J_{\rhobar}=\cJ$ is proved by \cite[Thm.~3.1.3]{BHHMS3}. The case $J_{\rhobar}=\emptyset$ is proved by \cite[Thm.~1.1]{Wang3}. In the rest of the proof we assume that $J_{\rhobar}\notin\set{\emptyset,\cJ}$.
    
    As in the proof of \cite[Thm.~1.1]{Wang3}, it suffices to show that $\Hom_A(D_A(\pi),A)(1)\cong D_A^{\otimes}(\rhobar)$. By \cite[Prop.~C.3(iii)]{Wang2} and \cite[Cor.~C.4]{Wang2}, it suffices to compare the matrices $\Mat(\varphi)$ (see \eqref{Constants Eq Mat phi}) and $\Mat(\varphi)'$ (see \eqref{Constants Eq Mat phi'}). Then by \eqref{Constants Eq gammaJ new}, Lemma \ref{Constants Lem conjugation} and the sentence that follows, it suffices to show that
    \begin{enumerate}
        \item 
        $\gamma_{\emptyset,\emptyset}=\nu_{\emptyset,\emptyset}$;
        \item 
        $\gamma_{(J^*-1)^{\ss},J^*}\gamma_{J^*,(J^*-1)^{\ss}}=\nu_{(J^*-1)^{\ss},J^*}\nu_{J^*,(J^*-1)^{\ss}}$;
        \item 
        $\gamma(J)=\nu(J)$ for $J\nsubseteq J_{\rhobar}$ and $J\neq J^*$ (see \eqref{Constants Eq gammaJ} for $\gamma(J)$ and \eqref{Constants Eq nuJ} for $\nu(J)$.
    \end{enumerate}
    Indeed, by Lemma \ref{Constants Lem muempty} and \eqref{Constants Eq gammaJJ'} we have $\gamma_{\emptyset,\emptyset}=\xi$ (see \eqref{Constants Eq rhobar} for $\xi$), which equals $\nu_{\emptyset,\emptyset}$ by \eqref{Constants Eq nuJJ'}. By Corollary \ref{Constants Cor muJ*} and \eqref{Constants Eq gammaJJ'} we have $\gamma_{(J^*-1)^{\ss},J^*}\gamma_{J^*,(J^*-1)^{\ss}}=1$, which equals $\nu_{(J^*-1)^{\ss},J^*}\nu_{J^*,(J^*-1)^{\ss}}$ by \eqref{Constants Eq nuJJ'}. Finally, for $J\subseteq\cJ$ such that $J\nsubseteq J_{\rhobar}$ and $J\neq J^*$, by \eqref{Constants gamma Up c'}, Proposition \ref{Constants Prop UpJ}, Proposition \ref{Constants Prop c'J simple} and Proposition \ref{Constants Prop c'J new} we have $\gamma(J)=\beta^{B(J)}d(J)$ (see \eqref{Constants Eq ABd} for $B(J)$ and $d(J)$), which equals $\nu(J)$ by definition. This completes the proof.
\end{proof}

\bibliography{1}
\bibliographystyle{alpha}

\end{document}